\documentclass[12pt,a4paper,reqno]{amsart} 
\usepackage{footnote}
\usepackage{amssymb}
\usepackage{latexsym}
\usepackage{amsmath}
\usepackage{mathrsfs}
\usepackage{mathabx}
\usepackage[X2,T1]{fontenc}
\usepackage[applemac]{inputenc}
\usepackage{cite}
\usepackage{amscd}
\usepackage{color}

\usepackage{calc}                   

\newcommand{\scal}[2]{\langle #1,#2\rangle}
\newcommand{\rr}[1]{\mathbf R^{#1}}
\newcommand{\zz}[1]{\mathbf Z^{#1}}

\newcommand{\cc}[1]{\mathbf C^{#1}}
\newcommand{\nm}[2]{\Vert #1\Vert _{#2}}

\newcommand{\op}{\operatorname{Op}}

\newcommand{\sets}[2]{\{ \, #1\, ;\, #2\, \} }
\newcommand{\Sets}[2]{\left \{ \, #1\, ;\, #2\, \right \} }

\newcommand{\cdo}{\, \cdot \, }

\newcommand{\loc}{\operatorname{loc}}
\newcommand{\wpr}{{\text{\footnotesize $\#$}}}

\newcommand{\eabs}[1]{\langle #1\rangle}     

\newcommand{\tp}{\operatorname{Tp}}
\newcommand{\vrum}{\vspace{0.1cm}}

\newcommand{\Sh}{\operatorname{Sh}}

\newcommand{\indlim}{\operatorname{ind \, lim\, }}
\renewcommand{\projlim}{\operatorname{proj \, lim\, }}
\newcommand{\aw}{\operatorname{aw}}

\newcommand{\repart}{\operatorname{Re}}

\newcommand{\GL}{\mathbf{M}}

\newcommand{\nn}[1]{{\mathbf N}^{#1}}
\newcommand{\maclA}{\mathcal A}
\newcommand{\maclB}{\mathcal B}
\newcommand{\maclC}{\mathcal C}
\newcommand{\maclH}{\mathcal H}

\newcommand{\maclL}{\mathcal L}

\newcommand{\maclS}{\mathcal S}
\newcommand{\maclT}{\mathcal T}

\newcommand{\mascF}{\mathscr F}
\newcommand{\mascP}{\mathscr P}
\newcommand{\mascS}{\mathscr S}

\newcommand{\fka}{\mathfrak a}

\numberwithin{equation}{section}          

\newtheorem{thm}{Theorem}
\numberwithin{thm}{section}

\newcommand{\rubrik}{}
\newtheorem{prop}[thm]{Proposition}
\newtheorem{cor}[thm]{Corollary}
\newtheorem{lemma}[thm]{Lemma}

\theoremstyle{definition}

\newtheorem{defn}[thm]{Definition}
\newtheorem{example}[thm]{Example}

\theoremstyle{remark}
\newtheorem{rem}[thm]{Remark}

\title{Wick and anti-Wick characterizations of linear operators on
spaces of power series expansions}

\author{Joachim Toft}

\address{Department of Mathematics,
Linn{\ae}us University, V{\"a}xj{\"o}, Sweden}

\email{joachim.toft@lnu.se}



\begin{document}

\begin{abstract}
We study links between
Wick, anti-Wick and analytic kernel operators on the Bargmann
transform side. We consider classes of kernels,
whose corresponding operators agree with
the sets of linear and continuous operators on 
spaces of power series expansions, which are Bargmann
images of Pilipovi{\'c} spaces.
We show that in several situations, the sets of Wick
and kernel operators with symbols and kernels
in such classes agree. We also find suitable subclasses
to these kernel classes, whose
corresponding sets of Wick and anti-Wick operators agree.
We also show ring, module and composition properties for
such classes.
\end{abstract}

\maketitle

\par

\section{Introduction}\label{sec0}

\par

A convenient way to represent linear operators acting on functions
defined on
the configuration space $\rr d$, is to apply the Bargmann
transform. In this process, these operators are transformed
into (analytic) integral operators $T_K$ with
kernels $K$, or as analytic pseudo-differential operators
$\op _{\mathfrak V}(a)$ with symbols $a$, given by
\begin{align}
T_KF(z) &= \pi ^{-d}
\int _{\cc d} K(z,w)F(w)e^{-|w|^2}\, d\lambda (w)
\label{Eq:AnalKernelopIntro}
\intertext{and}
\op _{\mathfrak V}(a)F (z) &= \pi ^{-d}
\int _{\cc d} a(z,w)F(w)e^{(z-w,w)}\, d\lambda (w),
\label{Eq:AnalPseudoIntro}
\end{align}
respectively, when $F$ is a suitable analytic function on $\cc d$,
where $d\lambda$ is the Lebesgue measure
and $(\cdo ,\cdo )$ is the scalar product on $\cc d$.
Here $K(z,w)$ and $a(z,w)$ are analytic
functions or, more generally, power series expansions with respect to
$(z,\overline w)\in \cc d\times \cc d\simeq \cc {2d}$. (See \cite{Ho1}
or Section \ref{sec1} for notations.) The operator
$\op _{\mathfrak V}(a)$ is often called the Wick or Berezin operator with
(Wick) symbol $a$.

\par

An important subclass of Wick operators are the anti-Wick operators,
$\op _{\mathfrak V}^{\aw}(a)$ with suitable symbols $a(z,w)$ being analytic with
respect to $(z,\overline w)\in \cc {2d}$, and given by
\begin{align}
\op _{\mathfrak V}^{\aw}(a)F (z) &= \pi ^{-d}
\int _{\cc d} a(w,w)F(w)e^{(z-w,w)}\, d\lambda (w),
\label{Eq:AntiWickIntro}
\end{align}
where $F$ is a suitable analytic function on $\cc d$. An important
feature for anti-Wick operators is that $\op _{\mathfrak V}^{\aw}(a)$ is
positive semi-definite when $a(z,z)\ge 0$ for every $z\in \cc d$, which
is essential when performing certain types of energy estimates in quantum
physics. Any similar transitions on positivity from symbols to Wick operators are
in general not true. On the other hand for suitable $a$ we have the expansion
formula 
\begin{equation}\label{Eq:WickToAntiWickIntro}
\op _{\mathfrak V}(a)
=
\sum _{|\alpha |<N} \frac {(-1)^{|\alpha |}}{\alpha !}
\op _{\mathfrak V}^{\aw}(b_\alpha ) + \op _{\mathfrak V}(c_N)
\quad \text{where} \quad
b_\alpha 
=
\partial _z^\alpha \overline \partial _{w}^\alpha a,
\end{equation}
of Wick operators into a superposition of anti-Wick operators with error term
$\op _{\mathfrak V}(c_N)$, deduced in \cite{TeToWa}. Here the first term
$\op _{\mathfrak V}^{\aw}(b_0)=\op _{\mathfrak V}^{\aw}(a)$
on the right-hand side, possess the convenient positivity transition from
symbols to anti-Wick operators, while $c_N$ is a linear combination of 
expressions of the form
$$
\int _0^1 (1-t)^{|\alpha |-1}\partial _z^\alpha \overline \partial _w^\alpha
a(w+t(z-w),w)\, dt,
\qquad |\alpha |=N.
$$
In several situations, this term is dominating, while 
the reminder term $\op _{\mathfrak V}(c_N)$ is small compared to the other operators
in the expansion. (See e.{\,}g. \cite{Shubin1,TeToWa}.)

\par
 
Wick and anti-Wick operators
act between suitable spaces of analytic functions, which are
the Bargmann transforms of spaces of functions or distributions on the
configuration space $\rr d$. Some ideas on this approach were
performed by G. C. Wick in \cite{Wick}. Later on, F. Berezin improved
and extended the theory \cite{Berezin71,Berezin72}, where more
general classes of operators were considered, compared to \cite{Wick}.
(See also \cite{Fo,Shubin1} for some other earlier results on Wick and anti-Wick
operators.)
An important part of the investigations in \cite{Berezin71,Berezin72,Wick}
concerns the links between Wick and anti-Wick operators. Especially in
\cite{Berezin71,Berezin72} convenient formulae were established
which explain such links in the case when the Wick and anti-Wick
symbols are polynomials. These formulae are in some sense related to
\eqref{Eq:WickToAntiWickIntro}.

\par

Some recent continuity properties for (analytic) integral, Wick and anti-Wick
operators are obtained in \cite{Teofanov2,TeToWa}. In \cite{Teofanov2}, such
continuity properties are obtained in the framework of the spaces of power
series expansions,
\begin{equation}\label{Eq:IntroAspaces}
\maclA _{0,s}(\cc d),\quad  \maclA _s(\cc d),
\quad \text{with duals}\quad
\maclA _s'(\cc d), \quad
\maclA _{0,s}'(\cc d),
\end{equation}
respectively.
These spaces are the Bargmann images of the so-called Pilipovi{\'c} spaces
and their distribution (dual) spaces
\begin{equation}\label{Eq:IntroHspaces}
\maclH _{0,s}(\rr d),\quad  \maclH _s(\rr d),\quad \maclH _s'(\rr d)
\quad \text{and}\quad
\maclH _{0,s}'(\rr d),
\end{equation}
a family of function and (ultra-)distribution spaces, which contains
the Fourier invariant Gelfand-Shilov spaces and their ultra-distribution spaces,
and which are thoroughly investigated in \cite{Toft18}.
Some early ideas to
some of the spaces $\maclH _s(\rr d)$, $\maclH _{0,s}(\rr d)$ and their duals
is given by S. Pilipovi{\'c} in \cite{Pil}.
For example in \cite{Pil} it is proved that $\maclH _{s}(\rr d)$ and
$\maclH _{0,t}(\rr d)$ agree with the Fourier invariant Gelfand-Shilov spaces
$\maclS _{s}(\rr d)$ and $\Sigma _t(\rr d)$, respectively, when $s\ge \frac 12$
and $t>\frac 12$. (See also Proposition \ref{Prop:PilSpacChar1} in Section
\ref{sec1} for a more complete view.)

\par

In \cite{Teofanov2} it is remarked that the usual distribution kernel results for
linear operators from $\maclH _s(\rr d)$ or $\maclH _{0,s}(\rr d)$
to their duals,
give one-to-one correspondence of linear and continuous operators
between $\maclA _s(\cc d)$ and $\maclA _s'(\cc d)$ and analytic
integral operators with kernels in $\wideparen \maclA _s(\cc {2d})$ or in
$\wideparen \maclA _s'(\cc {2d})$. Here $K(z,w)\in \wideparen
\maclA _s(\cc {2d})$, if and only if $K(z,\overline w)\in \maclA _s(\cc {2d})$,
and similarly for other spaces. It is then proved that for suitable $s$,
the set of Wick operators with symbols in $\wideparen \maclA _s'(\cc {2d})$
agrees with the set of integral operators with kernels in
$\wideparen \maclA _s'(\cc {2d})$. Consequently, for such $s$, there is
a one-to-one correspondence of linear and continuous operators
between $\maclA _s(\cc d)$ and $\maclA _s'(\cc d)$ and Wick operators
with symbols in $\wideparen \maclA _s'(\cc {2d})$. Similar facts hold true
with $\maclA _{0,s}$ in place of $\maclA _s$ at each occurrence. (See
Theorems 2.7 and 2.8 in \cite{Teofanov2}.)

\medspace

In the paper we find characterizations of the kernels to (analytic)
integral operators and to symbols of Wick and anti-Wick
operators in order for these operators should be
continuous on the spaces in \eqref{Eq:IntroAspaces}.
This leads to introduction of certain classes of power series expansions
\begin{alignat}{5}
&\wideparen \maclB _{0,s}(\cc {2d}), &
\quad
&\wideparen \maclB _s(\cc {2d}), &
\quad
&\wideparen \maclB _s^\star (\cc {2d}) &
\quad &\text{and} & \quad
&\wideparen \maclB _{0,s}^\star (\cc {2d}),
\label{Eq:IntroBspaces}
\intertext{and slightly smaller subclasses}
&\wideparen \maclC _{0,s}(\cc {2d}), &
\quad
&\wideparen \maclC _s(\cc {2d}), &
\quad
&\wideparen \maclC _s^\star (\cc {2d}) &
\quad &\text{and} &\quad
&\wideparen \maclC _{0,s}^\star (\cc {2d}),
\label{Eq:IntroCspaces}
\end{alignat}
which are defined in similar and convenient ways as 
$\wideparen \maclA _{0,s}(\cc {2d})$,
$\wideparen \maclA _s'(\cc {2d})$ and their
duals (see Definition \ref{Def:elltauSpaces2}).
In Section \ref{sec2} we prove that $K\mapsto T_K$ is bijective from
the spaces in \eqref{Eq:IntroBspaces} into the sets of continuous mappings
on the respective spaces in \eqref{Eq:IntroAspaces}. In particular,
$$
\sets {T_K}{K\in \wideparen \maclB _s(\cc {2d})} = \maclL (\maclA _s(\cc d))
$$
(see Propositions \ref{Prop:KernelsDifficultDirectionNew}
and \ref{Prop:KernelsEasyDirectionNew}).
For $s<\frac 12$ we extend the last equality in Section \ref{sec4} into
\begin{equation}\label{Eq:IntroReprIdent}
\begin{aligned}
\op _{\mathfrak V}^{\aw}(\maclC _s(\cc {2d}))
=
\op _{\mathfrak V}(\maclC _s(\cc {2d}))
&\subseteq
\op _{\mathfrak V}(\maclB _s(\cc {2d}))
\\[1ex]
&=
\sets {T_K}{K\in \wideparen \maclB _s(\cc {2d})}
=
\maclL (\maclA _s(\cc d)),
\end{aligned}
\end{equation}
and similarly for the other spaces in \eqref{Eq:IntroAspaces},
\eqref{Eq:IntroBspaces} and \eqref{Eq:IntroCspaces}
(see
Theorems \ref{Thm:OpReprIdentB}
and \ref{Thm:OpReprIdentC}).

\par

For certain $s$, the spaces in \eqref{Eq:IntroBspaces} and 
\eqref{Eq:IntroCspaces}
are subspaces of $\wideparen A(\cc {2d})$ and
have simple characterizations in terms of
convenient estimates on involved symbols. For example,
if $s=\flat _\sigma$ with $\sigma >1$, then
\begin{align*}
\maclA _s(\cc d) &= \Sets {F\in A(\cc d)}
{\exists \, r_0>0,\ |F(z)|\lesssim e^{r_0|z|^{\frac {2\sigma}{\sigma +1}}}}
\intertext{and}
\wideparen \maclB _s(\cc {2d}) &= \Sets {a\in \wideparen A(\cc {2d})}
{\forall r>0,\ \exists \, r_0>0,\ |a(z,w)|\lesssim
e^{r_0|z|^{\frac {2\sigma}{\sigma +1}}+r|w|^{\frac {2\sigma}{\sigma -1}}}}.
\end{align*}
(Cf. \cite{AbFeGaToUs,Toft18}.)

\medspace

In several situations it is straight-forward to carry over
the characterizations above to characterize linear and continuous
operators acting on the spaces in \eqref{Eq:IntroHspaces},
by applying the Bargmann transform and its inverse in suitable ways. Here
the anti-Wick operator $\op _{\mathfrak V}^{\aw}(a)$ corresponds to the
Toeplitz operator $\tp _{\mathfrak V}(a)$, in the sense that
$$
\op _{\mathfrak V}^{\aw}(a) = \mathfrak V _d \circ \tp _{\mathfrak V}(a)
\circ \mathfrak V _d^{-1},
$$
when $\mathfrak V_d$ is the Bargmann transform. Here $\tp _{\mathfrak V}(a)$
is defined by the formula
$$
(\tp _{\mathfrak V}(a)f,g)_{L^2(\rr d)} = (\fka \cdot V_\phi f,V_\phi g)_{L^2(\rr {2d})},
\quad
\fka (\sqrt 2\, x,-\sqrt 2\, \xi )=a(z,z),\ z=x+i\xi ,
$$
where $V_\phi f$ is the short-time Fourier transform of $f$ with respect to the
window function
$$
\phi (x)=\pi ^{-\frac d4}e^{-\frac 12|x|^2}
$$
(see Subsection \ref{subsec1.7}). Especially it follows from 
\eqref{Eq:IntroReprIdent} that
\begin{equation}\label{Eq:IntroReprIdentPil}
\tp _{\mathfrak V}(\maclC _s(\cc {2d})),
\subseteq
\maclL (\maclH _s(\rr d)),\quad s<\frac 12,
\end{equation}
(see Theorem \ref{Thm:ToeplExt}).
Here we remark that in \eqref{Eq:IntroReprIdentPil},
the former class of operators is, in some sense close to the latter
class. In particular several linear and continuous operators
on $\maclH _s(\rr d)$ can be described as Toeplitz operators.

\par

If instead $s\ge \frac 12$, then the closed links \eqref{Eq:IntroReprIdent}
and \eqref{Eq:IntroReprIdentPil} between Wick, anti-Wick and
kernel operators are violated.
For example, it is shown
in Subsection \ref{Subsec4.3} that estimates, which seems rather
natural, can not be performed when finding an anti-Wick operator
to a Wick operator (see Remark \ref{Rem:DifferenceEstWickAntiWick}).

\medspace

The analysis behind \eqref{Eq:IntroReprIdent} and the other related results
are based on some continuity properties for linear and bilinear
binomial type operators of
independent interests, which acts on coefficients of power
series expansions, given in \cite{Teofanov2}
and in Section \ref{sec3}. The linear binomial operators are given by
\begin{alignat*}{1}
(\maclT _{0,t} c) (\alpha ,\beta) &= \sum_{\gamma \leq \alpha ,\beta}
\left ( { \alpha  \choose \gamma } { \beta  \choose \gamma } \right )^{1/2}
t^{|\gamma|}
c(\alpha - \gamma, \beta - \gamma) ,
\intertext{and their formal adjoints}
(\maclT _{0,t} ^*c) (\alpha ,\beta) &= \sum_{\gamma \in \nn d}
\left ( { {\alpha +\gamma}  \choose \gamma } { {\beta +\gamma}  \choose \gamma }
\right )^{1/2}
 t^{|\gamma|}
 c(\alpha + \gamma, \beta + \gamma) ,
\end{alignat*}
when $c$ is a suitable sequence on $\nn {2d}$ and $t\in \mathbf C$ is fixed. Let
$T_{\maclA}$ be the linear and bijective map which takes the sequence
$\{c(\alpha ,\beta )\} _{\alpha ,\beta \in \nn d}$ into the power series expansion
$$
\sum _{\alpha ,\beta \in \nn d}c(\alpha ,\beta )e_\alpha (z)e_\beta (\overline w),
\qquad e_\alpha (z)= \frac {z^\alpha}{\sqrt {\alpha !}}\, .
$$
Also let
$$
\ell _{\maclA ,s}'(\nn {2d})
\quad \text{and}\quad
\ell _{\maclA ,0,s}'(\nn {2d})
$$
be the counter images of $\wideparen \maclA _s'(\cc {2d})$ and 
$\wideparen \maclA _{0,s}'(\cc {2d})$, respectively, of $T_\maclA$, and let
\begin{alignat}{5}
&\ell _{\maclB ,0,s}(\nn {2d}), &
\quad
&\ell _{\maclB ,s}(\nn {2d}), &
\quad
&\ell _{\maclB ,s}^\star (\nn {2d}) &
\quad &\text{and} & \quad
&\ell _{\maclB ,0,s}^\star (\nn {2d})
\label{Eq:DiscrBSpaceIntro}
\intertext{and}
&\ell _{\maclC ,0,s}(\nn {2d}), &
\quad
&\ell _{\maclC ,s}(\nn {2d}), &
\quad
&\ell _{\maclC ,s}^\star (\nn {2d}) &
\quad &\text{and} & \quad
&\ell _{\maclC ,0,s}^\star (\nn {2d})
\label{Eq:DiscrCSpaceIntro}
\end{alignat}
be the counter images of the respective spaces in \eqref{Eq:IntroBspaces}
and \eqref{Eq:IntroCspaces}
under $T_\maclA$. Then it is proved in \cite{Teofanov2} that $\maclT _{0,t}$ is homeomorphic
on $\ell _{\maclA ,s}'(\nn {2d})$ when $s<\frac12$ and on
$\ell _{\maclA ,0,s}'(\nn {2d})$ when $s\le \frac12$.

\par

In Sections \ref{sec3} and \ref{sec4} we complete
these mapping properties and prove that for similar $s$ one has that
$\maclT _{0,t}$ is homeomorphic on each of the spaces in
\eqref{Eq:DiscrBSpaceIntro}, and that $\maclT _{0,t}^*$ is
homeomorphic on each of the spaces in
\eqref{Eq:DiscrCSpaceIntro}. By letting
\begin{equation}\label{TtOperatorsIntro}
\maclT _{t} = T_{\maclA} \circ \maclT _{0,t}\circ T_{\maclA}^{-1}
\quad \text{and}\quad
\maclT _{t}^* = T_{\maclA} \circ \maclT _{0,t}^*\circ T_{\maclA}^{-1},
\end{equation}
the relations of the form \eqref{Eq:IntroReprIdent}
then follows from
$$
T_K=\op _{\mathfrak V}(a),
\quad
\op _{\mathfrak V}(\maclT _1^*a) = \op _{\mathfrak V}^{\aw}(a)
\quad \text{and}\quad
\op _{\mathfrak V}^{\aw}(\maclT _{-1}^*a) = \op _{\mathfrak V}(a),
$$
when $K=\maclT _{1}a$, or equivalently, when $a=\maclT _{-1}K$.
Here we observe that the inverses of the operators in
\eqref{TtOperatorsIntro} are given by
$$
\maclT _{t}^{-1} =  \maclT _{-t}
\quad \text{and}\quad (\maclT _{t}^*)^{-1} =  \maclT _{-t}^*.
$$
(See also \eqref{Eq:TtDef} and \eqref{Eq:Tt*Def} in
Section \ref{sec3} for exact formulae for $\maclT _{t}$
and $\maclT _{t}^*$.)

\par

The paper is organized as follows. In Section \ref{sec1} we
set the stage by providing necessary background notions and fixing the notation.
It contains useful properties for weight functions, Pilipovi{\'c} spaces, the
Bargmann
transform, Toeplitz operators, Wick and anti-Wick operators. Here we also
introduce the spaces \eqref{Eq:IntroBspaces} and \eqref{Eq:IntroCspaces}, and
explain some of their basic properties.

\par

In Section \ref{sec2} we deduce kernel results. Especially we show that
the spaces of linear and continuous operators on the spaces in
\eqref{Eq:IntroAspaces} agree with sets of operators with kernels
belonging to the spaces in \eqref{Eq:IntroBspaces}.

\par

In Section \ref{sec3} we discuss continuity and bijectivity
properties of the binomial operators $\maclT _{0,t}$ and $\maclT _{0,t}^*$,
especially on the spaces in \eqref{Eq:DiscrBSpaceIntro}
and \eqref{Eq:DiscrCSpaceIntro}. We link these
binomial operators to transitions between kernel, Wick 
and anti-Wick operators. Here we also introduce
certain bilinear binomial operators which are linked to
multiplications and compositions of Wick symbols, and prove
continuity properties for such operators.

\par

In Section \ref{sec4} we apply the results and analysis from Section 
\ref{sec3} to find continuity and identifications between
kernel, Wick and anti-Wick operators, as well as continuity
properties for compositions of Wick and anti-Wick symbols. For example,
here we show relationships like \eqref{Eq:IntroReprIdent}.


\par

In Section \ref{sec5} we present some consequences of our investigations
concerning linear operators which acts on functions and distributions on
$\rr d$. For example we show that for suitable $s$,
%
the sets of Toeplitz operators with symbols in
\eqref{Eq:IntroCspaces} is in some sense close to the set of
linear and continuous operators on
$\maclA _s(\cc d)$, $\maclA _{0,s}(\cc d)$ and their duals.

\par

In Section \ref{sec6} we deduce mapping properties for
$\maclA _s(\cc d)$,
$\maclA _{0,s}(\cc d)$ and their duals, as well as
the spaces in \eqref{Eq:IntroBspaces} and \eqref{Eq:IntroCspaces}
under linear pullbacks and trace mappings. These results can
be used to get alternative proofs of the multiplication properties
in Section \ref{sec4}.

\par

Finally some background analyses are presented in Appendices
\ref{App:A} and \ref{App:B}. In Appendix \ref{App:A} we present some
identifications of the spaces in \eqref{Eq:IntroAspaces}, \eqref{Eq:IntroBspaces}
and some other spaces in terms of spaces of analytic functions. In
Appendix \ref{App:B} we deduce some basic formulae in the transition between
Wick and anti-Wick symbols.

\par

\section*{Acknowledgement}

\par

The author was supported by Vetenskapsr{\aa}det (Swedish 
Science Council) within the project 2019-04890.

\par

\section{Preliminaries}\label{sec1}

\par

In this section we recall some facts on involved function and distribution
spaces as well as on pseudo-differential operators. In
Subsections \ref{subsec1.2} and \ref{subsec1.3} we give definitions and
review some basic
properties for Gelfand-Shilov spaces, Pilipovi{\'c} spaces and the
spaces in \eqref{Eq:IntroBspaces}. Thereafter we
discuss in Subsection \ref{subsec1.3} the Bargmann transform and
recall some topological spaces of entire
functions or power series expansions on $\cc d$. The section is concluded with a
review of some facts on pseudo-differential operators, Toeplitz operators, Wick and
anti-Wick operators.

\par

\subsection{Gelfand-Shilov spaces}\label{subsec1.2}
Let $0<s \in \mathbf R$ be fixed. Then the (Fourier invariant)
Gelfand-Shilov space $\maclS _s(\rr d)$ ($\Sigma _s(\rr d)$) of
Roumieu type (Beurling type) consists of all $f\in C^\infty (\rr d)$
such that
\begin{equation}\label{gfseminorm}
\nm f{\maclS _{s,h}}\equiv \sup
\left (
\frac {|x^\alpha \partial ^\beta
f(x)|}{h^{|\alpha  + \beta |}(\alpha !\, \beta !)^s}
\right )
\end{equation}
is finite for some $h>0$ (for every $h>0$). Here the supremum should be taken
over all $\alpha ,\beta \in \mathbf N^d$ and $x\in \rr d$. The semi-norms
$\nm \cdo {\maclS _{s,h}^\sigma}$ induce an inductive limit topology for the
space $\maclS _s(\rr d)$ and projective limit topology for $\Sigma _s(\rr d)$, and
the latter space becomes a Fr{\'e}chet space under this topology.

\par

The space $\maclS _s(\rr d)\neq \{ 0\}$ ($\Sigma _s(\rr d)\neq \{0\}$), if and only if
$s\ge \frac 12$ ($s> \frac 12$).

\par

The \emph{Gelfand-Shilov distribution spaces} $\maclS _s'(\rr d)$
and $\Sigma _s'(\rr d)$ are the (strong) duals of $\maclS _s(\rr d)$
and $\Sigma _s(\rr d)$, respectively.

\par

Let $(\cdo ,\cdo )_{L^2}$
be the scalar product in $L^2(\rr d)$. Then the duality between
$\maclS _s(\rr d)$ and $\maclS _s'(\rr d)$ can be obtained
by straight-forward extensions of the restriction of $(\cdo ,\cdo )_{L^2}$
to $\maclS _{1/2}(\rr d)$. In this setting we have
\begin{equation}\label{GSembeddings}
\begin{aligned}
\maclS _{1/2} (\rr d) &\hookrightarrow \Sigma _s  (\rr d) \hookrightarrow
\maclS _s (\rr d)
\hookrightarrow  \Sigma _t(\rr d)
\\[1ex]
&\hookrightarrow
\mascS (\rr d)
\hookrightarrow \mascS '(\rr d) 
\hookrightarrow \Sigma _t' (\rr d)
\\[1ex]
&\hookrightarrow  \maclS _s'(\rr d)
\hookrightarrow  \Sigma _s'(\rr d) \hookrightarrow \maclS _{1/2} '(\rr d),
\quad \frac 12<s<t.
\end{aligned}
\end{equation}
Here and
in what follows we use the notation $A\hookrightarrow B$ when the topological
spaces $A$ and $B$ satisfy $A\subseteq B$ with continuous embeddings.

\par

A convenient family of functions concerns the Hermite functions
$$
h_\alpha (x) = \pi ^{-\frac d4}(-1)^{|\alpha |}
(2^{|\alpha |}\alpha !)^{-\frac 12}e^{\frac 12{|x|^2}}
(\partial ^\alpha e^{-|x|^2}),\quad \alpha \in \nn d.
$$
The set of Hermite functions on $\rr d$ is an orthonormal basis for
$L^2(\rr d)$. It is also a basis for the Schwartz space and its distribution space,
and for any $\Sigma _s$ when $s>\frac 12$,
$\maclS _s$ when $s\ge \frac 12$ and their distribution
spaces. They are also eigenfunctions to the Harmonic
oscillator $H=H_d\equiv |x|^2-\Delta$ and to the Fourier transform
$\mathscr F$, given by
$$
(\mathscr Ff)(\xi )= \widehat f(\xi ) \equiv (2\pi )^{-\frac d2}\int _{\rr
{d}} f(x)e^{-i\scal  x\xi }\, dx, \quad \xi \in \rr d,
$$
when $f\in L^1(\rr d)$. Here $\scal \cdo \cdo$ denotes the usual
scalar product on $\rr d$. In fact, we have
$$
H_dh_\alpha = (2|\alpha |+d)h_\alpha .
$$

\par

The Fourier transform $\mathscr F$ extends
uniquely to homeomorphisms on $\mathscr S'(\rr d)$,
$\maclS _s'(\rr d)$ and on $\Sigma _s'(\rr d)$. Furthermore,
$\mascF$ restricts to
homeomorphisms on $\mathscr S(\rr d)$,
$\maclS _s(\rr d)$ and on $\Sigma _s (\rr d)$,
and to a unitary operator on $L^2(\rr d)$. Similar facts hold true
when the Fourier transform is replaced by a partial
Fourier transform.

\par

Gelfand-Shilov spaces and their distribution spaces can also
be characterized by estimates of short-time Fourier
transforms, (see e.{\,}g. \cite{GZ,Teof,Toft18}).
They also obey convenient mapping properties under short-time
Fourier transforms.
More precisely, let $\phi \in \mascS  (\rr d)$ be
fixed.
Then the \emph{short-time
Fourier transform} $V_\phi f$ of $f\in \mascS '
(\rr d)$ with respect to the \emph{window function} $\phi$ is
the Gelfand-Shilov distribution on $\rr {2d}$, defined by
\begin{equation}\label{Eq:STFT}
V_\phi f(x,\xi )  =
\mascF (f \, \overline {\phi (\cdo -x)})(\xi ), \quad x,\xi \in \rr d.
\end{equation}
If $f ,\phi \in \mascS (\rr d)$, then it follows that
$$
V_\phi f(x,\xi ) = (2\pi )^{-\frac d2}\int _{\rr d} f(y)\overline {\phi
(y-x)}e^{-i\scal y\xi}\, dy, \quad x,\xi \in \rr d.
$$

\par

By \cite[Theorem 2.3]{To11} it follows that the definition of the map
$(f,\phi)\mapsto V_{\phi} f$ from $\mascS (\rr d) \times \mascS (\rr d)$
to $\mascS(\rr {2d})$ is uniquely extendable to a continuous map from
$\maclS _s'(\rr d)\times \maclS_s'(\rr d)$
to $\maclS_s'(\rr {2d})$, and restricts to a continuous map
from $\maclS _s (\rr d)\times \maclS _s (\rr d)$
to $\maclS _s(\rr {2d})$.
The same conclusions hold with $\Sigma _s$ in place of
$\maclS_s$, at each place.

\par

\subsection{Spaces of sequences, Hermite series and power series expansions}
\label{subsec1.3}

\par

Next we recall the definitions of topological vector spaces of
Hermite series expansions, given in \cite{Toft18}. As in \cite{Toft18},
it is convenient to use suitable extensions of
$\mathbf R_+$ when indexing our spaces. 

\par

\begin{defn}
The sets $\mathbf R_\flat$, $\overline {\mathbf R_\flat}$
$\mathbf R_{\flat ,\infty}$ and $\overline {\mathbf R_{\flat ,\infty}}$ are given by
$$
{\textstyle{\mathbf R_\flat = \mathbf R_+ \underset{\sigma >0}{\textstyle{\bigcup}}
\{ \flat _\sigma \} }},
\quad
{\textstyle{\overline {\mathbf R_\flat} = \mathbf R_\flat \bigcup \{ 0 \} }},
\quad
{\textstyle{\mathbf R_{\flat ,\infty} = \mathbf R_\flat \bigcup \{ \infty \} }}
\quad \text{and}\quad
{\textstyle{\overline {\mathbf R_{\flat ,\infty}} = \overline {\mathbf R_\flat} \bigcup \{ \infty \} }}.
$$

\par

Beside the usual ordering in $\overline {\mathbf R_+}\bigcup \{\infty \}$,
the elements $\flat _\sigma$
in these subsets of $\overline {\mathbf R_{\flat ,\infty}}$ are ordered by
the relations $x_1<\flat _{\sigma _1}<\flat _{\sigma _2}<x_2$, when
$\sigma _1,\sigma _2,x_1,x_2\in \mathbf R_+$ satisfy
$x_1<\frac 12$, $x_2\ge \frac 12$ and $\sigma _1<\sigma _2$.
\end{defn}

\par

In order for defining the sequence spaces we shall make use of the weight
\begin{align}
\vartheta _{r,s}(\alpha )
&\equiv
\begin{cases}
1 & \text{when}\quad s=0,\ |\alpha |\le r,
\\[1ex]
\infty & \text{when}\quad s=0,\ |\alpha |> r,
\\[1ex]
e^{r|\alpha |^{\frac 1{2s}}} & \text{when}\quad s\in \mathbf R_+,
\\[1ex]
r^{|\alpha |}(\alpha !)^{\frac 1{2\sigma}} & \text{when}\quad s = \flat _\sigma ,
\\[1ex]
\eabs \alpha ^r & \text{when}\quad s=\infty ,
\qquad \alpha \in \nn d,
\end{cases}
\label{Eq:SeqWeightDef}
\end{align}
and we observe that such weights obey estimates given in the following
lemma.

\par

\begin{lemma}\label{Lemma:ThetaWeightBasic}
Let $\alpha ,\beta \in \nn d$ and $s,\sigma \in \mathbf R_+$. Then
\begin{equation}\label{Eq:ThetaWeightBasic1}
\begin{aligned}
\vartheta _{c_1r,s}(\alpha )\vartheta _{c_1r,s}(\beta )
&\le
\vartheta _{r,s}(\alpha +\beta )
\le
\vartheta _{c_2r,s}(\alpha )\vartheta _{c_2r,s}(\beta ),
\\[1ex]
c_1 &= \min (2^{\frac 1{2s}-1},1),\ c_2=\max (2^{\frac 1{2s}-1},1),
\end{aligned}
\end{equation}
and
\begin{equation}\label{Eq:ThetaWeightBasic2}
\begin{aligned}
\vartheta _{r,\flat _\sigma}(\alpha )\vartheta _{r,s}(\beta )
&\le
\vartheta _{r,\flat _\sigma}(\alpha +\beta )
\le
C\vartheta _{c\, r,\flat _\sigma}(\alpha )\vartheta _{c\, r,s}(\beta ),
\\[1ex]
c&=2^{\frac 1{2\sigma}},
\quad
C=
\begin{cases}
1, & |\alpha +\beta |=0,
\\[1ex]
2^{-\frac 1{2\sigma}}, & |\alpha +\beta | > 0.
\end{cases}
\end{aligned}
\end{equation}

\par

If $R\ge 0$ and in addition $s<\frac 12$, then it also holds
\begin{equation}\label{Eq:ThetaWeightBasic3}
R^{|\alpha |} \lesssim
\vartheta _{r,s}(\alpha )
\quad \text{and}\quad
R^{|\alpha |} \lesssim
\vartheta _{r,\flat _\sigma}(\alpha ),
\quad \alpha \in \nn d.
\end{equation}
\end{lemma}

\par

\begin{proof}
Let $\theta =\frac 1{2s}$. Then the inequalities
in \eqref{Eq:ThetaWeightBasic1} follows from
$$
c_1(s^\theta +t^\theta )\le (s+t)^\theta \le c_2(s^\theta +t^\theta ),
\quad
s,t\ge 0.
$$
The inequalities in \eqref{Eq:ThetaWeightBasic2} follows from
$$
\vartheta _{r,\flat _\sigma}(\alpha +\beta ) = \vartheta _{r,\flat _\sigma}(\alpha )
\vartheta _{r,\flat _\sigma}(\beta )
{{\alpha +\beta}\choose {\alpha}}^{\frac 1{2\sigma}},
$$
and that
$$
1\le {{\alpha +\beta}\choose {\alpha}} \le 2^{|\alpha +\beta |-1},
\quad \text{when}\quad
|\alpha +\beta |\ge 1.
$$

\par

The estimates in \eqref{Eq:ThetaWeightBasic3} follows from
$$
R^{|\alpha |}\lesssim \alpha !^{\frac 1{2\sigma}}
\lesssim
R^{|\alpha |^{\frac 1{2s}}}
$$
when $R>1$ and $s<\frac 12$, and the result follows.
\end{proof}

\par

\begin{defn}\label{DefSeqSpaces}
Let $p\in (0,\infty ]$, $s\in \overline {\mathbf R_{\flat ,\infty}}$, $r\in \mathbf R$,
$\vartheta$ be a map from $\nn d$ to $\mathbf R_+\bigcup \infty$,
and let $\vartheta _{r,s}$ be as in \eqref{Eq:SeqWeightDef}.
\begin{enumerate}
\item The set $\ell ^p_{[\vartheta ]}(\nn d)$ consists of
all sequences $\{ c (\alpha) \} _{\alpha \in \nn d} \subseteq \mathbf C$
such that
$$
\nm {\{ c (\alpha) \} _{\alpha \in \nn d} }{\ell ^p_{[\vartheta ]}}\equiv
\nm {\{ c (\alpha) \vartheta (\alpha )\} _{\alpha \in \nn d} }{\ell ^p} < \infty ;
$$

\vrum

\item if $s<\infty$, then
\begin{equation}\label{Eq:SeqSpaces1}
\ell _s(\nn d)\equiv \underset {r>0}\indlim \ell ^p_{[\vartheta _{r,s}]}(\nn d)
\quad \text{and}\quad
\ell _s^\star (\nn d)\equiv \underset {r>0} \projlim \ell ^p_{[1/\vartheta _{r,s}]}(\nn d)
\text ;
\end{equation}

\vrum

\item if $s>0$, then
\begin{equation}\label{Eq:SeqSpaces2}
\ell _{0,s}(\nn d)\equiv \underset {r>0}
\projlim \ell ^p_{[\vartheta _{r,s}]}(\nn d)
\quad \text{and} \quad
\ell _{0,s}^\star (\nn d)\equiv \underset {r>0}\indlim
\ell ^p_{[1/\vartheta _{r,s}]}(\nn d).
\end{equation}
\end{enumerate}
\end{defn}

\par

We observe that $\ell ^p_{[\vartheta _{r,s}]}(\nn d)$ in Definition
\ref{DefSeqSpaces} is a quasi-Banach
space under the quasi-norm $\nm \cdo{\ell ^p_{[\vartheta _{r,s}]}}$ when
$s>0$. If in addition $p\ge 1$, then $\ell ^p_{[\vartheta _{r,s}]}(\nn d)$
is a Banach space with norm $\nm \cdo{\ell ^p_{[\vartheta _{r,s}]}}$.

\par

In several situations we deal with  two-parameter version of the
spaces in Definition \ref{DefSeqSpaces}, which are denoted by
\begin{alignat}{5}
&\ell _{\maclA ,0,(s_2,s_1)}(\Lambda ), &
\quad
&\ell _{\maclA ,(s_2,s_1)} (\Lambda ), &
\quad
&\ell _{\maclA ,0,(s_2,s_1)}^\star (\Lambda ), &
\quad  & &  
& \ell _{\maclA ,(s_2,s_1)}^\star  (\Lambda ),
\label{Eq:ellAparenSpaces1to4}
\\[1ex]
&\ell _{\maclB ,0,(s_2,s_1)}(\Lambda ), &
\quad
&\ell _{\maclB ,(s_2,s_1)}(\Lambda ), &
\quad
&\ell _{\maclB ,0,(s_2,s_1)}^\star (\Lambda ), &
\quad &  & 
&\ell _{\maclB ,(s_2,s_1)}^\star (\Lambda ),
\label{Eq:ellBparenSpaces1to4}
\\[1ex]
&\ell _{\maclC ,0,(s_2,s_1)}(\Lambda ), &
\quad
&\ell _{\maclC ,(s_2,s_1)}(\Lambda ), &
\quad
&\ell _{\maclC ,0,(s_2,s_1)}^\star (\Lambda ) &
\quad &\text{and} &\quad
&\ell _{\maclC ,(s_2,s_1)}^\star (\Lambda ).
\label{Eq:ellCparenSpaces1to4}
\end{alignat}

\par

\begin{defn}\label{Def:SeqSpacesStraightMixed}
Let $\vartheta _{r,s}$ be as in \eqref{Eq:SeqWeightDef},
$s_1,s_2\in \overline{\mathbf R_{\flat ,\infty}}$,
$\Lambda =\nn {d_2}\times \nn {d_1}$, $p\in (0,\infty ]$ and let
$$
\vartheta _{r,(s_2,s_1)}(\alpha ) = \vartheta _{r,s_1}(\alpha _1)
\vartheta _{r,s_2}(\alpha _2),\quad \alpha =(\alpha _2,\alpha _1), \qquad
\alpha _1\in \nn {d_1},\ \alpha _2\in \nn {d_2}.
$$
Then the spaces in \eqref{Eq:ellAparenSpaces1to4}
are given by
\begin{alignat}{2}
\ell _{\maclA ,0,(s_2,s_1)}(\Lambda )
&\equiv \underset {r>0}\projlim
\ell ^p_{[\vartheta _{r,(s_2,s_1)}]}(\Lambda ), &
\qquad s_1,s_2 &>0,
\label{Eq:ellAparenSpaces1}
\\[1ex]
\ell _{\maclA ,(s_2,s_1)}(\Lambda )
&\equiv
\underset {r>0} \indlim \ell ^p_{[\vartheta _{r,(s_2,s_1)}]}(\Lambda ), &
\qquad s_1,s_2 &<\infty,
\label{Eq:ellAparenSpaces2}
\\[1ex]
\ell _{\maclA ,0,(s_2,s_1)}^\star (\Lambda )
&\equiv
\underset {r>0}\indlim \ell ^p_{[1/\vartheta _{r,(s_2,s_1)}]}(\Lambda ), &
\qquad s_1,s_2 &>0,
\label{Eq:ellAparenSpaces3}
\intertext{and}
\ell _{\maclA ,(s_2,s_1)}^\star(\Lambda )
&\equiv
\underset {r>0} \projlim \ell ^p_{[1/\vartheta _{r,(s_2,s_1)}]}(\Lambda ), &
\qquad s_1,s_2 &<\infty .
\label{Eq:ellAparenSpaces4}
\end{alignat}
\end{defn}

\par

\begin{defn}\label{Def:elltauSpaces2}
Let $\vartheta _{r,s}$ be as in Definition \ref{DefSeqSpaces},
$s_1,s_2\in \overline{\mathbf R_{\flat ,\infty}}$,
$\Lambda =\nn {d_2}\times \nn {d_1}$, $p\in (0,\infty ]$ and let
$$
\omega _{s_2,s_1;r_2,r_1}(\alpha )
\equiv 
\frac {\vartheta _{r_2,s_2}(\alpha _2)}{\vartheta _{r_1,s_1}(\alpha _1)},
\quad
\alpha = (\alpha _2,\alpha _1),
\qquad \alpha _1\in \nn {d_1},\ \alpha _2\in \nn {d_2}.
$$
Then the spaces in \eqref{Eq:ellBparenSpaces1to4}
are given by
\begin{alignat}{2}
\ell _{\maclB ,0,(s_2,s_1)}(\Lambda )
&= \underset {r_2>0} \projlim
\left (
\underset {r_1>0} \indlim
\left ( 
\ell _{[\omega _{s_2,s_1;r_2,r_1}]}^p (\Lambda )
\right ) \right ), &
\quad
s_1,s_2 &>0 ,
\label{Eq:ellBparenSpaces1}
\\[1ex]
\ell _{\maclB ,(s_2,s_1)}(\Lambda )
&= \underset {r_1>0} \projlim
\left (
\underset {r_2>0} \indlim
\left ( 
\ell _{[\omega _{s_2,s_1;r_2,r_1}]}^p (\Lambda )
\right ) \right ), &
\quad
s_1,s_2 &<\infty ,
\label{Eq:ellBparenSpaces2}
\\[1ex]
\ell _{\maclB ,0,(s_2,s_1)}^\star (\Lambda )
&= \underset {r_1>0} \projlim
\left (
\underset {r_2>0} \indlim
\left ( 
\ell _{[1/\omega _{s_2,s_1;r_2,r_1}]}^p (\Lambda )
\right ) \right ), &
\quad
s_1,s_2 &>0.
\label{Eq:ellBparenSpaces4}
\intertext{and}
\ell _{\maclB ,(s_2,s_1)}^\star (\Lambda )
&= \underset {r_2>0} \projlim
\left (
\underset {r_1>0} \indlim
\left ( 
\ell _{[1/\omega _{s_2,s_1;r_2,r_1}]}^p (\Lambda )
\right ) \right ), &
\quad
s_1,s_2 &<\infty .
\label{Eq:ellBparenSpaces3}
\end{alignat}
The spaces in \eqref{Eq:ellCparenSpaces1to4} for
admissible $s_1,s_2$ are given by
the right-hand sides of
\eqref{Eq:ellBparenSpaces1}--\eqref{Eq:ellBparenSpaces3},
after the orders of inductive and projective limits are swapped.
\end{defn}

\par

\begin{rem}\label{Rem:LebExpIndep}
By playing with $r$,$r_1$ and $r_2$ it follows that the topological vector spaces in
\eqref{Eq:SeqSpaces1}, \eqref{Eq:SeqSpaces2},
\eqref{Eq:ellAparenSpaces1}--\eqref{Eq:ellAparenSpaces4},
\eqref{Eq:ellBparenSpaces1to4}
and
\eqref{Eq:ellCparenSpaces1to4} are independent of
$p\in (0,\infty ]$.
\end{rem}

\par

For conveniency
we also complete the spaces in Definitions \ref{DefSeqSpaces},
\ref{Def:SeqSpacesStraightMixed} and \ref{Def:elltauSpaces2}
with the following.

\par

\begin{defn}\label{Def:CompletingSpacesEndcases}
Let $s_1,s_2\in \overline{\mathbf R_{\flat ,\infty}}$ and
$\Lambda = \nn {d_2}\times \nn {d_1}$. 
\begin{enumerate}
\item $\ell _\infty (\nn d)$ and $\ell _\infty ^\star(\nn d)$ are given by
\eqref{Eq:SeqSpaces1} with $s=\infty$ and $p=2$;

\vrum

\item if $s_1=0$ or $s_2=0$, then
\begin{align*}
\ell _{0,0}(\Lambda )
&=
\ell _{0,0}^\star (\Lambda )
=
\ell _{\maclA ,0,(s_2,s_1)}(\Lambda )
=
\ell _{\maclB ,0,(s_2,s_1)}(\Lambda )
=
\ell _{\maclC ,0,(s_2,s_1)}(\Lambda )
\\[1ex]
&=
\ell _{\maclA ,0,(s_2,s_1)}^\star(\Lambda )
=
\ell _{\maclB ,0,(s_2,s_1)}^\star (\Lambda )
=
\ell _{\maclC ,0,(s_2,s_1)}^\star (\Lambda )
\equiv
\{ 0 \} = \big \{  \{ 0 \}_{\alpha \in \Lambda}  \big \} \text ;
\end{align*}

\vrum

\item if $s_1=\infty$ or $s_2=\infty$, then
$$
\ell _{\maclA ,(s_2,s_1)}(\Lambda ),
\quad
\ell _{\maclA ,(s_2,s_1)}^\star(\Lambda ),
\quad
\ell _{\maclB ,(s_2,s_1)}(\Lambda )
\quad \text{and}\quad
\ell _{\maclB ,(s_2,s_1)}^\star (\Lambda )
$$
are defined by \eqref{Eq:ellAparenSpaces2},
\eqref{Eq:ellAparenSpaces4}, \eqref{Eq:ellBparenSpaces2}
and \eqref{Eq:ellBparenSpaces4} with $p=2$, and
$\ell _{\maclC ,s_1,s_2}(\Lambda )$ and
$\ell _{\maclC ,s_1,s_2}^\star (\Lambda )$ are defined by
\eqref{Eq:ellBparenSpaces2}
and \eqref{Eq:ellBparenSpaces4} with $p=2$ after the orders of inductive
and projective limits are swapped.
\end{enumerate}
\end{defn}

\par

\begin{rem}\label{Rem:FeaturesSeqSpaces}
By the definitions it follows that \eqref{Eq:SeqSpaces1},
\eqref{Eq:SeqSpaces2},
\eqref{Eq:ellAparenSpaces1}--\eqref{Eq:ellAparenSpaces4}
and
\eqref{Eq:ellBparenSpaces1}--\eqref{Eq:ellBparenSpaces4}
hold true with $\bigcup$ and $\bigcap$ in place of
$\indlim$ and $\projlim$, respectively, at each occurrence.

\par

We observe that the following holds true:
\begin{enumerate}
\item the space $\ell _0^\star (\nn d)$ is the set of all sequences
$\{c (\alpha) \} _{\alpha \in \nn d} \subseteq \mathbf C$ on $\nn d$, and that
$\ell _0(\nn d)$ is the set of all such sequences
such that $c (\alpha) \neq 0$
for at most finite numbers of $\alpha$.
In similar ways, the condition $s_1=0$ or $s_2=0$ impose
support restrictions of the elements in the spaces
\eqref{Eq:ellAparenSpaces1to4}--\eqref{Eq:ellCparenSpaces1to4};

\vrum

\item the spaces in
\eqref{Eq:ellAparenSpaces1to4}
are complete Hausdorff topological vector spaces,
and that $\ell _{\maclA ,0,(s_2,s_1)}(\Lambda )$ and
$\ell _{\maclA ,(s_2,s_1)}^\star(\Lambda )$
are Fr{\'e}chet spaces. It holds that
$\ell _{\maclA ,(s,s)} = \ell _s$, $\ell _{\maclA ,0,(s,s)} = \ell _{0,s}$;

\vrum

\item $\ell _0(\Lambda )$ is dense in all of the spaces in
\eqref{Eq:ellAparenSpaces1to4}--\eqref{Eq:ellCparenSpaces1to4}.
\end{enumerate}
\end{rem}

\par

\begin{rem}\label{Rem:elltauSpaces2}
Let $s_1,s_2\in \overline{\mathbf R_\flat}$ and
$\Lambda =\nn {d_2}\times \nn {d_1}$. Then it follows by straight-forward
computations that the $\ell ^2(\Lambda )$ form
$(\cdo ,\cdo )_{\ell ^2(\Lambda )}$
on $\ell _0(\Lambda )$
is uniquely extendable to continuous mappings from
\begin{equation}\label{Eq:DualPairsDiscr}
\begin{aligned}
\ell _{\maclA ,(s_2,s_1)}^\star (\Lambda )\times
\ell _{\maclA ,(s_2,s_1)}(\Lambda ),
\qquad
&\ell _{\maclC ,(s_2,s_1)}^\star (\Lambda )\times
\ell _{\maclB ,(s_2,s_1)}(\Lambda )
\\[1ex]
\text{and} \qquad
&\ell _{\maclB ,(s_2,s_1)}^\star (\Lambda )\times
\ell _{\maclC ,(s_2,s_1)}(\Lambda )
\end{aligned}
\end{equation}
to $\mathbf C$. These mappings constitutes dualities
in the sense that they are non-degenerate, i.{\,}e. if
$c_1\in \ell _{\maclA ,(s_2,s_1)}^\star (\Lambda )\setminus 0$
and
$c_2\in \ell _{\maclA ,(s_2,s_1)} (\Lambda )\setminus 0$,
then
$$
\sets {(c_1,c)_{\ell ^2}}{c\in \ell _{\maclA ,(s_2,s_1)} (\Lambda )}\neq 0
\quad \text{and}\quad
\sets {(c,c_2)_{\ell ^2}}{c\in \ell _{\maclA ,(s_2,s_1)}^\star (\Lambda )}\neq 0,
$$
and similarly for the other pairs of spaces in
\eqref{Eq:DualPairsDiscr}. We also observe that
$\ell _{\maclA ,(s_2,s_1)}^\star (\Lambda )$ and
$\ell _{\maclA ,(s_2,s_1)}(\Lambda )$ are (strong)
duals to each others, through the form
$(\cdo ,\cdo )_{\ell ^2(\Lambda )}$.
%

\par

If instead $s_1,s_2\in \mathbf R_{\flat ,\infty}$, then
the same holds true with $\ell _{\maclA ,0,(s_2,s_1)}$,
$\ell _{\maclB ,0,(s_2,s_1)}$ and $\ell _{\maclC ,0,(s_2,s_1)}$
in place of $\ell _{\maclA ,(s_2,s_1)}$, $\ell _{\maclB ,(s_2,s_1)}$
and $\ell _{\maclC ,(s_2,s_1)}$, respectively, at each occurrence.
\end{rem}

\par

Next we consider, in similar ways as in \cite{Toft18}, spaces of formal
Hermite series expansions
\begin{alignat}{2}
f&=\sum _{\alpha \in \nn d}c (f;\alpha) h_\alpha ,&\quad 
\{ c (\alpha) \}
_{\alpha \in \nn d} &\in \ell _0^\star (\nn d).\label{Hermiteseries}
\intertext{and spaces of formal power series expansions}
F&=\sum _{\alpha \in \nn d}c (F;\alpha) e_\alpha ,&\quad \{c (\alpha) \}
_{\alpha \in \nn d} &\in \ell _0^\star (\nn d).\label{Powerseries}
\end{alignat}
which correspond to
\begin{equation}\label{ellSpaces}
\ell _{0,s}(\nn d),\quad \ell _s(\nn d),
\ell _{0,s}^\star(\nn d)
\quad \text{and}\quad \quad \ell _s^\star(\nn d).
\end{equation}
Here
\begin{equation} \label{Eq:basiselements}
e_\alpha (z) \equiv \frac {z^\alpha}{\sqrt {\alpha !}},
\qquad z\in \cc d,\ \alpha \in \nn d.
\end{equation}
We consider the mappings
\begin{equation}
\label{T12Map}
T_{\maclH} : \,
\{ c (\alpha) \} _{\alpha \in \nn d} \mapsto \sum _{\alpha \in \nn d}
c (\alpha) h_\alpha
\quad \text{and}\quad
T_{\maclA} : \,
\{ c (\alpha) \} _{\alpha \in \nn d} \mapsto \sum _{\alpha \in \nn d}
c (\alpha) e_\alpha
\end{equation}
between sequences, and formal Hermite series and power series expansions.

\par

\begin{defn}\label{DefclHclASpaces}
If $s\in \overline{\mathbf R_{\flat ,\infty}}$, then
\begin{alignat}{2}
\maclH _{0,s}(\rr d),\quad \maclH _s(\rr d),
\quad
\maclH _{0,s}^\star(\rr d)
\quad \text{and}\quad \maclH _s^\star(\rr d),
\label{clHSpaces}
\intertext{and}
\maclA _{0,s}(\cc d),\quad \maclA _s(\cc d),
\quad
\maclA _{0,s}^\star(\cc d)
\quad \text{and}\quad 
\maclA _s^\star(\cc d),
\label{clASpaces}
\end{alignat}
are the images of $T_{\maclH}$ and $T_{\maclA}$ respectively in
\eqref{T12Map} of corresponding spaces in \eqref{ellSpaces}.
The topologies of the spaces in \eqref{clHSpaces} and \eqref{clASpaces}
are inherited from the corresponding spaces in \eqref{ellSpaces}.
\end{defn}

\par

By Remark \ref{Rem:elltauSpaces2} it follows that the latter spaces in
\eqref{clHSpaces} are the (strong) duals of
the former spaces with respect to unique extensions of the form
$(\cdo ,\cdo )_{L^2}$ on $\maclH _0(\rr d)$. That is,
\begin{equation}\label{Eq:DualPilSp}
\maclH _s'(\rr d) = \maclH _s^\star (\rr d)
\quad \text{and}\quad
\maclH _{0,s} '(\rr d) = \maclH _{0,s} ^\star (\rr d).
\end{equation}
In the same way as in \eqref{clASpaces} we let
\begin{alignat}{5}
&\maclA _{0,(s_2,s_1)}(W), &
\quad
&\maclA _{(s_2,s_1)}(W), &
\quad
&\maclA _{0,(s_2,s_1)}^\star(W) &
\quad &\text{and} &\quad
&\maclA _{(s_2,s_1)}^\star(W)
\tag*{(\ref{clASpaces})$'$}
\end{alignat}
be the images of the spaces \eqref{Eq:ellAparenSpaces1to4}
under the map $T_{\maclA}$
when $W=\cc {d_2}\times \cc {d_1}$, $\Lambda =\nn {d_2}\times \nn {d_1}$
and $s_1,s_2\in \overline{\mathbf R_{\flat ,\infty}}$.

\par

Since locally absolutely convergent power series expansions
can be identified with
entire functions, several of the spaces in \eqref{clASpaces} are identified
with topological vector spaces contained in $A(\cc d)$ or $A(W)$.
(See \cite{Toft18}, Theorem \ref{Thm:AnalSpacesChar} in Appendix
\ref{App:A} and the introduction.)
Here $A(\Omega )$
is the set of all (complex valued) functions which are analytic in
the open set $\Omega \subseteq \cc d$. If $\Omega _0\subseteq \cc d$
is arbitrary, then $A(\Omega _0)=\cup A(\Omega )$, where the union is taken
over all open sets $\Omega \subseteq \cc d$ such that
$\Omega _0\subseteq \Omega$.

\par

We recall that $f\in \mascS (\rr d)$ if and only if
\eqref{Hermiteseries} holds with
$$
|c (\alpha) |\lesssim \eabs \alpha ^{-N},
$$
for every $N\ge 0$.
That is, $\mascS (\rr d)=\maclH _{0,\infty}(\rr d)$. In the same way,
$\mascS '(\rr d)=\maclH _{0,\infty}^\star(\rr d)$, and $f\in L^2(\rr d)$,
if and only if $\{ c(\alpha )\} _{\alpha \in \nn d}\in \ell ^2(\nn d)$
(cf. e.{\,}g. \cite{RS}). In particular it follows from the definitions that the
inclusions
\begin{align}
\maclH _0(\rr d) &\hookrightarrow \maclH _{0,s}(\rr d)\hookrightarrow
\maclH _{s}(\rr d) \hookrightarrow \maclH _{0,t}(\rr d)
\notag
\\[1ex]
&\hookrightarrow
\maclH _{0,t}^\star(\rr d)
\hookrightarrow \maclH _{s}^\star(\rr d)
\hookrightarrow \maclH _{0,s}^\star(\rr d)\hookrightarrow \maclH _0^\star(\rr d),
\quad
s,t\in \mathbf R_{\flat ,\infty} ,\ s<t,
\label{inclHermExpSpaces}
\end{align}
are dense.

\par

\begin{rem}\label{Rem:LinkEllHs}
By the definition it follows that $T_{\maclH}$ in \eqref{T12Map} is a homeomorphism
between any of the spaces in \eqref{ellSpaces} and corresponding space
in \eqref{clHSpaces}, and that $T_{\maclA}$ in \eqref{T12Map} is a homeomorphism
between any of the spaces in \eqref{ellSpaces} and corresponding space
in \eqref{clASpaces}.
\end{rem}

\par

The next results give some characterizations of $\maclH _s(\rr d)$ and
$\maclH _{0,s}(\rr d)$ when $s$ is a non-negative real number. See also
\cite{AbFeGaToUs} for similar characterizations of $\maclH _{\flat _\sigma}(\rr d)$
and $\maclH _{0,\flat _\sigma}(\rr d)$.

\par

\begin{prop}\label{Prop:PilSpacChar1}
Let $0\le s,s_1,s_2\in \mathbf R$ and let $f\in \maclH _0^\star(\rr d)$.
Then $f\in \maclH _s(\rr d)$ ($f\in \maclH _{0,s}(\rr d)$), if and only if
$f\in C^\infty (\rr d)$ and satisfies
\begin{equation}\label{GFHarmCond}
\nm{H_d^Nf}{L^\infty}\lesssim h^NN!^{2s},
\end{equation}
for some (every) $h>0$. If $s_1<\frac 12$ and
$s_2\ge \frac 12$ ($0<s_1\le \frac 12$ and $s_2> \frac 12$), then
\begin{alignat*}{2}
\maclH _{s_1}(\rr d) &\neq \maclS _{s_1}(\rr d) = \{ 0\}, &
\quad
\big (
\maclH _{0,s_1}(\rr d) &\neq \Sigma _{s_1}(\rr d) \neq \{ 0\}
\big ),
\\[1ex]
\maclH _{s_2}(\rr d) &= \maclS _{s_2}(\rr d) \neq \{ 0\} &
\quad \text{and}\quad
\big (
\maclH _{0,s_2}(\rr d) &= \Sigma _{s_2}(\rr d) \neq \{ 0\} 
\big ) .
\end{alignat*}
\end{prop}

\par

We refer to \cite{Toft18} for the proof of Proposition \ref{Prop:PilSpacChar1}.

\par

Due to the pioneering investigations related
to Proposition \ref{Prop:PilSpacChar1} by Pilipovi{\'c} in
\cite{Pil}, we call the spaces $\maclH _s(\rr d)$ and $\maclH _{0,s}(\rr d)$
\emph{Pilipovi{\'c} spaces of Roumieu and Beurling types}, respectively.
In fact, in the restricted case $s,s_1,s_2\ge \frac 12$,
Proposition \ref{Prop:PilSpacChar1} was proved
already in \cite{Pil}.

\medspace

In what follows we let $F(z_2,\overline z_1)$ be
the formal power series
\begin{equation}\label{Eq:FormalConjPowerSeries}
\sum
c(\alpha _2,\alpha _1)e_{\alpha _2}(z_2)e_{\alpha _1}(\overline z_1),
\quad z_j \in \cc {d_j}, j=1,2, 
\end{equation}
when $F(z_2,z_1)$ is the formal power series
\begin{equation}\label{Eq:FormalNonConjPowerSeries}
\sum
c(\alpha _2,\alpha _1)e_{\alpha _2}(z_2)e_{\alpha _1}(z_1).
\end{equation}
Here $z_j \in \cc {d_j}$, $j=1,2$, and the sums should be taken over all
$(\alpha _2,\alpha _1)\in \nn {d_2}\times \nn {d_1}$.

\par

\begin{defn}\label{Def:tauSpaces}
Let $s,s_1,s_2\in \overline{\mathbf R_{\flat ,\infty}}$,
$\Theta _{C}$, $W=\cc {d_2}\times \cc {d_1}$ and be the operator
\begin{equation}\label{Eq:ThetaCOp}
(\Theta _{C}F)(z_2,z_1)=F(z_2,\overline z_1)
\end{equation}
between formal power series in \eqref{Eq:FormalConjPowerSeries}
and \eqref{Eq:FormalNonConjPowerSeries}, $z_j \in \cc {d_j}$, $j=1,2$.
Then
\begin{alignat}{4}
&\wideparen \maclA _{0,s}(W), &
\quad
&\wideparen \maclA _s(W), &
\quad
&\wideparen \maclA _s^\star(W), &
\quad
&\wideparen \maclA _{0,s}^\star(W)
\label{Eq:SesAnalSp}
\intertext{and}
&\wideparen \maclA _{0,(s_2,s_1)}(W), &
\quad
&\wideparen \maclA _{(s_2,s_1)}(W), &
\quad
&\wideparen \maclA _{(s_2,s_1)}^\star(W), &
\quad
&\wideparen \maclA _{0,(s_2,s_1)}^\star(W)
\label{Eq:SesAnalSp2}
\end{alignat}
are the images of \eqref{clASpaces} respectively \eqref{clASpaces}$'$
under $\Theta _{C}$, and $\wideparen A (W)$
is the image of $A(W)$
under $\Theta _{C}$.
The topologies of the spaces in \eqref{Eq:SesAnalSp}, \eqref{Eq:SesAnalSp2}
and $\wideparen A (W)$
are inherited from the topologies
of the spaces \eqref{clASpaces}, \eqref{clASpaces}$'$ and $A(W)$
respectively, through the map $\Theta _{C}$.
\end{defn}

\par

\begin{rem}\label{Rem:SpaceSpecCase}
Let $W=\cc {d_2}\times \cc {d_1}$. By letting $d_2=d$ and $d_1=0$,
it follows that $A(\cc d)$ and the spaces in \eqref{clASpaces} can be
considered as special cases of $\wideparen A (W)$ and the spaces
in \eqref{Eq:SesAnalSp}.

\par

Since $\maclA _{\flat _1}^\star(\cc d) = A(\cc d)$ and
$\maclA _{0,\flat _1}^\star(\cc d)
= A_d(\{ 0\} )$, it follows that
\begin{equation}\label{Eq:AnalIdents}
\begin{aligned}
\wideparen
\maclA _{\flat _1}^\star(W)
=
\wideparen A(W)
\quad \text{and}\quad
\wideparen
\maclA _{0,\flat _1}^\star(W)
=
\wideparen A _{d_2,d_1}(\{ 0 \} ).
\end{aligned}
\end{equation}
\end{rem}

\par

The subspaces of $\wideparen A(W)$ in the following definition
are important when considering analytic kernel operators which are
continuous on the spaces in \eqref{clASpaces}.

\par

\begin{defn}\label{Def:tauSpaces2}
Let $s_1,s_2\in \overline{\mathbf R_{\flat ,\infty}}$,
$W=\cc {d_2}\times \cc {d_1}$, $T_{\maclA}$ and $\Theta _C$
be given by \eqref{T12Map} and \eqref{Eq:ThetaCOp}. Then
\begin{alignat}{5}
&\wideparen \maclB _{s}(W), &
\quad
&\wideparen \maclB _{0,s}(W), &
\quad
&\wideparen \maclB _{s}^\star (W), &
\quad 
&\wideparen \maclB _{0,s}^\star (W), &
\quad
s&=(s_2,s_1),
\label{Eq:BparenSpaces}
\intertext{and}
&\wideparen \maclC _{s}(W), &
\quad
&\wideparen \maclC _{0,s}(W), &
\quad
&\wideparen \maclC _{s}^\star (W), &
\quad 
&\wideparen \maclC _{0,s}^\star (W), &
\quad
s&=(s_2,s_1),
\label{Eq:CparenSpaces}
\end{alignat}
are the images of the spaces in \eqref{Eq:ellBparenSpaces1to4}
and \eqref{Eq:ellCparenSpaces1to4} under the
map $\Theta _C\circ T_{\maclA}$.
The topologies of the spaces in \eqref{Eq:BparenSpaces} and
\eqref{Eq:CparenSpaces} are inherited from the topologies of the
corresponding spaces in \eqref{Eq:ellBparenSpaces1to4}
and \eqref{Eq:ellCparenSpaces1to4}.
\end{defn}

\par

For conveniency we set
\begin{alignat*}{4}
\ell _{\maclB ,s} &= \ell _{\maclB ,(s,s)}, &
\quad
\ell _{\maclB ,0,s} &= \ell _{\maclB ,0,(s,s)}, &
\quad
\ell _{\maclC ,s} &= \ell _{\maclC ,(s,s)}, &
\quad
\ell _{\maclC ,0,s} &= \ell _{\maclC ,0,(s,s)},
\\[1ex]
\wideparen \maclB _s &= \wideparen \maclB _{(s,s)}, &
\quad
\wideparen \maclB _{0,s} &= \wideparen \maclB _{0,(s,s)}, &
\quad
\wideparen \maclC _s &= \wideparen \maclC _{(s,s)} &
\quad \text{and}\quad
\wideparen \maclC _{0,s} &= \wideparen \maclC _{0,(s,s)}.
\end{alignat*}
In particular, if $W=\cc d\times \cc d\simeq \cc {2d}$, then the
spaces in \eqref{Eq:BparenSpaces} and \eqref{Eq:CparenSpaces}
becomes
\begin{equation}\tag*{(\ref{Eq:BparenSpaces})$'$}
\wideparen \maclB _{s}(\cc {2d}),
\quad
\wideparen \maclB _{0,s}(\cc {2d}),
\quad
\wideparen \maclB _{s}^\star (\cc {2d}),
\quad 
\wideparen \maclB _{0,s}^\star (\cc {2d})
\end{equation}
and
\begin{equation}\tag*{(\ref{Eq:CparenSpaces})$'$}
\wideparen \maclC _{s}(\cc {2d}),
\quad
\wideparen \maclC _{0,s}(\cc {2d}),
\quad
\wideparen \maclC _{s}^\star (\cc {2d}),
\quad 
\wideparen \maclC _{0,s}^\star (\cc {2d}).
\end{equation}

\par

In Proposition \ref{Prop:DualtauSpaces} in Section \ref{sec2} we present some
duality properties of the spaces in Definitions \ref{Def:tauSpaces}
and \ref{Def:tauSpaces2}.
In Appendix \ref{App:A} we identify the spaces in \eqref{Eq:BparenSpaces}
and \eqref{Eq:CparenSpaces} with convenient subspaces of
$\wideparen A(W)$, for suitable $s_1,s_2\in \overline {\mathbf R_\flat}$.

\par

\subsection{The Bargmann transform and spaces of analytic
functions}\label{subsec1.5}

\par

Let $p\in [1,\infty ]$. Then the Bargmann transform $\mathfrak V_df$ of
$f\in L^p(\rr d)$ is the entire function given by
\begin{equation*}
(\mathfrak V_df)(z) =\pi ^{-\frac d4}\int _{\rr d}\exp \Big ( -\frac 12(\scal
z z+|y|^2)+2^{1/2}\scal zy \Big )f(y)\, dy,\quad z \in \cc d.
\end{equation*}
We have
$$
(\mathfrak V_df)(z) =\int_{\rr d} \mathfrak A_d(z,y)f(y)\, dy,
\quad z \in \cc d,
$$
or
\begin{equation}\label{bargdistrform}
(\mathfrak V_df)(z) =\scal f{\mathfrak A_d(z,\cdo )},
\quad z \in \cc d,
\end{equation}
where the Bargmann kernel $\mathfrak A_d$ is given by
$$
\mathfrak A_d(z,y)=\pi ^{-\frac d4} \exp \Big ( -\frac 12(\scal
zz+|y|^2)+2^{1/2}\scal zy\Big ), \quad z \in \cc d, y \in \rr d.
$$
(Cf. \cite{B1,B2}.)
Here
$$
\scal zw = \sum _{j=1}^dz_jw_j\quad \text{and} \quad
(z,w)= \scal z{\overline w}
$$
when
$$
z=(z_1,\dots ,z_d) \in \cc d\quad  \text{and} \quad w=(w_1,\dots ,w_d)\in \cc d,
$$
and otherwise $\scal \cdo \cdo $ denotes the duality between test function
spaces and their corresponding duals which is clear from the context.
We note that the right-hand side of \eqref{bargdistrform} makes sense
when $f\in \maclS _{1/2}'(\rr d)$ and defines an element in $A(\cc d)$,
since $y\mapsto \mathfrak A_d(z,y)$ can be interpreted as an element
in $\maclS _{1/2} (\rr d)$ with values in $A(\cc d)$.

\par

It was proved by Bargmann in \cite{B1} that $f\mapsto \mathfrak V_df$ is a bijective
and isometric map from $L^2(\rr d)$ to the Hilbert space $A^2(\cc d)$,
the set of entire functions $F$ on $\cc d$ which fullfils
\begin{equation}\label{A2norm}
\nm F{A^2}\equiv \Big ( \int _{\cc d}|F(z)|^2d\mu (z)  \Big )^{1/2}<\infty .
\end{equation}
Here $d\mu (z)=\pi ^{-d} e^{-|z|^2}\, d\lambda (z)$, where $d\lambda (z)$ is
the Lebesgue measure on $\cc d$. The scalar product on
$A^2(\cc d)$ is given by
\begin{equation}\label{A2scalar}
(F,G)_{A^2}\equiv  \int _{\cc d} F(z)\overline {G(z)}\, d\mu (z),\quad F,G\in A^2(\cc d).
\end{equation}
For future references we note that the latter scalar product induces the bilinear form
\begin{equation}\label{A2scalarBil}
(F,G)\mapsto \scal FG _{A^2}=\scal FG _{A^2(\cc d)}\equiv
\int _{\cc d} F(z)G(z)\, d\mu (z)
\end{equation}
on $A^2(\cc d)\times \overline{A^2(\cc d)}$.

\par

In \cite{B1} it was proved that the Bargmann transform maps the
Hermite functions to monomials as
\begin{equation}\label{Eq:BargmannHermiteMap}
\mathfrak V_dh_\alpha = e_\alpha ,
\quad
z\in \cc d,\quad \alpha \in \nn d
\end{equation}
(cf. \eqref{Eq:basiselements}). The orthonormal basis
$\{ h_\alpha \}_{\alpha \in \nn d} \subseteq L^2(\rr d)$ 
is thus mapped to the orthonormal basis
$\{ e_\alpha \} _{\alpha \in \nn d}\subseteq A^2(\cc d)$. 

\par

In particular it follows that the definition of the Bargmann transform
from $\maclH _0(\rr d)$ to $A^2(\cc d)$ is uniquely extendable
to a homeomorphism from $\maclH _0^\star(\rr d)=\maclH _0'(\rr d)$
to $\maclA _0^\star(\cc d)$,
by letting
\begin{equation}\label{Eq:GeneralBargmTrans}
\mathfrak V_d=T_{\maclA}\circ T_{\maclH}^{-1},
\end{equation}
where
$T_{\maclA}$ and $T_{\maclH}$, where $T_{\maclH}$ and
$T_{\maclA}$ are given by \eqref{T12Map}.
From the definitions of
$\maclH _s(\rr d)$ to $\maclA _s(\cc d)$ and their duals it follows that
the Bargmann transform restricts to homeomorphisms from
$\maclH _s(\rr d)$ to $\maclA _s(\cc d)$ and from
$\maclH _s^\star(\rr d)$ to $\maclA _s^\star (\cc d)$. Similar facts hold true with
$\maclH _{0,s}$ and $\maclA _{0,s}$ in place of
$\maclH _{s}$ and $\maclA _{s}$, respectively, at each occurrence.
(Cf. \cite{Toft18}.)
In particular, Remark \ref{Rem:elltauSpaces2}, \eqref{Eq:DualAnalPilSp}
and \eqref{Eq:BargmannHermiteMap}
show that the latter spaces in
\eqref{clASpaces} are the (strong) duals of
the former spaces with respect to unique extensions of the form
$(\cdo ,\cdo )_{A^2}$ on $\maclA _0(\cc d)$. That is,
\begin{equation}\label{Eq:DualAnalPilSp}
\maclA _s'(\cc d) = \maclA _s^\star (\cc d)
\quad \text{and}\quad
\maclA _{0,s} '(\cc d) = \maclA _{0,s} ^\star (\cc d).
\end{equation}

\par

Bargmann also proved that there is a reproducing formula for
$A^2(\cc d)$. In fact, let $\Pi _A$ be the operator from $L^2(d\mu )$
to $A(\cc d)$, given by
\begin{equation}\label{reproducing}
(\Pi _AF)(z)= \int _{\cc d} F(w)e^{(z,w)}\, d\mu (w),\quad z \in \cc d.
\end{equation}
 Then it is proved in \cite{B1} that $\Pi _A$ is an orthonormal
projection from $L^2(d\mu)$ to $A^2(\cc d)$.

\par

From now on we assume that $\phi$ in short-time
Fourier transforms \eqref{Eq:STFT} is given by
\begin{equation}\label{phidef}
\phi (x)=\pi ^{-\frac d4}e^{-|x|^2/2}, \quad x\in \rr d,
\end{equation}
if nothing else is stated. For such $\phi$, it follows by straight-forward
computations that the relationship between
the Bargmann transform and the short-time Fourier transform
is given by
\begin{equation}\label{bargstft1}
\mathfrak V_d = U_{\mathfrak V}\circ V_\phi ,\quad \text{and}\quad
U_{\mathfrak V}^{-1} \circ \mathfrak V_d =  V_\phi ,
\end{equation}
where $U_{\mathfrak V}$ is the linear, continuous and bijective operator from
$\mathscr D'(\rr {2d})$ to $\mathscr D'(\cc d)$, given by
\begin{alignat}{2}
(U_{\mathfrak V}F)(x+i\xi ) &= (2\pi )^{\frac d2} e^{\frac 12(|x|^2+|\xi |^2)}e^{-i\scal x\xi}
F(\sqrt 2\, x,-\sqrt 2\, \xi ),& \quad x,\xi &\in \rr d.
\label{UVdef}
\intertext{We notice that the inverse of $U_{\mathfrak V}$ is given by}
(U_{\mathfrak V}^{-1}F)(x,\xi ) &= (2\pi )^{-\frac d2} e^{-\frac 14(|x|^2+|\xi |^2)}
e^{-\frac i2\scal x\xi} F\left (\frac {x-i\xi }{\sqrt 2} \right ), & \quad x,\xi &\in \rr d
\label{Eq:UVdefInv}
\end{alignat}
(cf. \cite{To11}).

\medspace

If $W=\cc {d_2}\times \cc {d_1}$, then $\wideparen A^2(W)$ consists
of all $K\in C(W)$ such that $(z_2,z_1)\mapsto K(z_2,\overline z_1)$
belongs to $A^2(W)$. Then $\wideparen A^2(W)$ is a Hilbert
space under the scalar product
\begin{gather}
(K_1,K_2)_{\wideparen A^2} =(K_1,K_2)_{\wideparen A^2(W)}
\equiv
(K_{0,1},K_{0,2})_{A^2(W)},
\label{Eq:WideparA2Def1}
\intertext{when}
K_j\in \wideparen A^2(W),\quad
K_{0,j}(z_2,z_1)= K_j(z_2,\overline z_1),\quad j=1,2.
\label{Eq:WideparA2Def2}
\end{gather}

\par

\subsection{Wick and anti-Wick operators}\label{subsec1.6}

\par

Next we recall the definition of analytic
pseudo-differential operators, so-called Wick operators. Suppose that
$a(z,w)\in \wideparen A(\cc {2d})$
and $F\in A(\cc d)$ satisfy 
\begin{equation}\label{Eq:AntiWickL1Cond}
w\mapsto a(z,w)F(w)e^{r|w|-|w|^2} \in L^1(\cc d)
\end{equation}
for every $r>0$ and $z\in \cc d$. Then the \emph{analytic
pseudo-differential operator}, or \emph{Wick operator}
$\op _{\mathfrak V}(a)$ with symbol $a$, is the linear and
continuous operator from $\maclA _0(\cc d)$ to $A(\cc d)$,
defined by the formula
\begin{align}
\op _{\mathfrak V}(a)F(z) &= \int _{\cc d} a(z,w)F(w)e^{(z,w)}\, d\mu (w),
\quad z \in \cc d,
\label{Eq:AnalPseudo}
\end{align}
when $F\in \maclA _0(\cc d)$.
(Cf. e.{\,}g. \cite{Berezin71,Fo,Teofanov2,To11,Toft18}.) Here we remark
that $\op _{\mathfrak V}(a)F$ is extendable in several ways, allowing
$a$ and $F$ to belong to different spaces
where \eqref{Eq:AntiWickL1Cond} is violated. For example, in
\cite{Teofanov2} it is proved that the definition of $\op _{\mathfrak V}(a)$
is uniquely extendable to any $a\in \wideparen \maclA _0'(\cc {2d})$,
and then $\op _{\mathfrak V}(a)$ is continuous from $\maclA _0(\cc d)$
to $\maclA _0'(\cc d)$.

\par

The definition of the Wick operator in \eqref{Eq:AnalPseudo} resembles
on the definition of the classical Kohn-Nirenberg pseudo-differential operators
on $\rr d$. Let $A$ be a real $d\times d$ matrix.
Then the \emph{pseudo-differential operator}
$\op _A(\fka )$ with \emph{symbol}
$\fka \in \maclS _{1/2} (\rr {2d})$ is the linear and continuous operator on
$\maclS _{1/2} (\rr d)$, given by
$$
(\op _A(\fka )f)(x)
=
(2\pi  ) ^{-d}\iint \fka (x-A(x-y),\xi )
f(y)e^{i\scal {x-y}\xi }\,
dyd\xi, \quad x\in \rr d.
$$
The definition of $\op _A(\fka )$ extends to any $\fka \in \maclS _{1/2}'(\rr {2d})$,
and then $\op _A(\fka )$ is continuous from $\maclS _{1/2}(\rr d)$ to $\maclS _{1/2}'(\rr d)$.
(See e.{\,}g. \cite{Ho1,Toft17} and the references therein.)

\par

The normal (Kohn-Nirenberg) representation and the Weyl quantization
are obtained by choosing $A=0$ and $A=\frac 12I$ respectively,
where $I=I_d$ is the $d\times d$ identity matrix.

\par

In the literature it is also
common to consider anti-Wick operators. Suppose that
$a\in \wideparen A(\cc {2d})$ satisfies
\begin{equation}\label{Eq:AntiWickL1Cond2}
w\mapsto a(w,w)e^{r|w|-|w|^2} \in L^1(\cc d)
\end{equation}
for every $r>0$. Then the \emph{anti-Wick operator}
$\op _{\mathfrak V}^{\aw}(a)$ with symbol $a$ is the linear and continuous
operator from $\maclA _0(\cc d)$ to $A(\cc d)$, given by
\begin{align}
\op _{\mathfrak V}^{\aw}(a)F(z) &= \int _{\cc d} a(w,w)F(w)e^{(z,w)}\, d\mu (w),
\quad z \in \cc d,
\label{Eq:AntiWick}
\end{align}
when $F\in \maclA _0(\cc d)$.

\par

Berezin established in \cite{Berezin71}
convenient links between Wick and anti-Wick operators
(see page 587 in \cite{Berezin71}).
More precisely, if $a\in \wideparen \maclA _0(\cc {2d})$, then
\begin{align}
\op _{\mathfrak V}^{\aw}(a) &= \op _{\mathfrak V}(a^{\aw})
\label{Eq:AntiWickAnalPseudoRel}
\intertext{when $a^{\aw}$ is given by}
a^{\aw} (z,w) &= \pi ^{-d} \int _{\cc d}a(w_1,w_1)e^{-(z-w_1,w-w_1)}
\, d\lambda (w_1).
\label{Eq:AntiWickAnalPseudoRelIntForm}
\intertext{If instead $a^{\aw}\in \wideparen \maclA _0(\cc {2d})$,
then \eqref{Eq:AntiWickAnalPseudoRel} holds true when
$a^{\aw}\in \wideparen A(\cc {2d})$ is given by}
a(z,-w) &= \pi ^{-d} \int _{\cc d}a^{\aw}(w_1,-w_1)e^{-(z-w_1,w-w_1)}\, d\lambda (w_1).
\label{Eq:AntiWickAnalPseudoRelIntForm2}
\end{align}
In fact, if $F\in \maclA _0(\cc d)$ and $a^{\aw}$ is given by 
\eqref{Eq:AntiWickAnalPseudoRelIntForm}, then
$$
\op _{\mathfrak V}(a^{\aw})F(z)
=
\pi ^{-d}\int _{\cc {d}} a(w_1,w_1)e^{(z,w_1)-|w_1|^2}
\left (
\int _{\cc d}F(w)e^{(w_1,w)}\, d\mu (w)
\right )
\, d\lambda (w_1),
$$
and by applying the reproducing formula \eqref{reproducing}
on the inner integral, it follows that the right-hand side becomes
$\op _{\mathfrak V}^{\aw}(a)F(z)$ in \eqref{Eq:AntiWick}.
This shows that \eqref{Eq:AntiWickAnalPseudoRel} holds true
when $a^{\aw}$ is given by \eqref{Eq:AntiWickAnalPseudoRelIntForm}.
In Section \ref{sec3} and \ref{sec4} we show the equivalence between
\eqref{Eq:AntiWickAnalPseudoRelIntForm}
and \eqref{Eq:AntiWickAnalPseudoRelIntForm2}, when $a$
or $a^{\aw}$ are allowed to belong to larger classes than $\wideparen
\maclA _0(\cc {2d})$.

\par

Since any element $a(z,w)$
in $\wideparen A(\cc {2d})$ is uniquely determined on its
values at the diagonal $z=w$ or anti-diagonal $z=-w$, it
follows that \eqref{Eq:AntiWickAnalPseudoRelIntForm} and
\eqref{Eq:AntiWickAnalPseudoRelIntForm2} are equivalent to
\begin{align}
a^{\aw} (z,z) &= \pi ^{-d} \int _{\cc d}a(w_1,w_1)e^{-|z-w_1|^2}
\, d\lambda (w_1)
\tag*{(\ref{Eq:AntiWickAnalPseudoRelIntForm})$'$}
\intertext{and}
a (z,-z) &= \pi ^{-d} \int _{\cc d}a^{\aw}(w_1,-w_1)e^{-|z-w_1|^2}
\, d\lambda (w_1).
\tag*{(\ref{Eq:AntiWickAnalPseudoRelIntForm2})$'$}
\end{align}
These formulae were especially emphasized on page
587 in \cite{Berezin71}.

\medspace

Let
$a_1,a_2 \in \wideparen \maclA _0'(\cc {2d})$.
For suitable additional conditions on $a_1$ and $a_2$,
it is possible to compose the
Wick operators $\op _{\mathfrak V}(a_1)$ and
$\op _{\mathfrak V}(a_2)$, and
the resulting operator is again a Wick operator.
The (complex) twisted product $a_1 \wpr _{\mathfrak V}a_2$ is
then defined by
\begin{equation}\label{Eq:DefCompTwistProd0}
\op _{\mathfrak V}(a_1)\circ \op _{\mathfrak V}(a_2) =
\op _{\mathfrak V}(a_1 \wpr _{\mathfrak V}a_2),
\end{equation}
provided the composition on the left-hand side is
well-defined as a continuous operator from $\maclA _0(\cc d)$
to $\maclA _0'(\cc d)$.
By straight-forward computations it follows that
the product $\wpr _{\mathfrak V}$ is given by
\begin{equation}\label{Eq:DefCompTwistProd}
a_1 \wpr _{\mathfrak V} a_2 (z,w)
= \pi ^{-d}\int_{\cc d} a_1(z,u) a_2 (u,w) e^{-(z-u,w-u)}\, d \lambda (u),
\quad z,w \in \cc d,
\end{equation}
when the integrand belongs to $L^1(\cc d)$ for every $z,w \in \cc d$
(see e.{\,}g. \cite{Berezin71,TeToWa}).

\par

\subsection{Toeplitz operators}\label{subsec1.7}

\par

Let $\phi _1,\phi _2\in \maclS _{1/2}(\rr d)\setminus 0$.
Then the Toeplitz operator $\tp _{\phi _1,\phi _2}(\fka )$ with symbol
$\fka \in \maclS _{1/2}'(\rr {2d})$ is the linear and continuous operator
from $\maclS _{1/2}(\rr d)$ to $\maclS _{1/2}'(\rr d)$, given by the formula
\begin{equation}\label{Eq:ToeplitzDef}
(\tp _{\phi _1,\phi _2}(\fka )f,g)_{L^2(\rr d)} =
(\fka \cdot V_{\phi _1} f,V_{\phi _2} g)_{L^2(\rr {2d})}.
\end{equation}
It is proved in \cite{Toft18} that $\tp _{\phi _1,\phi _2}(\fka )$ is continuous on
$\maclS _{1/2}(\rr d)$
and uniquely extendable to a continuous operator on $\maclS _{1/2}'(\rr d)$.

\par

In our situation we have $\phi _1=\phi _2$ is equal to $\phi$ in \eqref{phidef}, 
and for conveniency we put $\tp (\fka )=\tp _{\phi ,\phi}(\fka )$.
In fact, for suitable $a\in \wideparen A(\cc {2d})$
we shall mainly consider modified Toeplitz operators $\tp _{\mathfrak V}(a)$
given by
\begin{equation}\tag*{(\ref{Eq:ToeplitzDef})$'$}
\tp _{\mathfrak V}(a)=\tp (\fka )
\quad \text{when}\quad
a(z,z) = \fka (\sqrt 2\, x,-\sqrt 2\, \xi ),\ z=x+i\xi .
\end{equation}
By straight-forward applications of \eqref{UVdef} and \eqref{Eq:UVdefInv}
it follows
\begin{equation}\label{Eq:ToeplAntiWick}
\op _{\mathfrak V}^{\aw}(a)
=
\mathfrak V_d\circ \tp _{\mathfrak V}(a)\circ \mathfrak V_d^{-1}.
\end{equation}
(See also Section 6 in \cite{To11}.)

\par

We recall that Toeplitz operators can be formulated as pseudo-differential
operators, due to the formula
\begin{equation}\label{Eq:ToeplPseudo}
\tp (\fka) = \left ( \frac 2\pi \right )^{\frac d2} \op ^w(\fka *e^{-|\cdo |^2}),
\end{equation}
which is equivalent to \eqref{Eq:AntiWickAnalPseudoRelIntForm}
and \eqref{Eq:AntiWickAnalPseudoRelIntForm2}.
(See e.{\,}g. \cite{Shubin1,Fo,To11} and the references therein.)

\par

\section{Operator kernels and multiplications for
spaces of power series expansions}\label{sec2}

\par

In this section we deduce kernel theorems for linear operators acting
between (different) $\maclA _s$ spaces and their duals. In \cite{Teofanov2}
such kernel results were explained for linear operators which
map $\maclA _s(\cc {d_1})$ into $\maclA _s'(\cc {d_2})$ or
which map $\maclA _{0,s}(\cc {d_1})$ into $\maclA _{0,s}'(\cc {d_2})$. 
The latter results can also be obtained by classical kernel results for
linear operators acting on nuclear spaces (see e.{\,}g. \cite{Tre}).

\par

\subsection{Kernels to linear operators acting on spaces of
power series expansions}

\par

The continuity results for kernel operators rely on duality properties
between the spaces through the $\wideparen A^2$ form in
\eqref{Eq:WideparA2Def1} and \eqref{Eq:WideparA2Def2}. In fact, we
have the following proposition which describes important topological
properties of the spaces in Definitions
\ref{Def:tauSpaces} and \ref{Def:tauSpaces2},
and which is a straight-forward
consequence of Remarks \ref{Rem:elltauSpaces2}
and \ref{Rem:FeaturesSeqSpaces}, and the definitions.
The details are left for the reader.

\par

\begin{prop}\label{Prop:DualtauSpaces}
Let $s_0,s_1,s_2\in \overline{\mathbf R_\flat}$, $s=(s_2,s_1)$ and
$W =\cc {d_2}\times \cc {d_1}$. Then the following is true:
\begin{enumerate}
\item the spaces in
\eqref{Eq:SesAnalSp2}
are complete Hausdorff topological vector spaces,
and $\wideparen \maclA _{0,s}(W)$ and
$\wideparen \maclA ^\star _{s}(W)$ are Fr{\'e}chet spaces. It holds
$\wideparen \maclA _{(s_0,s_0)} = \wideparen \maclA _{s_0}$
and
$\wideparen \maclA ^\star _{(s_0,s_0)} = \wideparen \maclA ^\star _{s_0}$;

\vrum

\item $\wideparen \maclA _0(W)$ is dense in the spaces in
\eqref{Eq:SesAnalSp2}, \eqref{Eq:BparenSpaces}
and \eqref{Eq:CparenSpaces};

\vrum

\item the form
$(\cdo ,\cdo )_{\wideparen A^2(W)}$
on $\wideparen \maclA _0(W)$
is uniquely extendable to separate continuous non-degenerate
mappings from
\begin{equation}\label{Eq:DualPairs}
\begin{aligned}
\wideparen \maclA _{s}^\star (W)\times
\wideparen \maclA _{s}(W),
\qquad
&\wideparen \maclC _{s}^\star (W)\times
\wideparen \maclB _{s}(W)
\quad \text{and} \quad
\wideparen \maclB _{s}^\star (W)\times
\wideparen \maclC _{s}(W)
\end{aligned}
\end{equation}
to $\mathbf C$. 
Furthermore,
$\wideparen \maclA _{s}^\star (W)$ and
$\wideparen \maclA _{s}(W)$ are (strong)
duals to each others, through the form
$(\cdo ,\cdo )_{\wideparen A^2(W)}$.
\end{enumerate}

\par

If instead $s_1,s_2\in \mathbf R_{\flat ,\infty}$, then
the same holds true with $\wideparen \maclA _{0,s}$,
$\wideparen \maclB _{0,s}$ and
$\wideparen \maclC _{0,s}$ in place of $\wideparen \maclA _{s}$,
$\wideparen \maclB _{s}$
and $\wideparen \maclC _{s}$, respectively, at each occurrence.
\end{prop}

\par

Evidently, it follows from (3) in Proposition \ref{Prop:DualtauSpaces}
that
\begin{equation}
\begin{alignedat}{2}
\wideparen \maclA _{(s_2,s_1)}'(W)
&=
\wideparen \maclA _{(s_2,s_1)}^\star (W), &
\quad
(\wideparen \maclA _{(s_2,s_1)}^\star )'(W)
&=
\wideparen \maclA _{(s_2,s_1)}(W),
\end{alignedat}
\end{equation}
through the form $(\cdo ,\cdo )_{\wideparen A^2(W)}$.

\par

The following kernel results characterize linear operators which
map (different) $\maclA _s$ or $\maclA _{0,s}$ spaces and their duals
into each others. Here and in what follows we let $\maclL (V_1;V_2)$
be the set of linear and continuous mappings from the topological vector
space $V_1$ into the topological vector space $V_2$. We also set
$\maclL (V)=\maclL (V;V)$.

\par

\begin{prop}\label{Prop:KernelsDifficultDirectionExt}
Let $s_1,s_2\in \overline{\mathbf R_{\flat ,\infty}}$, $s=(s_2,s_1)$,
$W=\cc {d_2}\times \cc {d_1}$, and let
$T\in \maclL (\maclA _0(\cc {d_1});\maclA _0'(\cc {d_2}))$.
Then the following is true:
\begin{enumerate}
\item if $T\in \maclL (\maclA _{s_1}'(\cc {d_1});\maclA _{s_2}(\cc {d_2}))$,
then there is a unique
$K\in \wideparen \maclA _{s}(W)$ such that
\begin{equation}\label{Kmap2}
(TF)(z_2) = (F,\, \overline{K(z_2,\cdo)}\, )_{A^2(\cc {d_1})},
\qquad z_2\in \cc {d_2},
\end{equation}
holds true for every $F\in \maclA _{s_1}(\cc {d_1})$;

\vrum

\item if $T\in \maclL (\maclA _{s_1}(\cc {d_1});\maclA _{s_2}'(\cc {d_2}))$,
then there is a unique
$K\in \wideparen \maclA _{s}'(W)$ such that
\eqref{Kmap2} holds true for every $F\in \maclA _{s_1}'(\cc {d_1})$.
\end{enumerate}

\par

The same holds true with $\maclA _{0,s_j}$ and
$\wideparen \maclA _{0,s}$ in place of
$\maclA _{s_j}$ and $\wideparen \maclA _{s}$,
respectively, $j=1,2$, at each occurrence.
\end{prop}

\par

Here \eqref{Kmap2} should be interpreted
as
\begin{equation}\tag*{(\ref{Kmap2})$'$}
(TF,G) _{A^2(\cc {d_2})} =  (K ,G\otimes \overline F\, )
_{\wideparen A^2(\cc {d_2}\times \cc {d_1})},
\end{equation}
when $F\in \maclA _0(\cc {d_1})$ and $G\in \maclA _0(\cc {d_2})$.

\par

\begin{prop}\label{Prop:KernelsEasyDirectionExt}
Let $s_1,s_2\in \overline{\mathbf R_{\flat ,\infty}}$, $s=(s_2,s_1)$,
$W=\cc {d_2}\times \cc {d_1}$, $K\in \wideparen \maclA _0'(W)$,
and let
$T=T_K\in \maclL (\maclA _0(\cc {d_1});\maclA _0'(\cc {d_2}))$
be given by
\begin{equation}\label{Eq:Kmap}
F\mapsto T_KF = \big ( z_2\mapsto (F,\overline {K(z_2,\cdo)})
_{A^2(\cc {d_1})} \big ).
\end{equation}
Then the following is true:
\begin{enumerate}
\item if $K\in \wideparen \maclA _{s}(W)$, then $T_K$
extends uniquely to a linear and continuous map from
$\maclA _{s_1}'(\cc {d_1})$ to $\maclA _{s_2}(\cc {d_2})$;

\vrum

\item if $K\in \wideparen \maclA _{s}'(W)$, then
$T_K$ extends uniquely
to a linear and continuous map from $\maclA _{s_1}(\cc {d_1})$ to
$\maclA _{s_2}'(\cc {d_2})$.
\end{enumerate}

\par

The same holds true with $\maclA _{0,s_j}$ and
$\wideparen \maclA _{0,s}$ in place of
$\maclA _{s_j}$ and $\wideparen \maclA _{s}$,
respectively, $j=1,2$, at each occurrence.
\end{prop}

\par

We observe that if $K$ and $F$ are the same as in
\eqref{Eq:Kmap} and $z_1\mapsto K(z_2,z_1)F(z_1)$
is integrable with respect to $d\mu (z_1)$, then \eqref{Eq:Kmap}
is the same as
\begin{equation}\tag*{(\ref{Eq:Kmap})$'$}
(T_KF)(z_2) = \int _{\cc {d_1}}K(z_2,z_1)F(z_1)\, d\mu (z_1).
\end{equation}

\par

Propositions \ref{Prop:KernelsDifficultDirectionExt} and
\ref{Prop:KernelsEasyDirectionExt} follow essentially from
abstract kernel results for linear operators acting between
topological vector spaces.(See e.{\,}g. \cite{Tre}. See also
\cite{Teofanov2} for more explicit approaches.) On the other hand,
the following result might be more cumbersome to
deduce from such abstract kernel results.

\par

\begin{prop}\label{Prop:KernelsDifficultDirectionNew}
Let $s_1,s_2\in \overline{\mathbf R_{\flat ,\infty}}$, $s=(s_2,s_1)$,
$W=\cc {d_2}\times \cc {d_1}$, and let
$T\in \maclL (\maclA _0(\cc {d_1});\maclA _0'(\cc {d_2}))$.
Then the following is true:
\begin{enumerate}
\item if $T\in \maclL (\maclA _{s_1}(\cc {d_1});\maclA _{s_2}(\cc {d_2}))$,
then there is a unique
$K\in \wideparen \maclB _{s}(W)$ such that
\eqref{Kmap2} holds true for every $F\in \maclA _{s_1}(\cc {d_1})$;

\vrum

\item if $T\in \maclL (\maclA _{s_1}'(\cc {d_1});\maclA _{s_2}'(\cc {d_2}))$,
then there is a unique
$K\in \wideparen \maclB _{s}^\star (W)$ such that
\eqref{Kmap2} holds true for every $F\in \maclA _{s_1}'(\cc {d_1})$.
\end{enumerate}

\par

The same holds true with $\maclA _{0,s_j}$ and
$\wideparen \maclB _{0,s}$ in place of
$\maclA _{s_j}$ and $\wideparen \maclB _{s}$,
respectively, $j=1,2$, at each occurrence.
\end{prop}

\par

We also have the following converse of the preceding proposition.

\par

\begin{prop}\label{Prop:KernelsEasyDirectionNew}
Let $s_1,s_2\in \overline{\mathbf R_{\flat ,\infty}}$, $s=(s_2,s_1)$,
$W=\cc {d_2}\times \cc {d_1}$, $K\in \wideparen \maclA _0'(W)$,
and let
$T=T_K\in \maclL (\maclA _0(\cc {d_1});\maclA _0'(\cc {d_2}))$
be given by \eqref{Eq:Kmap}. Then the following is true:
\begin{enumerate}
\item if $K\in \wideparen \maclB _{s}(W)$, then $T_K$
extends uniquely to a linear and continuous map from
$\maclA _{s_1}(\cc {d_1})$ to $\maclA _{s_2}(\cc {d_2})$;

\vrum

\item if $K\in \wideparen \maclB _{s}^\star (W)$,
then $T_K$ extends uniquely to a linear and continuous map from
$\maclA _{s_1}'(\cc {d_1})$ to $\maclA _{s_2}'(\cc {d_2})$.
\end{enumerate}

\par

The same holds true with $\maclA _{0,s_j}$ and
$\wideparen \maclB _{0,s}$ in place of
$\maclA _{s_j}$ and $\wideparen \maclB _{s}$,
respectively, $j=1,2$, at each occurrence.
\end{prop}

\par

By combining Propositions
\ref{Prop:KernelsDifficultDirectionExt}--\ref{Prop:KernelsEasyDirectionNew},
we get the following.

\par

\begin{cor}\label{Cor:KernelLinOpChar}
Let $s_1,s_2 \in \overline{\mathbf R_\flat}$, $s=(s_1,s_2)$,
$W=\cc {d_2}\times \cc {d_1}$ and let $T_K$ be the map in
\eqref{Eq:Kmap} when $K\in \wideparen \maclA _0'(W)$. Then the mappings
\begin{alignat*}{3}
\wideparen \maclA _s(W) &\ni K &
\ &\mapsto \ &T_K &\in \maclL (\maclA _{s_1}'(\cc {d_1}),\maclA _{s_2}(\cc {d_2})),
\\[1ex]
\wideparen \maclA _s'(W) &\ni K &
\ &\mapsto \ &T_K &\in \maclL (\maclA _{s_2}(\cc {d_1}),\maclA _{s_2}'(\cc {d_2})),
\\[1ex]
\wideparen \maclB _s(W) &\ni K &
\ &\mapsto \ &T_K &\in \maclL (\maclA _{s_1}(\cc {d_1}),\maclA _{s_2}(\cc {d_2}))
\intertext{and}
\wideparen \maclB _s^\star (W) &\ni K &
\ &\mapsto \ &T_K &\in \maclL (\maclA _{s_1}'(\cc {d_1}),\maclA _{s_2}'(\cc {d_2}))
\end{alignat*}
are isomorphisms. The same holds true with
$\wideparen \maclA _{0,s}$, $\wideparen \maclB _{0,s}$,
$\maclA _{0,s_1}$ and $\maclA _{0,s_2}$ in place of
$\wideparen \maclA _{s}$, $\wideparen \maclB _{s}$,
$\maclA _{s_1}$ and $\maclA _{s_2}$, respectively, at each occurrence.
\end{cor}

\par

Proposition \ref{Prop:KernelsEasyDirectionNew} follows by straight-forward
computations and is left for the reader.

\par

\begin{proof}[Proof of Proposition \ref{Prop:KernelsDifficultDirectionNew}]
We only consider the case when $T$ is continuous from $\maclA _{s_1}(\cc {d_1})$
to $\maclA _{s_2}(\cc {d_2})$. The other cases follow by similar arguments and are
left for the reader.
Let $\vartheta _{r,s}$ be as in Definition \ref{DefSeqSpaces}.
By Propositions \ref{Prop:KernelsDifficultDirectionExt}
and \ref{Prop:KernelsEasyDirectionExt},
there is a unique
\begin{equation}\label{Eq:AnalKernelExp}
K(z_2,z_1)=\sum _{\alpha _j\in \nn {d_j}}
c(K;\alpha _2,\alpha _1)e_{\alpha _2}(z_2)e_{\alpha _1}(\overline {z_1})
\in \wideparen \maclA _0'(W),
\qquad
z_j\in \cc {d_j},
\end{equation}
such that \eqref{Kmap2} holds true for every $F\in \maclA _0(\cc {d_1})$. We need
to show that $K\in \wideparen \maclB _{s}(W)$.

\par

Let $r>0$ and set $F_{\alpha}(z_1) = \vartheta _{r,s_1}^{-1}(\alpha )e_\alpha (z_1)$
for every $\alpha \in \nn {d_1}$ and $z_1\in \zz {d_1}$. Then
$\{ F_\alpha \} _{\alpha \in \nn {d_1}}$ is a bounded subset of
$\maclA _{s_1}(\cc {d_1})$. By the continuity of $T$ it follows that
$\{ TF_\alpha \} _{\alpha \in \nn {d_1}}$ is a bounded subset of
$\maclA _{s_2}(\cc {d_2})$. Hence, if
$$
(TF_\alpha )(z_2) = \sum _{\alpha _2\in \nn {d_2}}c(TF_\alpha ;\alpha _2)
e_{\alpha _2}(z_2), \qquad z_2\in \cc {d_2},
$$
then
$$
c(TF_\alpha ;\alpha _2) = c(K;\alpha _2,\alpha )\vartheta _{r,s_1}^{-1}(\alpha )
$$
satisfies
$$
|c(TF_\alpha ;\alpha _2)|\le C\vartheta _{r_0,s_2}^{-1}(\alpha _2),
$$
for some constants $C,r_0>0$ which are independent of $\alpha _2$ and $\alpha$.
This gives,
$$
|c(K;\alpha _2,\alpha _1)|
\lesssim
\vartheta _{r,s_1}(\alpha _1)\vartheta _{r_0,s_2}^{-1}(\alpha _2),
\qquad \alpha _j\in \cc {d_j}.
$$
Since $r>0$ is arbitrary, the latter estimate implies that
$K\in \wideparen \maclB _{s}(\cc {d_2}\times \cc {d_1})$.
\end{proof}

\par

In Section \ref{sec4} we deduce related results compared to
Propositions
\ref{Prop:KernelsDifficultDirectionExt}--\ref{Prop:KernelsEasyDirectionNew}
and Corollary \ref{Cor:KernelLinOpChar} with Wick and anti-Wick operators
in place of kernel operators.

\par

\begin{example}\label{Example:ShubinWick}
Let $\omega \in L^\infty _{\loc}(\cc d)\simeq L^\infty _{\loc}(\rr {2d})$ be such that
\begin{equation}\label{Eq:Moderate}
\omega (z)>0
\quad \text{and}\quad
\omega (z+w)\lesssim \omega (z)\eabs w^r,\qquad
z,w\in \cc d,
\end{equation}
for some $r>0$. Then we recall that
the Shubin class $\Sh ^{(\omega )}_\rho (\rr {2d})$ (with respect to $\rho$
and $\omega$) is the set of all
$\fka \in C^\infty (\rr {2d})$ such that
$$
|\partial ^\alpha \fka (x,\xi )|\lesssim \omega (x+i\xi )\eabs {(x,\xi )}^{-\rho |\alpha |}.
$$
(See \cite{Shubin1}.)

\par

As in \cite{TeToWa}, we let
$\wideparen \maclA _{\Sh ,\rho} ^{(\omega)}(\cc {2d})$ be the class of Wick
symbols which consists of all $a\in \wideparen A(\cc {2d})$ such that
\begin{equation}\label{Eq:WickShubinDef}
|\partial _z^\alpha \overline \partial _w^\beta a(z,w)|
\lesssim
e^{\frac 12|z-w|^2}\omega (\sqrt 2\overline z)\eabs {z+w}^{-\rho |\alpha +\beta |}
\eabs {z-w}^{-N}
\end{equation}
for every $\alpha ,\beta \in \nn d$ and $N\ge 0$. Let $A$ be areal $d\times d$ matrix.
Then recall that for $\fka \in \mascS '(\rr {2d})$, we have
$\fka \in \Sh ^{(\omega )}_\rho (\rr {2d})$, if and only if
$\op _{\mathfrak V}(a)=\mathfrak V_d \circ \op _A(\fka )\circ \mathfrak V_d^{-1}$
for some $a\in \wideparen \maclA _{\Sh ,\rho} ^{(\omega)}(\cc {2d})$
(cf. \cite{TeToWa}).

\par

Let $a\in \wideparen \maclA _{\Sh ,\rho} ^{(\omega)}(\cc {2d})$
and let $K$ be the kernel of $\op _{\mathfrak V}(a)$. We claim that
\begin{equation}\label{Eq:ShubinFine}
K\in \wideparen \maclB _{0,\infty}(\cc {2d})\bigcap \wideparen \maclB
_{0,\infty}^\star (\cc {2d}).
\end{equation}

\par

In fact, by \eqref{Eq:WickShubinDef} and that $K(z,w)=e^{(z,w)}a(z,w)$
we get
\begin{equation}\label{Eq:WickShubinKernelEst}
|K(z,w)| \lesssim e^{\frac 12(|z|^2+|w|^2)}\omega (\sqrt 2\overline z)
\eabs {z-w}^{-N}
\quad \text{for every}\ N\ge 0.
\end{equation}
Since $\omega$ satisfies \eqref{Eq:Moderate} we get
$$
\omega (\sqrt 2\overline z) \lesssim \eabs z^{N_0}
\quad \text{and}\quad
\omega (\sqrt 2\overline z) \lesssim \omega (\sqrt 2\overline w)\eabs {z-w}^{N_0}
\lesssim \eabs w^{N_0}\eabs {z-w}^{N_0},
$$
for some $N_0\ge 0$. We also have
$$
\eabs {z-w}^{-N} \lesssim \eabs z^{-N}\eabs w^N
\quad \text{and}\quad
\eabs {z-w}^{-N} \lesssim \eabs z^{N}\eabs w^{-N}.
$$
By playing with $N\ge 0$, a combination of these estimates with
\eqref{Eq:WickShubinKernelEst} shows that
\begin{align*}
|K(z,w)| &\lesssim e^{\frac 12(|z|^2+|w|^2)}
\eabs z^{N_0+N}\eabs w^{-N}
\intertext{and}
|K(z,w)| &\lesssim e^{\frac 12(|z|^2+|w|^2)}
\eabs z^{-N}\eabs w^{N_0+N},
\end{align*}
which implies that \eqref{Eq:ShubinFine} holds.

\par

By Proposition \ref{Prop:KernelsDifficultDirectionNew} it now follows that
$\op _{\mathfrak V}(a)$ is continuous on both $\maclA _{0,\infty}(\cc d)$
and on $\maclA _{0,\infty}'(\cc d)$. This shows that any Shubin
operator is continuous on $\mascS (\rr d)$ and on $\mascS '(\rr d)$
(see e.{\,}g. \cite[Theorem 18.6.2]{Ho1}).
\end{example}

\par

\subsection{Compositions of analytic Kernel operators}.

\par

In what follows we identify any linear and continuous operator $T$
from $\maclA _0(\cc {d_1})$ to $\maclA _0'(\cc {d_2})$ by its kernel
$K\in \wideparen \maclA _0'(\cc {d_2}\times \cc {d_1})$ in \eqref{Kmap2}
and \eqref{Kmap2}$'$. By general continuity properties we get
the following theorem concerning the composition map
\begin{equation}\label{Eq:CompMarKernel}
(K_2,K_1) \mapsto  K_2\circ K_1.
\end{equation}

\par

\begin{thm}\label{Thm:KernelComp1}
Let $W_j=\cc {d_{j+1}}\times \cc {d_j}$, $j=1,2$, $W_3=\cc {d_3}\times \cc {d_1}$,
$s_1,s_2,s_3\in \overline {\mathbf R_{\flat}}$ and let $T$ be the map
from $\wideparen \maclA _0(W_2)\times \wideparen \maclA _0(W_1)$ to
$\wideparen \maclA _0(W_3)$, given by \eqref{Eq:CompMarKernel}.
Then the following is true:
\begin{enumerate}
\item $T$ is uniquely extendable to continuous mappings
from $\wideparen \maclA _{(s_3,s_2)}(W_2)\times \wideparen
\maclB _{(s_2,s_1)}^\star (W_1)$ or from
$\wideparen \maclB _{(s_3,s_2)}(W_2)\times \wideparen
\maclA _{(s_2,s_1)}(W_1)$ to $\wideparen \maclA _{(s_3,s_1)}(W_3)$;

\vrum

\item $T$ is uniquely extendable to continuous mappings
from $\wideparen \maclB _{(s_3,s_2)}^\star (W_2)\times \wideparen
\maclA _{(s_2,s_1)}'(W_1)$ or from $\wideparen \maclA _{(s_3,s_2)}'(W_2)\times \wideparen
\maclB _{(s_2,s_1)}(W_1)$ to $\wideparen \maclA _{(s_3,s_1)}'(W_3)$;
\end{enumerate}
If instead $s_1,s_2,s_3\in \mathbf R_{\flat ,\infty}$, then the same hold
true with $\wideparen \maclA _{0,(s_j,s_k)}$ and $\wideparen \maclB _{0,(s_j,s_k)}$
in place of $\wideparen \maclA _{(s_j,s_k)}$ and $\wideparen \maclB _{(s_j,s_k)}$,
respectively, at each occurrence.
\end{thm}

\par

\begin{thm}\label{Thm:KernelComp2}
Let $W_j=\cc {d_{j+1}}\times \cc {d_j}$, $j=1,2$, $W_3=\cc {d_3}\times \cc {d_1}$,
$s_1,s_2,s_3\in \overline {\mathbf R_{\flat}}$ and let $T$ be the map
from $\wideparen \maclA _0(W_2)\times \wideparen \maclA _0(W_1)$ to
$\wideparen \maclA _0(W_3)$, given by \eqref{Eq:CompMarKernel}.
Then the following is true:
\begin{enumerate}
\item $T$ is uniquely extendable to continuous mappings
from $\wideparen \maclB _{(s_3,s_2)}(W_2)\times \wideparen
\maclC _{(s_2,s_1)}(W_1)$ or from
$\wideparen \maclC _{(s_3,s_2)}(W_2)\times \wideparen
\maclB _{(s_2,s_1)}(W_1)$ to
$\wideparen \maclC _{(s_3,s_1)}(W_3)$.  The same
holds true with $\wideparen \maclB _{(s_j,s_k)}$ in place of
$\wideparen \maclC _{(s_j,s_k)}$ at each occurrence;

\vrum

\item $T$ is uniquely extendable to continuous mappings
from $\wideparen \maclB _{(s_3,s_2)}^\star (W_2)\times \wideparen
\maclC _{(s_2,s_1)}^\star (W_1)$ or from
$\wideparen \maclC _{(s_3,s_2)}^\star (W_2)\times \wideparen
\maclB _{(s_2,s_1)}^\star (W_1)$ to
$\wideparen \maclC _{(s_3,s_1)}^\star (W_3)$. The same
holds true with $\wideparen \maclB _{(s_j,s_k)}$ in place of
$\wideparen \maclC _{(s_j,s_k)}$ at each occurrence.
\end{enumerate}
If instead $s_1,s_2,s_3\in \mathbf R_{\flat ,\infty}$, then the same hold
true with $\wideparen \maclA _{0,(s_j,s_k)}$, $\wideparen \maclB _{0,(s_j,s_k)}$
and $\wideparen \maclC _{0,(s_j,s_k)}$ in place of
$\wideparen \maclA _{(s_j,s_k)}$, $\wideparen \maclB _{(s_j,s_k)}$
and $\wideparen \maclC _{(s_j,s_k)}$, respectively, at each occurrence.
\end{thm}

\par

\begin{proof}[Proof of Theorems \ref{Thm:KernelComp1} and
\ref{Thm:KernelComp2}]
The assertion, except for those parts which involve
$\wideparen \maclC _{(s_j,s_k)}$
and $\wideparen \maclC _{(s_j,s_k)}^\star$ spaces,
follows from the kernel results Propositions \ref{Prop:KernelsDifficultDirectionExt}
to \ref{Prop:KernelsEasyDirectionNew} and the composition properties of the form
$$
\maclL (\maclA _{s_2}(\cc {d_2}),\maclA _{s_3}(\cc {d_3}))
\circ 
\maclL (\maclA _{s_1}(\cc {d_1}),\maclA _{s_2}(\cc {d_2}))
\subseteq
\maclL (\maclA _{s_1}(\cc {d_1}),\maclA _{s_3}(\cc {d_3}))
$$
and
$$
\maclL (\maclA _{s_2}(\cc {d_2}),\maclA _{s_3}'(\cc {d_3}))
\circ 
\maclL (\maclA _{s_1}(\cc {d_1}),\maclA _{s_2}(\cc {d_2}))
\subseteq
\maclL (\maclA _{s_1}(\cc {d_1}),\maclA _{s_3}'(\cc {d_3})).
$$

\par

Next we prove $\wideparen \maclB _{(s_3,s_2)}\circ
\wideparen \maclC _{(s_2,s_1)}\subseteq
\wideparen \maclC _{(s_3,s_1)}$. Suppose
$K_1\in \wideparen \maclC _{(s_2,s_1)}(W_1)$ and $K_2\in
\wideparen \maclB _{(s_3,s_2)}(W_2)$, and let $K_3=K_2\circ K_1$,
with coefficients $c(K_1;\alpha _2,\alpha _1)$, $c(K_2;\alpha _3,\alpha _2)$
and $c(K_3;\alpha _3,\alpha _1)$, $\alpha _j\in \nn {d_j}$, in their power
series expansions. Also let $\vartheta _{r,s}$ be the same as in Definition
\ref{DefSeqSpaces}. Then
$$
c(K_3;\alpha _3,\alpha _1) = \sum _{\beta \in \nn {d_2}}
c(K_2;\alpha _3,\beta )c(K_1;\beta ,\alpha _1).
$$
By the definitions, there is an $r_{0,1}>0$ such that
$$
|c(K_1;\alpha _2,\alpha _1)|
\lesssim
\frac {\vartheta _{r_{1},s_1}(\alpha _1)}{\vartheta _{r_{0,1},s_2}(\alpha _2)}
$$
holds for every $r_{1}>0$, and for every $r_{2}>0$ there
is an $r_{0,2}>0$ such that
\begin{equation}\label{Eq:CompCoeffEst}
|c(K_2;\alpha _3,\alpha _2)|
\lesssim
\frac {\vartheta _{r_{2},s_2}(\alpha _2)}{\vartheta _{r_{0,2},s_3}(\alpha _3)}
\end{equation}
holds. Hence, if we let $r_{2}<r_{0,1}$ and choose $r_{0,2}>0$ such that
\eqref{Eq:CompCoeffEst} holds, then
$$
|c(K_3;\alpha _3,\alpha _1)|
\lesssim
\sum _{\beta \in \nn {d_2}} \left (
\frac {\vartheta _{r_{2},s_2}(\beta )}{\vartheta _{r_{0,1},s_2}(\beta )}
\right )
\frac {\vartheta _{r_{1},s_1}(\alpha _1)}{\vartheta _{r_{0,2},s_3}(\alpha _3)}
\asymp
\frac {\vartheta _{r_{1},s_1}(\alpha _1)}{\vartheta _{r_{0,2},s_3}(\alpha _3)}
$$
for every $r_{1}>0$. This is the same as $K_3\in
\wideparen \maclC _{(s_3,s_1)}(W_3)$, and the assertion follows. In the same
way, we obtain the other assertions. The details are left for the reader.
\end{proof}

\par

In Section \ref{sec4} we present an analogy of
Theorem \ref{Thm:KernelComp2} for Wick operators (see Theorem
\ref{Thm:WickComp}).

\par

\section{Binomial operators and their continuity
properties}\label{sec3}

\par

In this section we consider linear and bilinear
operators of binomial types which act on subspaces
of $\ell _0'(\nn {2d})$ and which contain binomial expressions.
In the interplay between coefficients of Wick symbols and the
kernels to Wick operators, the
linear binomial operators are the actions on the coefficients
in power series
expansions, when passing between operator kernels, Wick
symbols and anti-Wick symbols of linear operators.
The bilinear binomial operators are in similar ways the actions on
the coefficients in power series expansions, which
correspond to pure multiplications, their adjoint actions,
and compositions of Wick operators.

\par

\subsection{Binomial type operators}

\par

Let $t\in \mathbf C$. Essential parts of our investigations concerns
continuity properties of the linear binomial operators
\begin{alignat}{2}
(\maclT _{0,t} c) (\alpha ,\beta) &= \sum_{\gamma \leq \alpha ,\beta}
\left ( { \alpha  \choose \gamma } { \beta  \choose \gamma } \right )^{1/2}
t^{|\gamma|}
c(\alpha - \gamma, \beta - \gamma) ,&
\quad t &\in \mathbf C,
\label{Eq:T_{0,t}Def}
\intertext{when $c\in \ell _0'(\nn {2d})$, and their formal $\ell ^2$
adjoint, which are given by
$\maclT _{0,\overline t} ^*c$, where}
(\maclT _{0,t} ^*c) (\alpha ,\beta) &= \sum_{\gamma \in \nn d}
\left ( { {\alpha +\gamma}  \choose \gamma }
{ {\beta +\gamma}  \choose \gamma }
\right )^{1/2}
 t^{|\gamma|}
 c(\alpha + \gamma, \beta + \gamma) ,&
\quad t &\in \mathbf C,
\label{Eq:T_{0,t}Dual}
\end{alignat}
when $c\in \ell _0(\nn {2d})$. In terms of the operator
\begin{equation}\label{Eq:SOpDef}
(S_0c)(\alpha ,\beta )=i^{|\alpha +\beta |}c(\alpha ,\beta ),
\qquad
\alpha ,\beta \in \nn d,\ c\in \ell _0'(\nn {2d}),
\end{equation}
we observe that
\begin{align}
\maclT _{0,-t} = S_0^{-1}\circ \maclT _{0,t} \circ S_0
\quad \text{and}\quad
\maclT _{0,-t}^* = S_0^{-1}\circ \maclT _{0,t}^* \circ S_0
\label{Eq:T0tSOpsConj}
\end{align}
in the domains of $\maclT _{0,-t}$ and $\maclT _{0,-t}^*$.

\par

The operators $\maclT _{0,1}$ and $\maclT _{0,1}^*$ are
linked to transitions between
kernel operators and Wick operators, and to transitions between Wick
and anti-Wick operators. In fact, for any fixed $t\in \mathbf C$, let $\maclT _t$
be the map on $\wideparen A(\cc {2d})$, given by
\begin{equation}\label{Eq:TtDef}
(\maclT _ta)(z,w) = e^{t(z,w)}a(z,w),\qquad a\in \wideparen A(\cc {2d})
=\wideparen \maclA _{\flat _1}(\cc {2d}), \quad
a\in \wideparen A(\cc {2d}).
\end{equation}
It follows in particular from \eqref{Eq:AnalPseudo},
\eqref{Eq:Kmap} and \eqref{Eq:Kmap}$'$ that
\begin{equation}\label{Eq:KernelWickRelByT}
T_K = \op _{\mathfrak V}(a)
\quad \Leftrightarrow \quad
K=\maclT _1a,
\end{equation}
when $a\in \wideparen A(\cc {2d})$, because
\begin{equation}\tag*{(\ref{Eq:KernelWickRelByT})$'$}
T_K = \op _{\mathfrak V}(a)
\quad \text{with}\quad
K(z,w)=e^{(z,w)}a(z,w)=(\maclT _1a)(z,w),
\end{equation}

\par

As a consequence of the following proposition we have
\begin{equation}
\maclT _{0,t}c(a;\cdo ) = c(\maclT _ta;\cdo )
\end{equation}
when $a\in \wideparen \maclA _0(\cc {2d})$, which gives the link
between the operator $\maclT _{0,1}$ and transitions between kernel
and Wick operators.
We omit the proof because the result is essentially a restatement of
\cite[Theorem 2.6]{Teofanov2}.

\par

\begin{prop}\label{Prop:TransKernelWickASpaces}
Let $s,s_0\in \overline {\mathbf R_\flat}$ be such that
$s<\frac 12$ and $s_0\le \frac 12$, $t\in \mathbf C$ and let
$T_{\maclA}$ be the map in \eqref{T12Map}.
Then $\maclT _t$ from $\wideparen \maclA _0(\cc {2d})$
to $\wideparen A(\cc {2d})$
extends uniquely to continuous mappings on
$\wideparen \maclA _s'(\cc {2d})$ and on
$\wideparen \maclA _{0,s_0}'(\cc {2d})$, and
the diagrams
\begin{equation}\label{Eq:CommDiagram}
\begin{CD}
\ell _s ' (\nn {2d})    @>\maclT _{0,t}>> \ell _s ' (\nn {2d})
\\
@V T _{\maclA}VV        @VV T_{\maclA}V\\
\wideparen \maclA _s'(\cc {2d})     @>>\maclT _t>
\wideparen \maclA _s'(\cc {2d})
\end{CD}
\quad \text{and}\quad
\begin{CD}
\ell _{0,s_0} ' (\nn {2d})    @>\maclT _{0,t}>> \ell _{0,s_0} ' (\nn {2d})
\\
@V T_{\maclA}VV        @VV T_{\maclA}V\\
\wideparen \maclA _{0,s_0}'(\cc {2d})     @>>\maclT _t>
\wideparen \maclA _{0,s_0}'(\cc {2d})
\end{CD}
\end{equation}
commute.
\end{prop}

\par

Here we observe the misprints in the commutative
diagram (2.5) in \cite{Teofanov2}, where $T_{\maclH}$ and
$\wideparen \maclA _{\flat _1}(\cc {2d})$ should be
replaced by $T_{\maclA}$ and
$\wideparen \maclA _{\flat _1}'(\cc {2d})$, respectively, at each occurrence.

\par

By \eqref{Eq:TtDef}, the commutative diagrams
\eqref{Eq:CommDiagram} and that $\maclT _{0,\overline t}^*$
is the $\ell ^2$ adjoint of $\maclT _{0,t}$ it follows that
\begin{equation}
\begin{alignedat}{2}
\maclT _{0,t_1}\circ \maclT _{0,t_2} &= \maclT _{0,t_1+t_2}, &
\quad
\maclT _{0,t}^{-1} &= \maclT _{0,-t},
\\[1ex]
\maclT _{0,t_1}^*\circ \maclT _{0,t_2}^* &= \maclT _{0,t_1+t_2}^*, &
\quad \text{and}\quad
(\maclT _{0,t}^*)^{-1} &= \maclT _{0,-t}^*,
\quad t,t_1,t_2\in \mathbf C,
\end{alignedat}
\end{equation}
since similar facts hold true for the operator $\maclT _t$.

\medspace

The operator $\maclT _{0,t}^*$ is linked to the operator
$\maclT ^*_t$, given by
\begin{equation}\label{Eq:Tt*Def}
(\maclT ^*_ta)(t_0z,\overline {t_0}w)
= \pi ^{-d}
\int _{\cc d}a({t_0}w_1,\overline {t_0}w_1)
e^{-(z-w_1,w-w_1)}\, d\lambda (w_1),\quad t_0^2=t, 
\end{equation}
when $a\in \wideparen \maclA _0(\cc {2d})$. We observe that due
to the definitions it follows that $a^{\aw} (z,w)$ in \eqref{Eq:AntiWickAnalPseudoRel}
is equal to $\maclT _1^*a(z,w)$. Hence,
\begin{equation}\label{Eq:AntiWicktoWickRevised}
\op _{\mathfrak V}(\maclT _1^*a) = \op _{\mathfrak V}^{\aw}(a)
\quad \text{and}\quad
\op _{\mathfrak V}^{\aw}(\maclT _{-1}^*a) = \op _{\mathfrak V}(a)
\end{equation}
when $a\in \wideparen \maclA _0(\cc {2d})$. 

\par

By a straight-forward
consequence of \eqref{Eq:WickAntiWickBasicTransf2} in Appendix
\ref{App:B} it follows that the diagram
\begin{equation}\label{Eq:CommDiagramDual}
\begin{CD}
\ell _0 (\nn {2d})    @>\maclT _{0,t}^*>> \ell _0 (\nn {2d})
\\
@V T_{\maclA}VV        @VVT_{\maclA}V\\
\wideparen \maclA _0(\cc {2d})     @>>\maclT ^*_t>
\wideparen \maclA _0(\cc {2d})
\end{CD}
\end{equation}
commute. A combination of \eqref{Eq:AntiWicktoWickRevised}
and \eqref{Eq:CommDiagramDual} then gives
\begin{equation}
c(a^{\aw};\cdo ) = \maclT _{0,1}^*c(a;\cdo ),
\end{equation}
when \eqref{Eq:AntiWickAnalPseudoRel} holds and
$a\in \wideparen \maclA _0(\cc {2d})$, which in turn is equivalent
to
\begin{equation}
\op _{\mathfrak V}(b) = \op _{\mathfrak V}^{\aw}(a)
\quad \Leftrightarrow \quad
c(b;\cdo ) = \maclT _{0,1}^*c(a;\cdo ),
\end{equation}
when $a,b\in \wideparen \maclA _0(\cc {2d})$.
This gives the link between $\maclT _{0,\pm 1}^*$ and transitions
between Wick and anti-Wick operators.

\medspace

Next we introduce certain types of bilinear binomial operators.

\par

\begin{defn}\label{Def:BilinearBinomialOps}
Let $t\in \mathbf C$. 
\begin{enumerate}
\item
The bilinear operator
$(c_1,c_2)\mapsto (c_1\bullet _t c_2)$
from $\ell _0(\nn d )\times \ell _0(\nn d )$ to
$\ell _0(\nn d )$ is given by
\begin{equation}\label{Eq:BilinearBinomOpAdj}
(c_1 \bullet _t c_2)(\alpha )
\equiv
\sum _{\gamma \le \alpha }
{{\alpha }\choose {\gamma }}^{\frac 12}
t^{|\gamma |}
c_1(\alpha -\gamma )
c_2(\gamma ).
\end{equation}

\vrum

\item
The bilinear operator
$(c_1,c_2)\mapsto (c_1\bullet _{(t,\diamond )} c_2)$
from $\ell _0(\nn d )\times \ell _0(\nn d )$ to
$\ell _0(\nn d )$ is given by
\begin{equation}\label{Eq:BilinearBinomOp}
(c_1 \bullet _{(t,\diamond )} c_2)(\alpha )
\equiv
t^{|\alpha |}\sum _{\gamma \in \nn {d}}
{{\alpha +\gamma} \choose {\gamma}}^{\frac 12}
c_1(\gamma )
c_2(\alpha +\gamma ).
\end{equation}

\vrum

\item
The bilinear operator 
$(c_1,c_2)\mapsto (c_1\bullet _{(t,\wpr )} c_2)$
from $\ell _0(\nn {2d} )\times \ell _0(\nn {2d})$ to
$\ell _0(\nn {2d})$ is given by
\begin{multline}\label{Eq:WickCoeffComp2}
(c_1\bullet _{(t,\wpr )} c_2)(\alpha ,\beta )
\\[1ex]
\equiv
\sum 
\left (
{{\alpha }\choose {\alpha _0}}{{\beta }\choose {\beta _0}}
{{\alpha -\alpha _0+\gamma}\choose {\gamma}}
{{\beta -\beta _0+\gamma}\choose {\gamma}}
\right )^{\frac 12}t^{|\gamma |}
c_1(\alpha _0,\beta -\beta _0+\gamma)
c_2(\alpha -\alpha _0+\gamma ,\beta _0),
\end{multline}
where the sum is taken over all $\alpha _0,\beta _0,\gamma \in \nn d$
such that $\alpha _0\le \alpha$, $\beta _0\le \beta$.
\end{enumerate}
\end{defn}

\par

It follows by straight-forward computations that
\begin{align}
(c_1\bullet _t c_2)(\alpha )
&=
t^{|\alpha |}(c_2\bullet _{1/t} c_1)(\alpha )
\label{Eq:BilinBinomComm}
\intertext{and}
(c_1\bullet _t c_2,c_3)_{\ell ^2(\nn d)}
&= 
(c_2,\overline {c_1}\bullet _{(\overline t,\diamond )}
c_3)_{\ell ^2(\nn d)},
\label{Eq:BilinBinomAdj}
\end{align}
$c_j\in \ell _0(\Lambda )$, $j=1,2,3$.

\par

The different products in Definition \ref{Def:BilinearBinomialOps} are
linked into different bilinear mappings for spaces of power series expansions.
Let $W=\cc {d_2}\times \cc {d_1}$ and
$K_1,K_2\in \wideparen \maclA _0(W)$. Then
\begin{gather*}
K_1(z_2,z_1)K_2(z_2,z_1) = \sum _{\alpha _{j}\beta _{j}}
c(K_1;\alpha _{2},\alpha _{1}) c(K_2;\beta _{2},\beta _{1})
e_{\alpha _2}(z_2)e_{\beta _2}(z_2)
e_{\alpha _1}(z_1)e_{\beta _1}(z_1)
\\[1ex]
= \sum _{\alpha _{1}\alpha _{2}}
\left (
\sum _{\gamma _j\le \alpha _j} \left (
{{\alpha _1}\choose {\gamma _1}} {{\alpha _2}\choose {\gamma _2}}
\right ) ^{\frac 12}
c(K_1;\alpha _{2}-\gamma _2,\alpha _{1}-\gamma _1) 
c(K_2;\gamma _{2},\gamma _{1})
\right )
e_{\alpha _{2}}(z_2)e_{\alpha _{1}}(z_1),
\end{gather*}
and it follows that
\begin{equation}\label{Eq:MultCoeffRel}
c(K_1\cdot K_2;\cdo ) = c(K_1;\cdo ) \bullet _1 c(K_2;\cdo ).
\end{equation}

\par

By \eqref{Eq:BilinBinomAdj} it follows that $\bullet _{(1,\diamond )}$
is the adjoint operation of $\bullet _1$ on elements in $\ell _0(\Lambda)$.
In order to find corresponding
relationship for elements in $\wideparen \maclA _0(W)$, we observe  that the adjoint operation of $z_j$ is $\partial _{z_j}$. In fact,
if $F,G\in \maclA _{0}(\cc d)$, then it follows by integrating by parts
that
\begin{equation}\label{Eq:MonomMultAdjoint}
(z_jF,G)_{A^2} = (F,\partial _jG)_{A^2}
\quad \text{and}\quad
(\partial _jF,G)_{A^2} = (F,z_jG)_{A^2}.
\end{equation}
Hence, by letting $\diamond$ be the multiplication on $\maclA _0(\cc d)$,
given by the formula
\begin{equation}\label{Eq:ComplFourMult}
(F_1\diamond F_2,F)_{A^2} = (F_2,F_0\cdot F)_{A^2},
\qquad
F_0(z) = \overline {F_1(\overline z)},
\end{equation}
when $F,F_1,F_2\in \maclA _0(\cc d)$, it follows that
\begin{equation}\label{Eq:MonomMultAdjoint2}
(F_1\diamond F_2)(z)
= 
\sum _{\alpha \in \nn d}
c(F_1,\alpha )\frac {\partial _z^\alpha F_2(z)}{\alpha !^{\frac 12}}.
\end{equation}
Since
$\op _{\mathfrak V}^{\aw}(\overline w_j)= \partial _j=\partial _{z_j}$,
it also follows that
$$
F_1\diamond F_2 =\op _{\mathfrak V}^{\aw}(F_1)F_2=F_1(\nabla _z)F_2.
$$

\par

In similar ways, if $W=\cc {d_2}\times \cc {d_1}$ and
$K_1,K_2\in \wideparen \maclA _0(W)$, then $K_1\diamond K_2$ is defined
by
\begin{equation}\label{Eq:MonomMultAdjoint3}
(K_1\diamond K_2)(z_2,z_1)
= 
\sum _{\alpha _1\in \nn {d_1}}\sum _{\alpha _2\in \nn {d_2}}
c(K_1;\alpha _2,\alpha _1)\frac {\partial _{z_2}^{\alpha _2}
\partial _{\overline z_1}^{\alpha _1} K_2(z_2,z_1)}
{(\alpha _1!\alpha _2!)^{\frac 12}}.
\end{equation}
It follows that
\begin{equation}\label{Eq:ComplFourMult2}
(K_1\diamond K_2,K)_{\wideparen A^2}
=
(K_2,K_0\cdot K)_{\wideparen A^2},
\qquad
K_0(z_2,z_1) = \overline {K_1(\overline z_2,\overline z_1)},
\end{equation}
when $K,K_1,K_2\in \wideparen \maclA _0(W)$. A combination of
\eqref{Eq:BilinBinomAdj}, \eqref{Eq:MultCoeffRel} and
\eqref{Eq:ComplFourMult2}, now gives
\begin{equation}\label{Eq:MultDiamCoeffRel}
c(K_1\diamond K_2;\cdo ) = c(K_1;\cdo )
\bullet _{(1,\diamond )} c(K_2;\cdo ).
\end{equation}

\par

\begin{example}
We observe that
\eqref{Eq:MonomMultAdjoint2} and straight-forward computations
give
\begin{equation}\label{Eq:MonomMultAdjointBasic}
e_\alpha \diamond e_\beta =
\begin{cases}
{\beta \choose \alpha }^{\frac 12}e_{\beta -\alpha}, & \text{when}
\ \beta \ge \alpha ,
\\[1ex]
0, & \text{otherwise},
\end{cases}
\end{equation}
\end{example}

\medspace

The following lemma shows that
$\bullet _{(1,\wpr)}$ in Definition \ref{Def:BilinearBinomialOps}
is linked to compositions of Wick operators on the symbol side.

\par

\begin{lemma}\label{Lemma:WickCoeffComp}
Let $W=\cc d\times \cc d$, $\bullet _{(1,\wpr )}$ be the
multiplication on $\ell _0(\nn d\times \nn d)$ given by
\eqref{Eq:WickCoeffComp2}, and let
$$
u_{\alpha ,\beta}(z,w) = e_\alpha (z)e_\beta (\overline w)
\in \wideparen \maclA _0(W).
$$
Then the following is true:
\begin{enumerate}
\item if $\alpha _j,\beta _j\in \nn d$, $j=1,2$, then
\begin{multline}\label{Eq:WickCoeffComp1}
u_{\alpha _1,\beta _1} \wpr _{\mathfrak V} u_{\alpha _2,\beta _2}
\\[1ex]
= \!\!
\sum _{\gamma \le \alpha _2,\beta _1}\!\!
\left (
{{\alpha _2}\choose {\gamma}}{{\beta _1}\choose {\gamma}}
{{\alpha _1+\alpha _2-\gamma}\choose {\alpha _1}}
{{\beta _1+\beta _2-\gamma}\choose {\beta _2}}
\right )^{\frac 12}
u_{\alpha _1+\alpha _2-\gamma ,\beta _1+\beta _2-\gamma}
\text ;
\end{multline}

\vrum

\item
if $a_1,a_2 \in \wideparen \maclA _0(W)$,
then
\begin{equation}\label{Eq:MultWprCoeffRel}
c(a_1\wpr _{\mathfrak V}a_2;\cdo )= c(a_1;\cdo )
\bullet _{(1,\wpr )} c(a_2;\cdo ).
\end{equation}
\end{enumerate}
\end{lemma}

\par

\begin{proof}
By using that
$$
I_{\gamma ,\delta} = \pi ^{-d}\int _{\cc d}w^\gamma {\overline w}^\delta
e^{-|w|^2}\, d\lambda (w) =
\begin{cases}
\gamma ! , & \gamma =\delta
\\[1ex]
0 , & \gamma \neq \delta ,
\end{cases}
$$
we get
\begin{multline*}
\pi ^{-d}\int _{\cc d} e_\alpha (w_1)e_\beta (\overline w_1)
e^{-|z-w_1|^2}\, d\lambda (w_1)
=
\frac 1{\pi ^{d}(\alpha !\beta !)^{\frac 12}}\sum _{\gamma \le \alpha}
\sum _{\delta \le \beta} {{\alpha}\choose {\gamma}}{{\beta}\choose {\delta}}
z^{\alpha -\gamma}{\overline z}^{\beta -\delta}I_{\gamma ,\delta}
\\[1ex]
=
\sum _{\gamma \le \alpha ,\beta}
\left (
{{\alpha}\choose {\gamma}}{{\beta}\choose {\gamma}}
\right )^{\frac 12}
e_{\alpha -\gamma}(z)e_{\beta -\delta}(\overline z).
\end{multline*}
This gives
\begin{multline}\label{Eq:BasicIntegral1}
\pi ^{-d}\int _{\cc d} e_\alpha (w_1)e_\beta (\overline w_1)
e^{-(z-w_1,w-w_1)}\, d\lambda (w_1)
\\[1ex]
=
\sum _{\gamma \le \alpha ,\beta}
\left (
{{\alpha}\choose {\gamma}}{{\beta}\choose {\gamma}}
\right )^{\frac 12}
e_{\alpha -\gamma}(z)e_{\beta -\delta}(\overline w).
\end{multline}

\par

Hence, if $F_{\alpha ,\beta}(z,w)$ is the left-hand side of
\eqref{Eq:BasicIntegral1}, then we obtain
\begin{align*}
(u_{\alpha _1,\beta _1}\wpr _{\mathfrak V} u_{\alpha _2,\beta _2})(z,w)
&=
e_{\alpha _1}(z)e_{\beta _2}(\overline w)F_{\alpha _2,\beta _1}(z,w)
\\[1ex]
&=
\sum _{\gamma \le \alpha _2,\beta _1}
\left (
{{\alpha _2}\choose {\gamma}}{{\beta _1}\choose {\gamma}}
\right )^{\frac 12}
(e_{\alpha _2-\gamma}(z)e_{\alpha _1}(z))
(e_{\beta _1-\delta}(\overline w)e_{\beta _2}(\overline w)),
\end{align*}
which after some straight-forward computations lead to
the right-hand side of \eqref{Eq:WickCoeffComp1}. This gives
(1).

\par

Let
\begin{align*}
C_1(\alpha ,\beta ,\gamma )
&=
\left (
{{\alpha _2}\choose {\gamma}}{{\beta _1}\choose {\gamma}}
{{\alpha _1+\alpha _2-\gamma}\choose {\alpha _1}}
{{\beta _1+\beta _2-\gamma}\choose {\beta _2}}
\right )^{\frac 12}
\intertext{and}
C_2(\alpha ,\beta ,\gamma )
&=
\left (
{{\alpha _2}\choose {\alpha _1}}{{\beta _1}\choose {\beta _2}}
{{\alpha _2-\alpha _1+\gamma}\choose {\gamma}}
{{\beta _1-\beta _2+\gamma}\choose {\gamma}}
\right )^{\frac 12},
\end{align*}
for admissible $\alpha =(\alpha _1,\alpha _2)\in \nn {2d}$,
$\beta =(\beta _1,\beta _2)\in \nn {2d}$ and $\gamma \in \nn d$.
By (1) we have
\begin{multline*}
(a_1\wpr _{\mathfrak V}a_2)(z,w)
\\[1ex]
=
\sum _{\alpha _j,\beta _j} 
\sum _{\gamma \le \alpha _2,\beta _1}
C_1(\alpha ,\beta ,\gamma )c(a_1;\alpha _1,\beta _1)c(a_2;\alpha _2,\beta _2)
e_{\alpha _1+\alpha _2-\gamma}(z)e_{\beta _1+\beta _2-\gamma}(\overline w)
\\[1ex]
=
\sum _{\alpha _1,\beta _2}
\sum _{\gamma } \sum _{\alpha _2,\beta _1\ge \gamma}
C_1(\alpha ,\beta ,\gamma )c(a_1;\alpha _1,\beta _1)c(a_2;\alpha _2,\beta _2)
e_{\alpha _1+\alpha _2-\gamma}(z)e_{\beta _1+\beta _2-\gamma}(\overline w).
\end{multline*}
By taking
$$
\alpha _1 ,\quad  \alpha _1+\alpha _2-\gamma,
\quad \beta _1+\beta _2-\gamma, \quad
\beta _2
\quad \text{and}\quad
\gamma
$$
as new variables of summations, we get
\begin{multline*}
(a_1\wpr _{\mathfrak V}a_2)(z,w)
\\[1ex]
=
\sum _{\alpha _1,\beta _2}
\left (
\sum _{\gamma }
\underset {\beta _1\le \beta _1}{\sum _{\alpha _1\le \alpha _2}}
C_2(\alpha ,\beta ,\gamma)
c(a_1; \alpha _1,\beta _1-\beta _2+\gamma )
c(a_2; \alpha _2-\alpha _1+\gamma ,\beta _2)
\right )
e_{\alpha _2}(z)e_{\beta _1}(\overline w).
\end{multline*}
This is the same as \eqref{Eq:MultWprCoeffRel}, and the result follows.
\end{proof}

\par

\subsection{Continuity for binomial operators}

\par

The assertion (1) in the following proposition
was deduced in \cite{Teofanov2}, and (2) follows
from (1) and duality. The details are left for the reader.

\par

\begin{prop}\label{Prop:T0tBasichomeomorphism}
Let $t \in \mathbf C$, $ s,s_0 \in \overline{\mathbf R_\flat}$ be such that
$ s < \frac 12$ and $ 0 < s_0 \leq \frac 12$,
$\maclT _{0,t},S_0\in \maclL (\ell _0' (\nn {2d}))$ be given by \eqref{Eq:T_{0,t}Def}
and \eqref{Eq:SOpDef}, and $\maclT _{0,t}^*\in \maclL (\ell _0(\nn {2d}))$ be given by
\eqref{Eq:T_{0,t}Dual}.
Then the following is true:
\begin{enumerate}
\item $\maclT _{0,t} $ is a homeomorphism on $ \ell _0' (\nn {2d})$
with the inverse  $\maclT _{0,-t} $. Furthermore,  $\maclT _{0,t} $ restricts to
homeomorphisms on $ \ell _s ' (\nn {2d}) $ and
on $ \ell _{0,s_0} ' (\nn {2d})$, and $\maclT _{0,-t}$ can be obtained from
\eqref{Eq:T0tSOpsConj};

\vrum

\item $\maclT _{0,t}^*$ is a homeomorphism on $ \ell _0(\nn {2d})$
with the inverse  $\maclT _{0,-t}^*$. Furthermore,  $\maclT _{0,t}^*$ is uniquely extendable
to homeomorphisms on $ \ell _s (\nn {2d}) $ and
on $ \ell _{0,s_0} (\nn {2d})$, and $\maclT _{0,-t}^*$ can be obtained from
\eqref{Eq:T0tSOpsConj}.
\end{enumerate}
\end{prop}

\par

We have now the following complementary result to
Proposition \ref{Prop:T0tBasichomeomorphism}.

\par

\begin{prop}\label{Prop:T0thomeomorphismEllC}
Let $t \in \mathbf C$, $s,s_0 \in \overline{\mathbf R_\flat}$ be such that
$s < \frac 12$ and $0 < s_0 \leq \frac 12$,
$\maclT _{0,t}, S_0\in \maclL (\ell _0' (\nn {2d}))$ given by \eqref{Eq:T_{0,t}Def}
and \eqref{Eq:SOpDef}, and $\maclT _{0,t}^*\in \maclL (\ell _0(\nn {2d}))$ given by
\eqref{Eq:T_{0,t}Dual}. Then the following is true:
\begin{enumerate}
\item
the map $\maclT _{0,t}$ restricts to homeomorphisms on
\begin{alignat}{4}
&\ell _{\maclB ,s} (\nn {2d}), &
\quad
&\ell _{\maclB ,s} ^\star (\nn {2d}), &
\quad
&\ell _{\maclB ,0,s_0} (\nn {2d}), &
\quad 
&\ell _{\maclB ,0,s_0} ^\star (\nn {2d}),
\label{Eq:T0thomeomorphismEllC1}
\end{alignat}
with inverse $\maclT _{0,-t}$, and \eqref{Eq:T0tSOpsConj} holds;

\vrum

\item
the map $\maclT _{0,t}^*$ on $\ell _0(\nn {2d})$ extends uniquely
to homeomorphisms on
\begin{alignat}{4}
&\ell _{\maclC ,s} (\nn {2d}), &
\quad
&\ell _{\maclC ,s} ^\star (\nn {2d}), &
\quad
&\ell _{\maclC ,0,s_0} (\nn {2d}), &
\quad 
&\ell _{\maclC ,0,s_0} ^\star (\nn {2d}),
\label{Eq:T0thomeomorphismEllC2}
\end{alignat}
with inverse $\maclT _{0,-t}^*$, and \eqref{Eq:T0tSOpsConj} holds.
\end{enumerate}
\end{prop}

\par

\begin{proof}
We only prove (1) for the spaces
$\ell _{\maclB ,s} (\nn {2d})$, as well as for
$\ell _{\maclB ,0,s_0} (\nn {2d})$ when $s_0=\frac 12$, and we only
prove (2) for $\ell _{\maclC ,s} (\nn {2d})$. 
The other assertions follow by similar arguments
and are left for the reader.

\par

First suppose that $s<\frac 12$, $r_{0,j},r_j,\theta >0$, $j=1,2$, and that
\begin{equation}\label{Eq:CondSeqvartheta}
|c(\alpha ,\beta )|\lesssim
\frac {\vartheta _{r_2,s}(\beta )}{\vartheta _{r_{0,2},s}(\alpha )}.
\end{equation}
Then it follows from standard estimates of binomial coefficients that
\begin{align}
|(\maclT _{0,t}c)(\alpha ,\beta )
{\vartheta _{r_{0,1},s}(\alpha )}/{\vartheta _{r_1,s}(\beta )}|
&\lesssim
\sum _{\gamma \le \alpha ,\beta}
{\alpha \choose \gamma}^{\frac 12}{\beta \choose \gamma}^{\frac 12}
|t|^{|\gamma|}
\frac {\vartheta _{r_2,s}(\beta -\gamma )\vartheta _{r_{0,1},s}(\alpha )}
{\vartheta _{r_{0,2},s}(\alpha -\gamma )\vartheta _{r_1,s}(\beta )}
\notag
\\[1ex]
&\le
(2+2|t|)^{|\alpha +\beta |}
\sum _{\gamma \le \alpha ,\beta}
\frac {\vartheta _{r_2,s}(\beta -\gamma )\vartheta _{r_{0,1},s}(\alpha )}
{\vartheta _{r_{0,2},s}(\alpha -\gamma )\vartheta _{r_1,s}(\beta )}.
\label{Eq:EstTotOp}
\end{align}

\par

Suppose that $c\in \ell _{\maclB ,s}(\nn {2d})$ and let $r_1>0$
be arbitrary but fixed. Let $c_0\in (0,1)$ be a small constant which shall be
chosen later on. By Lemma \ref{Lemma:ThetaWeightBasic},
we may choose $r_2<r_1$ and then $r_{0,2}\in (0,r_2)$ such that
$$
\vartheta _{r_2,s}(\beta -\gamma )\vartheta _{r_2,s}(\gamma )
\lesssim
\vartheta _{c_0r_1,s}(\beta )
$$
and that \eqref{Eq:CondSeqvartheta}
holds. Then \eqref{Eq:EstTotOp} gives
\begin{align*}
|(\maclT _{0,t}c)(\alpha ,\beta )
{\vartheta _{r_{0,1},s}(\alpha )}/{\vartheta _{r_1,s}(\beta )}|
&\lesssim
(2+2|t|)^{|\alpha +\beta |}\frac {\vartheta _{c_0r_1,s}(\beta )}
{\vartheta _{r_1,s}(\beta )}
\sum _{\gamma \le \alpha ,\beta}
\frac {\vartheta _{r_{0,1},s}(\alpha )}
{\vartheta _{r_{0,2},s}(\alpha -\gamma )\vartheta _{r_2,s}(\gamma )}
\\[1ex]
&\lesssim
(2+2|t|)^{|\alpha +\beta |}\frac
{\vartheta _{r_{0,1},s}(\alpha )\vartheta _{c_0r_1,s}(\beta )}
{\vartheta _{r_{0,2}/C,s}(\alpha )\vartheta _{r_1,s}(\beta )}
\left (\sum _{\gamma \le \alpha ,\beta} 1 \right )
\\[1ex]
&\lesssim
(3+2|t|)^{|\alpha +\beta |}\frac
{\vartheta _{r_{0,1},s}(\alpha )\vartheta _{c_0r_1,s}(\beta )}
{\vartheta _{r_{0,2}/C,s}(\alpha )\vartheta _{r_1,s}(\beta )}
\end{align*}

\par

Since $s<\frac 12$, it follows that
$$
(3+2|t|)^{|\alpha |}\frac {\vartheta _{r_{0,1},s}(\alpha )}
{\vartheta _{r_{0,2}/C,s}(\alpha )} \le C_0
\quad \text{and}\quad
(3+2|t|)^{|\beta |}\frac
{\vartheta _{c_0r_1,s}(\beta )}
{\vartheta _{r_1,s}(\beta )}\le C_0,
$$
for some constant $C_0>0$ which is independent of $\alpha$
and $\beta$, provided $c_0$ and $r_{0,1}>0$ are chosen small enough.
A combination of these estimates shows that
for every $r>0$, there is an $r_0>0$
such that
\begin{equation}\label{Eq:FinalEstT0tOp}
|(\maclT _{0,t}c)(\alpha ,\beta )|
\lesssim
\frac
{\vartheta _{r,s}(\beta )}
{\vartheta _{r_0,s}(\alpha )},
\end{equation}
which implies that $\maclT _{0,t}c \in \ell _{\maclB ,s}(\nn {2d})$. Since it is
clear that the choice of $r_0$ depends continuously of the interplay between
$r_2$ and $r_{0,2}$ in \eqref{Eq:CondSeqvartheta}, it also
follows that $\maclT _{0,t}$ is continuous on $\ell _{\maclB ,s}(\nn {2d})$,
which proves the asserted continuity for $\maclT _{0,t}$ on 
$\ell _{\maclB ,s}(\nn {2d})$.

\par

Next suppose that $c\in \ell _{\maclB ,0,s_0}(\nn {2d})$ with
$s_0=\frac 12$. That is,
for every $r>0$, there is an $r_0>0$ such that
\begin{equation}\label{Eq:TotContNonflatLast}
|c(\alpha ,\beta )|\lesssim e^{-r|\alpha |+r_0|\beta |}.
\end{equation}
We shall prove that for every $r>0$, there is an $r_0>0$ such that
\begin{equation}\label{Eq:TotContNonflatLastOp}
|(\maclT _{0,t}c)(\alpha ,\beta )|\lesssim e^{-r|\alpha |+r_0|\beta |},
\end{equation}
and then it suffices to prove this when $r_0\ge r$.
By Cauchy-Schwartz inequality we get
\begin{align*}
|(\maclT _{0,t}c)(\alpha ,\beta )|
&\lesssim
\sum _{\gamma \le \alpha ,\beta}
{\alpha \choose \gamma}^{\frac 12}{\beta \choose \gamma}^{\frac 12}
|t|^{|\gamma|}
e^{-r|\alpha -\gamma |+r_0|\beta -\gamma |}
\\[1ex]
&\le
e^{-r|\alpha |+r_0|\beta |}
\left (
\sum _{\gamma \le \alpha}
{\alpha \choose \gamma}
|t|^{|\gamma|}
\right )^{\frac 12}
\left (
\sum _{\gamma \le \beta}
{\beta \choose \gamma}
|t|^{|\gamma|}
\right )^{\frac 12}
\\[1ex]
&=
e^{-(r-r_t)|\alpha |+(r_0+r_t)|\beta |},
\end{align*}
where
$$
r_t= \frac 12\ln (1+|t|).
$$
Since $r$ can be chosen arbitrarily large, it follows that for every
$r>0$, there is an $r_0>0$ such that \eqref{Eq:TotContNonflatLastOp}
holds. Consequently, $\maclT _{0,t}$ is continuous on
$\ell _{\maclB ,0,s_0}(\nn {2d})$.

\par

Next we prove (2) when $\maclT _{0,t}^*$ acts on
$\ell _{\maclC ,s} (\nn {2d})$. Suppose that $c\in
\ell _{\maclC ,s} (\nn {2d})$. Then there is an $r_0=r_{0,2}>0$
such that \eqref{Eq:CondSeqvartheta} holds for every $r=r_2>0$.
We may assume that $r<c_0r_0$ for some constant $c_0\in (0,1)$ which
shall be determined later on.

\par

By Lemma \ref{Lemma:ThetaWeightBasic} we obtain
\begin{align*}
|(\maclT _{0,t}^*c)(\alpha ,\beta )|
&\le
\sum _{\gamma}
\left (
{{\alpha +\gamma} \choose \gamma}
{{\beta +\gamma} \choose \gamma}
\right )^{\frac 12}|t|^{|\gamma |}|c(\alpha +\gamma ,\beta +\gamma )|
\\[1ex]
&\lesssim
2^{\frac 12 |\alpha +\beta |}
\sum _{\gamma}
|2t|^{|\gamma |}|
\frac {\vartheta _{r,s}(\beta +\gamma )}
{\vartheta _{r_0,s}(\alpha +\gamma )}
\\[1ex]
&\lesssim
\frac {2^{\frac 12 |\alpha +\beta |}\vartheta _{C_1r,s}(\beta )}
{\vartheta _{r_0/C_1,s}(\alpha )}\cdot J
\\[1ex]
&\lesssim
\frac {\vartheta _{C_2r,s}(\beta )}
{\vartheta _{r_0/C_2,s}(\alpha )}\cdot J,
\end{align*}
for some constants $C_2>C_1\ge 1$ which are independent of $c_0$, $r$ and
$r_0$, where
$$
J=
\sum _{\gamma}
|2t|^{|\gamma |}|
\frac {\vartheta _{C_1c_0r_0,s}(\gamma )}
{\vartheta _{r_0/C_1,s}(\gamma )}.
$$
In the last inequality we have used the fact that $s<\frac 12$.

\par

By choosing $c_0$ small enough it follows that $J$ is convergent. This shows
that for some $r_0>0$, \eqref{Eq:FinalEstT0tOp} holds for every $r>0$.
That is, $\maclT _{0,t}c \in \ell _{\maclC ,s}(\nn {2d})$
when $c \in \ell _{\maclC ,s}(\nn {2d})$. The continuity assertions of
$\maclT _{0,t}$ on $\ell _{\maclC ,s}(\nn {2d})$ now follows from
this fact and that $r_0$ in \eqref{Eq:FinalEstT0tOp} depends
continuously on $r_{0,2}$ in \eqref{Eq:CondSeqvartheta}.
This gives the result.
\end{proof}

\par

In the following three propositions we
show that the bilinear mappings in Definition
\ref{Def:BilinearBinomialOps} can be extended in suitable ways.

\par

\begin{prop}\label{Prop:RingModPropAnSpacesHelpAdj}
Let $t\in \mathbf C$, $s_1,s_2\in \overline {\mathbf R}_{\flat}$
be such that $s_1,s_2< \frac 12$,
$s=(s_1,s_2)$, $\Lambda
=\nn {d_2}\times \nn {d_1}$ and let $U_t$ be the map from
$\ell _0(\Lambda )\times \ell _0(\Lambda )$
to $\ell _0(\Lambda )$, given by $U_t(c_1,c_2)=
c_1 \bullet _t c_2$, where 
$c_1\bullet _t c_2$ is as in Definition
\ref{Def:BilinearBinomialOps}.
Then $U_t$ is uniquely extendable to continuous mappings
\begin{equation}
\begin{alignedat}{2}
U_t &: \ell _{\maclA ,s} (\Lambda )\times \ell _{\maclA ,s} (\Lambda )
\to \ell _{\maclA ,s} (\Lambda ), &
\quad
U_t &: \ell _{\maclA ,s}^\star (\Lambda )\times
\ell _{\maclA ,s}^\star (\Lambda )
\to
\ell _{\maclA ,s} ^\star (\Lambda ),
\\[1ex]
U_t &: \ell _{\maclB ,s} (\Lambda )\times \ell _{\maclB ,s} (\Lambda )
\to \ell _{\maclB ,s} (\Lambda ), &
\quad
U_t &: \ell _{\maclB ,s}^\star (\Lambda )\times
\ell _{\maclB ,s}^\star (\Lambda )
\to
\ell _{\maclB ,s} ^\star (\Lambda ),
\\[1ex]
U_t &: \ell _{\maclC ,s} (\Lambda )\times \ell _{\maclC ,s} (\Lambda )
\to \ell _{\maclC ,s} (\Lambda ) &
\quad \text{and}\quad
U_t &: \ell _{\maclC ,s}^\star (\Lambda )\times
\ell _{\maclC ,s}^\star (\Lambda )
\to
\ell _{\maclC ,s} ^\star (\Lambda ).
\end{alignedat}
\end{equation}

\par

If instead $0<s_1,s_2\le \frac 12$, then
the same holds true with
$\ell _{\maclA ,0,s}$, $\ell _{\maclB ,0,s}$
and $\ell _{\maclC ,0,s}$ in place of
$\ell _{\maclA ,s}$, $\ell _{\maclB ,s}$
and $\ell _{\maclC ,s}$, respectively, at
each occurrence.
\end{prop}

\par

\begin{prop}\label{Prop:RingModPropAnSpacesHelp}
Let $t\in \mathbf C$, $s_1,s_2\in \overline {\mathbf R}_{\flat}$
be such that $s_1,s_2< \frac 12$, $s=(s_1,s_2)$,
$\Lambda =\nn {d_2}\times \nn {d_1}$ and let $U_{t,\diamond}$ be the map from
$\ell _0(\Lambda )\times \ell _0(\Lambda )$
to $\ell _0(\Lambda )$, given by $U_{t,\diamond}(c_1,c_2)=
c_1 \bullet _{(t ,\diamond )} c_2$, where 
$c_1\bullet _{(t,\diamond )} c_2$ is as in Definition
\ref{Def:BilinearBinomialOps}.
Then $U_{t,\diamond}$ is uniquely extendable to continuous mappings
\begin{equation}
\begin{alignedat}{2}
U_{t,\diamond} &: \ell _{\maclA ,s}
(\Lambda )\times \ell _{\maclA ,s} ^\star(\Lambda )
\to \ell _{\maclA ,s} ^\star(\Lambda ), &
\quad
U_{t,\diamond} &: \ell _{\maclA ,s} ^\star (\Lambda )\times
\ell _{\maclA ,s} (\Lambda )
\to
\ell _{\maclA ,s}  (\Lambda ),
\\[1ex]
U_{t,\diamond} &: \ell _{\maclB ,s} (\Lambda )
\times \ell _{\maclC ,s}^\star (\Lambda )
\to \ell _{\maclC ,s}^\star (\Lambda ), &
\quad
U_{t,\diamond} &: \ell _{\maclB ,s}^\star (\Lambda )\times
\ell _{\maclC ,s} (\Lambda )
\to
\ell _{\maclC ,s} (\Lambda ),
\\[1ex]
U_{t,\diamond} &: \ell _{\maclC ,s} (\Lambda )
\times \ell _{\maclB ,s}^\star (\Lambda )
\to \ell _{\maclB ,s}^\star (\Lambda ) &
\quad \text{and}\quad
U_{t,\diamond} &: \ell _{\maclC ,s}^\star (\Lambda )\times
\ell _{\maclB ,s} (\Lambda )
\to
\ell _{\maclB ,s} (\Lambda ).
\end{alignedat}
\end{equation}

\par

If instead $0<s_1,s_2\le \frac 12$, then
the same holds true with
$\ell _{\maclA ,0,s}$, $\ell _{\maclB ,0,s}$
and $\ell _{\maclC ,0,s}$ in place of
$\ell _{\maclA ,s}$, $\ell _{\maclB ,s}$
and $\ell _{\maclC ,s}$, respectively, at
each occurrence.
\end{prop}

\par

\begin{prop}\label{Prop:RingModPropAnSpacesHelp2}
Let $t\in \mathbf C$, $s\in \overline {\mathbf R}_{\flat}$
be such that $s< \frac 12$, $\Lambda =\nn {2d}$,
and let $U_{t,\wpr}$ be the map from
$\ell _0(\Lambda )\times \ell _0(\Lambda )$
to $\ell _0(\Lambda )$, given by $U_{t,\wpr}(c_1,c_2)=
c_1 \bullet _{(t,\wpr )} c_2$, where 
$c_1\bullet _{(t,\wpr )} c_2$ is as in Definition
\ref{Def:BilinearBinomialOps}.
Then $U_{t,\wpr }$ is uniquely extendable to continuous mappings
\begin{equation}
\begin{alignedat}{2}
U_{t,\wpr} &: \ell _{\maclA ,s} (\Lambda )\times \ell _{\maclA ,s} (\Lambda )
\to \ell _{\maclA ,s} (\Lambda ), & &
\\[1ex]
U_{t,\wpr} &: \ell _{\maclA ,s}^\star (\Lambda )\times
\ell _{\maclB ,s} (\Lambda )
\to
\ell _{\maclA ,s} ^\star (\Lambda ), &
\quad
U_{t,\wpr} &: \ell _{\maclB ,s}^\star (\Lambda )\times
\ell _{\maclA ,s}^\star (\Lambda )
\to \ell _{\maclA ,s}^\star (\Lambda ),
\\[1ex]
U_{t,\wpr} &: \ell _{\maclB ,s} (\Lambda  )\times \ell _{\maclB ,s} (\Lambda  )
\to \ell _{\maclB ,s} (\Lambda  ), &
\quad
U_{t,\wpr} &: \ell _{\maclB ,s}^\star (\Lambda  )\times
\ell _{\maclB ,s}^\star (\Lambda  )
\to
\ell _{\maclB ,s} ^\star (\Lambda  ),
\\[1ex]
U_{t,\wpr} &: \ell _{\maclC ,s} (\Lambda  )\times \ell _{\maclC ,s} (\Lambda  )
\to \ell _{\maclC ,s} (\Lambda  ), &
\quad
U_{t,\wpr} &: \ell _{\maclC ,s}^\star (\Lambda  )\times
\ell _{\maclC ,s}^\star (\Lambda  )
\to
\ell _{\maclC ,s} ^\star (\Lambda  ),
\end{alignedat}
\end{equation}

\par

If instead $0<s\le \frac 12$, then
the same holds true with
$\ell _{\maclA ,0,s}$, $\ell _{\maclB ,0,s}$
and $\ell _{\maclC ,0,s}$ in place of
$\ell _{\maclA ,s}$, $\ell _{\maclB ,s}$
and $\ell _{\maclC ,s}$, respectively, at
each occurrence.
\end{prop}

\par

In each one of the propositions above, the different statements
are proved by similar arguments. We only prove selections of
these statements and leave the rest of the verifications for
the reader.

\par

\begin{proof}[Proof of Proposition \ref{Prop:RingModPropAnSpacesHelpAdj}]
We only prove that $U_t$ is uniquely extendable to a continuous map
from $\ell _{\maclB ,s} (\Lambda )\times \ell _{\maclB ,s} (\Lambda )$
to $\ell _{\maclB ,s} (\Lambda )$ and from
$\ell _{\maclC ,s} (\Lambda )\times \ell _{\maclC ,s} (\Lambda )$
to $\ell _{\maclC ,s} (\Lambda )$. The other cases follow by similar
arguments and are left for the reader.

\par

Let $\vartheta _{r,s}$ be as in Definition \ref{DefSeqSpaces},
$\alpha = (\alpha _2,\alpha _1)\in \Lambda$, and
suppose that $c_j\in \ell _0'(\Lambda )$ satisfy
\begin{equation}\label{Eq:cjEstBullett}
|c_j(\alpha )|
=
|c_j(\alpha _2,\alpha _1)|
\lesssim
\frac {\vartheta _{r_j,s_1}(\alpha _1)}{\vartheta _{r_{0,j},s_2}(\alpha _2)},
\end{equation}
for some $r_j,r_{0,j}>0$, $j=1,2$.
By choosing $r=\max (r_1,r_2)>0$
and $r_0=\min (r_{0,1},r_{0,2})>0$, it is clear that \eqref{Eq:cjEstBullett}
holds with $r$ and $r_0$ in place of
$r_j$ and $r_{0,j}$.
By Lemma \ref{Lemma:ThetaWeightBasic} we get
\begin{equation}\label{Eq:cjEstBullettEst}
\begin{aligned}
|(c_1\bullet _tc_2)(\alpha _2,\alpha _1)|
&\le
\sum _{\gamma \le \alpha} {{\alpha}\choose {\gamma}}^{\frac 12}
|t|^{|\gamma |}|c_1(\alpha -\gamma)c_2(\gamma)|
\\[1ex]
&\lesssim
(1+|t|)^{|\alpha |}
\sum _{\gamma \le \alpha} 2^{\frac 12|\alpha |}
\frac {\vartheta _{r,s_1}(\alpha _1-\gamma _1)\vartheta _{r,s_1}(\gamma _1)}
{\vartheta _{r_0,s_2}(\alpha _2-\gamma _2)\vartheta _{r_0,s_2}(\gamma _2)}
\\[1ex]
&\lesssim
(2+2|t|)^{|\alpha |}
\sum _{\gamma \le \alpha}
\frac {\vartheta _{C_0r,s_1}(\alpha _1)}
{\vartheta _{r_0/C_0,s_2}(\alpha _2)}
\\[1ex]
&\lesssim
(3+2|t|)^{|\alpha |}
\frac {\vartheta _{C_0r,s_1}(\alpha _1)}
{\vartheta _{r_0/C_0,s_2}(\alpha _2)}
\lesssim
\frac {\vartheta _{Cr,s_1}(\alpha _1)}
{\vartheta _{r_0/C,s_2}(\alpha _2)},
\end{aligned}
\end{equation}
for some constants $C>C_0> 1$ which are independent of $r$
and $r_0$. In the last step we have used $s_1,s_2<\frac 12$.

\par

If $c_j\in \ell _{\maclB ,s}(\Lambda)$, $j=1,2$, then for every $r=r_j>0$,
there is an $r_0=r_{0,j}>0$ such that \eqref{Eq:cjEstBullett} holds.
By \eqref{Eq:cjEstBullettEst} it follows that for every
$r>0$, there is an $r_0>0$ such that \eqref{Eq:cjEstBullett}
holds with $c_1\bullet _tc_2$ in place of $c_j$. This implies
that $U_t$ is extendable to a continuous map from
$\ell _{\maclB ,s} (\Lambda )\times \ell _{\maclB ,s} (\Lambda )$
to $\ell _{\maclB ,s} (\Lambda )$. Since $\ell _0(\Lambda )$
is dense in $\ell _{\maclB ,s} (\Lambda )$, it also follows that
the continuous extension of $U_t$ is unique.

\par

If $c_j\in \ell _{\maclC ,s}(\Lambda)$, $j=1,2$, then there is an $r_0=r_{0,j}>0$
such that \eqref{Eq:cjEstBullett} holds for every $r=r_j>0$.
By \eqref{Eq:cjEstBullettEst} it follows that there is an $r_0>0$ such that \eqref{Eq:cjEstBullett} holds
for every $r>0$, with $c_1\bullet _tc_2$ in place of $c_j$. Hence
$U_t$ extends to a continuous map from
$\ell _{\maclC ,s} (\Lambda )\times \ell _{\maclC ,s} (\Lambda )$
to $\ell _{\maclC ,s} (\Lambda )$. Since $\ell _0(\Lambda )$
is dense in $\ell _{\maclC ,s} (\Lambda )$, it also follows that
$U_t$ on $\ell _{\maclC ,s} (\Lambda )$ is uniquely defined,
which gives the assertion.
\end{proof}

\par

\begin{proof}[Proof of Proposition \ref{Prop:RingModPropAnSpacesHelp}]
We only prove that $U=U_{t,\diamond}$ extends uniquely to a continuous map
from $\ell _{\maclC ,s}^\star(\Lambda )\times \ell _{\maclB ,s}(\Lambda )$
to $\ell _{\maclB ,s}(\Lambda )$. The other
assertions follow by similar arguments and are left for the
reader.

\par

Since $\ell _{\maclC ,s}^\star(\Lambda )$ is invariant under the map
$$
c(\alpha _2,\alpha _1)\mapsto t^{|\alpha _1|+|\alpha _2|}c(\alpha _2,\alpha _1),
$$
it suffices to prove the result for $t=1$.

\par

First we prove the continuity extension of $U$. Let
$c_1\in \ell _{\maclC ,s}^\star (\Lambda )$
and
$c_2\in \ell _{\maclB ,s} (\Lambda )$. Then it
follows that for some $r_0>0$ one has
\begin{alignat}{2}
|c_1(\alpha _2,\alpha _1)|
&\lesssim
\frac {\vartheta _{r,s_2}(\alpha _2)}{\vartheta _{r_0,s_1}(\alpha _1)}, &
\quad
(\alpha _2,\alpha _1) &\in \Lambda ,
\label{Eq:Condc1}
\intertext{for every $r>0$. It also follows that for every $r>0$,
there is an $r_0>0$ such that}
|c_2(\alpha _2,\alpha _1)|
&\lesssim
\frac {\vartheta _{r,s_1}(\alpha _1)}{\vartheta _{r_0,s_2}(\alpha _2)}, &
\quad
(\alpha _2,\alpha _1) &\in \Lambda .
\label{Eq:Condc2}
\end{alignat}

\par

Let $r_{0,1}>0$ be fixed such that \eqref{Eq:Condc1} holds
for every $r>0$, with $r_{0,1}$ in place of $r_0$. Also let
$C>1$ be a constant which only dependent on $s_1,s_2$,
let $r_2\in (0,r_{0,1}/C)$ be arbitrary and choose $r_{0,2}>0$
such that \eqref{Eq:Condc2} holds with $r_2$ and $r_{0,2}$
in place of $r$ and $r_0$. Finally let $r_1\in (0,r_{0,2}/C)$
be arbitrary.

\par

If $C$ is chosen large enough, and using the same notations
as in Lemma \ref{Lemma:WickCoeffComp}, then it follows from Lemma
\ref{Lemma:ThetaWeightBasic} that
\begin{align*}
{{\alpha +\gamma}\choose {\gamma}}^{\frac 12}
|c_1(\gamma )c_2(\alpha +\gamma)|
 &\lesssim
2^{|\alpha +\gamma |}\frac
{\vartheta _{r_1,s_2}(\gamma _2)\vartheta _{r_2,s_1}(\alpha _1+\gamma _1)}
{\vartheta _{r_{0,1},s_1}(\gamma _1)
\vartheta _{r_{0,2},s_2}(\alpha _2+\gamma _2)}
\\[1ex]
&\lesssim
2^{-|\alpha +\gamma |}\frac
{\vartheta _{Cr_2,s_1}(\alpha _1)}
{\vartheta _{r_{0,2}/C,s_2}(\alpha _2)}.
\end{align*}
Hence for the series on the right-hand side of
\eqref{Eq:BilinearBinomOp} we have
\begin{align*}
\sum _{\gamma \in \Lambda}
{{\alpha +\gamma }\choose {\gamma }} ^{\frac 12}
|c_1(\gamma _2,\gamma _1)
c_2(\alpha _2+\gamma _2,\alpha _1+\gamma _1)|
&\lesssim
\sum _{\gamma \in \Lambda} 
2^{-|\alpha +\gamma |}\frac
{\vartheta _{Cr_2,s_1}(\alpha _1)}
{\vartheta _{r_{0,2}/C,s_2}(\alpha _2)}
\\[1ex]
&\asymp
\frac {\vartheta _{Cr_2,s_1}(\alpha _1)}
{\vartheta _{r_{0,2}/C,s_2}(\alpha _2)} .
\end{align*}

\par

It follows that the right-hand side of
\eqref{Eq:BilinearBinomOp} is convergent, and by defining
$c_1\bullet _{(1,\diamond )}c_2$ by \eqref{Eq:BilinearBinomOp},
it follows that for every $r>0$, there is an $r_0>0$ such that
\begin{equation}\label{Eq:Condc1c2}
|(c_1\bullet _{(1,\diamond )}c_2)(\alpha )|
\lesssim 
\frac {\vartheta _{r,s_1}(\alpha _1)}
{\vartheta _{r_0,s_2}(\alpha _2)} .
\end{equation}
This implies that $c_1\bullet _{(1,\diamond )}c_2\in
\ell _{\maclB ,s} (\Lambda )$. Furthermore, the involved
parameters of the right-hand side of \eqref{Eq:Condc1c2}
depend continuously on the involved parameters of the right-hand
sides of \eqref{Eq:Condc1} and \eqref{Eq:Condc2}, which
proves the continuity assertion.

\par

Finally, the uniqueness of the extension follows from
the fact that $\ell _0(\Lambda )$ is dense in
$\ell _{\maclB ,s} (\Lambda )$
and
$\ell _{\maclC ,s}^\star (\Lambda )$,
and the result follows.
\end{proof}

\par

\begin{proof}[Proof of Proposition \ref{Prop:RingModPropAnSpacesHelp2}]
We only prove that $U=U_{t,\wpr}$ extends uniquely to continuous products
on $\ell _{\maclB ,s}(\nn {2d} )$ and on $\ell _{\maclC ,s}(\nn {2d} )$.
The other assertions follows by similar arguments and are left for the
reader.

\par

Let
$$
C_0(\alpha ,\alpha _0, \beta ,\beta _0,\gamma )
=
\left (
{{\alpha }\choose {\alpha _0}}{{\beta }\choose {\beta _0}}
{{\alpha -\alpha _0+\gamma}\choose {\gamma}}
{{\beta -\beta _0+\gamma}\choose {\gamma}}
\right )^{\frac 12},
$$
when
 $\alpha ,\alpha _0,\beta ,\beta _0,\gamma \in \nn d$
 satisfy $\alpha _0\le \alpha$ and $\beta _0\le \beta$
(cf. $C_2(\alpha ,\beta ,\gamma )$ in the proof of
Lemma \ref{Lemma:WickCoeffComp}).
We observe
\begin{equation}
|C_0(\alpha ,\alpha _0, \beta ,\beta _0,\gamma )|
\le
2^{|\alpha +\beta +\gamma |}.
\end{equation}
Suppose that $c_j\in \ell _{\maclB ,s}(\nn {2d})$, $j=1,2$. Then it
follows from Lemma \ref{Lemma:ThetaWeightBasic}
and \eqref{Eq:cjEstBullett} for $s_j=s$
that for every $r_j>0$, there is an $r_{0,j}>0$ such that
\begin{align}
|c_1(\alpha _0,\beta -\beta _0+\gamma )|
&\lesssim
\frac
{\vartheta _{r_1,s}(\beta -\beta _0)\vartheta _{r_1,s}(\gamma )}
{\vartheta _{r_{0,1}}(\alpha _0)}
\label{Eq:cjEstWpr1}
\intertext{and}
|c_2(\alpha -\alpha _0+\gamma ,\beta _0)|
&\lesssim
\frac
{\vartheta _{r_2,s}(\beta _0)}
{\vartheta _{r_{0,2},s}(\alpha -\alpha _0)\vartheta _{r_{0,2},s}(\gamma )}.
\label{Eq:cjEstWpr2}
\end{align}

\par

Let $r>0$ be arbitrary, $C>1$ be a constant which is independent of
$r$ which is specified later and let $r_2>0$ be such that $Cr_2<r$. Choose
$r_{0,2}>0$ such that \eqref{Eq:cjEstWpr2} holds. Now choose
$r_1>0$ such that $r_1<r_2$ and $r_1<r_{0,2}/2$ and then choose
$r_{0,1}>0$ such that $r_{0,1}<r_{0,2}$
and \eqref{Eq:cjEstWpr1} hols.

\par

If $C$ is chosen large enough, then it follows from Lemma
\ref{Lemma:ThetaWeightBasic} again that
\begin{multline}\label{Eq:CoeffWprEst}
|c_1(\alpha _0,\beta -\beta _0+\gamma )
c_2(\alpha -\alpha _0+\gamma ,\beta _0)|
\lesssim
\frac
{\vartheta _{r_1,s}(\beta -\beta _0)\vartheta _{r_1,s}(\gamma )
\vartheta _{r_2,s}(\beta _0)}
{\vartheta _{r_{0,1},s}(\alpha _0)\vartheta _{r_{0,2},s}(\alpha -\alpha _0)
\vartheta _{r_{0,2},s}(\gamma )}
\\[1ex]
\lesssim
\frac
{2^{-|\alpha +\beta |}2^{-|\alpha _0+\beta _0|}\vartheta _{Cr_2,s}(\beta )}
{\vartheta _{r_{0,1}/C,s}(\alpha )\vartheta _{r_1,s}(\gamma )}
\lesssim
\frac
{2^{-|\alpha +\beta |}2^{-|\alpha _0+\beta _0|}\vartheta _{r,s}(\beta )}
{\vartheta _{r_0,s}(\alpha )\vartheta _{r_1,s}(\gamma )},
\end{multline}
when $r_0=r_{0,1}/C>0$.
By combining these estimates we obtain
\begin{equation}\label{Eq:CoeffWprEstFinal}
|(c_1 \bullet _{(t,\wpr )} c_2)(\alpha ,\beta )|
\lesssim
\left (
\sum _{\gamma} \frac {(2+2|t|)^{|\gamma |}}{\vartheta _{r_1,s}(\gamma )}
\right )
\left (
\sum _{\alpha _0,\beta _0} 2^{-|\alpha _0+\beta _0|}
\frac {\vartheta _{r,s}(\beta )}{\vartheta _{r_0,s}(\alpha )}
\right )
\asymp
\frac {\vartheta _{r,s}(\beta )}{\vartheta _{r_0,s}(\alpha )},
\end{equation}
and it follows that $U$ on $\ell _0(\nn {2d})$
extends to a continuous products
on $\ell _{\maclB ,s}(\nn {2d} )$.

\par

Suppose instead that
$c_1\in \ell _{\maclC ,s}(\nn {2d} )$, $c_2\in \ell _{\maclC ,s}(\nn {2d} )$.
Then for some $r_{0,1},r_{0,2}>0$ such that $r_{0,1}<r_{0,2}$ one has that
\eqref{Eq:cjEstWpr1} and \eqref{Eq:cjEstWpr2} hold for every $r_1,r_2>0$
such that $r_1<r_2$. Hence, if $C$ and $r_0$ are chosen as in the first part
of the proof, then \eqref{Eq:CoeffWprEst} and \eqref{Eq:CoeffWprEstFinal}
show that for some $r_0>0$ one has that
$$
|(c_1 \bullet _{(t,\wpr )} c_2)(\alpha ,\beta )|
\lesssim
\frac {\vartheta _{r,s}(\beta )}{\vartheta _{r_0,s}(\alpha )},
$$
for every $r>0$. This implies that $U$ on $\ell _0(\nn {2d})$
extends to a continuous products on $\ell _{\maclC ,s}(\nn {2d} )$.
The uniqueness assertions follow from the fact that $\ell _0(\nn {2d})$
is dense in $\ell _{\maclB ,s}(\nn {2d} )$ and in
$\ell _{\maclC ,s}(\nn {2d} )$.
\end{proof}

\par

\section{Continuity, composition and transitions of
Wick and anti-Wick operators}\label{sec4}

\par

In this section we apply the results from the previous sections to deduce
estimates on kernels, Wick and anti-Wick symbols of operators. Especially
we show that in several situations, any linear and continuous operator
on $\maclA _{0,s}(\cc d)$, $\maclA _s(\cc d)$, $\maclA _s'(\cc d)$
or on $\maclA _{0,s}'(\cc d)$ can be expressed as Wick or anti-Wick
operators (see Theorems \ref{Thm:OpReprIdentB} and
\ref{Thm:OpReprIdentC}). In the last part of the section we complete
some analysis in Section 3 in \cite{TeToWa}. Here we deduce some
estimates of Wick symbols of anti-Wick operators with Wick or
anti-Wick symbols belonging to Hilbert spaces, related to
$\wideparen \maclA _{0,1/2}'(\cc {2d})$.

\par

\subsection{Continuity of the operators $\maclT _t$, $\maclT _t^*$}

\par

We start by discussing the operator $\maclT _t$.
A combination of Proposition \ref{Prop:T0thomeomorphismEllC} (1) and
\eqref{Eq:CommDiagram} gives the following result,
related to Proposition \ref{Prop:TransKernelWickASpaces}.
The details are left for the reader.

\par

\begin{thm}\label{Thm:TthomeomorphismEllC}
Let $t \in \mathbf C$, $ s,s_0 \in \overline{\mathbf R_\flat}$ be such that
$ s < \frac 12$ and $s_0 \leq \frac 12$,
$\maclT _{0,t} $ be the map on $ \ell _0' (\nn {2d}) $  given by
\eqref{Eq:T_{0,t}Def},
and let $\maclT _t$ be the map on $\wideparen \maclA _0' (\cc {2d}) $  given by
\eqref{Eq:TtDef}. Then the following is true:
\begin{enumerate}
\item $\maclT _t$ restricts to homeomorphisms on each of the spaces
\begin{equation}\label{Eq:TthomeomorphismEllC}
\begin{alignedat}{4}
&\wideparen \maclB _{0,s_0}(\cc {2d}), &
\quad
&\wideparen \maclB _s(\cc {2d}), &
\quad
&\wideparen \maclB _{0,s_0}^\star (\cc {2d}) &
\quad \text{and}\quad
&\wideparen \maclB _s^\star (\cc {2d})
\text ;
\end{alignedat}
\end{equation}

\vrum

\item the diagram \eqref{Eq:CommDiagram} commutes, after $\ell _s'$ and
$\wideparen \maclA _s'$ are replaced by $\ell _{\maclB ,s}$ and
$\wideparen \maclB _s$, by $\ell _{\maclB ,0,s_0}$ and
$\wideparen \maclB _{0,s_0}$, by $\ell _{\maclB ,s}^\star$ and
$\wideparen \maclB _s^\star$, or by $\ell _{\maclB ,0,s_0}^\star$ and
$\wideparen \maclB _{0,s_0}^\star$, respectively, at each occurrence.
\end{enumerate}
\end{thm}

\par

Next we deduce extensions of the operator $\maclT _t^*$,
and begin with the following, which follows from
(2) in Proposition \ref{Prop:T0tBasichomeomorphism} and the commutative
diagram \eqref{Eq:CommDiagramDual}. The details are left for the reader.

\par

\begin{thm}\label{Thm:BasicTtStarCont}
Let $t \in \mathbf C$, $ s,s_0 \in \overline{\mathbf R_\flat}$ be such that
$ s < \frac 12$ and $s_0 \leq \frac 12$, and let
$\maclT _{0,t}^*$ be the map on $\ell _0(\nn {2d})$ given by
\eqref{Eq:T_{0,t}Dual} and let $\maclT _t^*$ be the map on
$\wideparen \maclA _0(\cc {2d})$ given by \eqref{Eq:Tt*Def}.
Then $\maclT _t^*$ is uniquely extendable
to homeomorphisms on $\wideparen \maclA _s(\cc {2d}) $ and
on $\wideparen \maclA _{0,s_0} (\cc {2d})$. The diagram
\eqref{Eq:CommDiagramDual} commutes after $\ell _0$ and $\wideparen A_0$
have been replaced by $\ell _s$ and $\wideparen A_s$, respectively, or by
$\ell _{0,s_0}$ and $\wideparen A_{0,s_0}$, respectively, at each occurrence.
\end{thm}

\par

In the same way we get the following by combining
Proposition \ref{Prop:T0thomeomorphismEllC} (2) and
\eqref{Eq:CommDiagramDual}. Again the details are left for the reader.

\par

\begin{thm}\label{Thm:TtDualhomeomorphismEllC}
Let $t \in \mathbf C$, $ s,s_0 \in \overline{\mathbf R_\flat}$ be such that
$ s < \frac 12$ and $s_0 \leq \frac 12$,
$\maclT _{0,t}^*$ be the map on $ \ell _0(\nn {2d}) $  given by
\eqref{Eq:T_{0,t}Dual},
and let $\maclT _t^*$ be the map on $\wideparen \maclA _0(\cc {2d}) $  given by
\eqref{Eq:Tt*Def}. Then the following is true:
\begin{enumerate}
\item $\maclT _t^*$ is uniquely extendable to homeomorphisms on
each of the spaces
\begin{equation}\label{Eq:TthomeomorphismEllC2}
\begin{alignedat}{4}
&\wideparen \maclC _{0,s_0}(\cc {2d}), &
\quad
&\wideparen \maclC _s(\cc {2d}), &
\quad
&\wideparen \maclC _{0,s_0}^\star (\cc {2d}) &
\quad \text{and}\quad
&\wideparen \maclC _s^\star (\cc {2d})\text ;
\end{alignedat}
\end{equation}

\vrum

\item the diagram \eqref{Eq:CommDiagramDual} commutes, after $\ell _0$ and
$\wideparen \maclA _0$ are replaced by $\ell _{\maclC ,s}$ and
$\wideparen \maclC _s$, by $\ell _{\maclC ,0,s_0}$ and
$\wideparen \maclC _{0,s_0}$, by $\ell _{\maclC ,s}^\star$ and
$\wideparen \maclC _s^\star$, or by $\ell _{\maclC ,0,s_0}^\star$ and
$\wideparen \maclC _{0,s_0}^\star$, respectively, at each occurrence.
\end{enumerate}
\end{thm}

\par

\subsection{Continuity and relationships between kernel, Wick
and anti-Wick operators}

\par

By
\eqref{Eq:AntiWickAnalPseudoRel}--\eqref{Eq:AntiWickAnalPseudoRelIntForm2},
Proposition \ref{Prop:T0thomeomorphismEllC},
Theorem \ref{Thm:TtDualhomeomorphismEllC},
and the fact that \eqref{Eq:AntiWicktoWickRevised} holds for
$a\in \wideparen \maclA _0(\cc {2d})$, we may now extend
the domain of anti-Wick operators as in the following definition. 

\par

\begin{defn}\label{Def:AntiWickDefExt}
Let $s_1,s_2\in \overline {\mathbf R}_\flat$ be such that $s_1<\frac 12$
and $s_2\le \frac 12$.
\begin{enumerate}
\item If $a\in \wideparen \maclC _{s_1}(\cc {2d})$
($a\in \wideparen \maclC _{s_1}^\star (\cc {2d})$), then the anti-Wick operator
$\op _{\mathfrak V}^{\aw}(a)$ is the linear and continuous operator
on $\maclA _{s_1}(\cc d)$ (on $\maclA _{s_1}'(\cc d)$), given by
\eqref{Eq:AntiWicktoWickRevised}.

\vrum

\item If $a\in \wideparen \maclC _{0,s_2}(\cc {2d})$
($a\in \wideparen \maclC _{0,s_2}^\star (\cc {2d})$), then the anti-Wick operator
$\op _{\mathfrak V}^{\aw}(a)$ is the linear and continuous operator
on $\maclA _{0,s_2}(\cc d)$ (on $\maclA _{0,s_2}'(\cc d)$), given by
\eqref{Eq:AntiWicktoWickRevised}.
\end{enumerate}
\end{defn}

\par

We can now combine the results in previous sections
with the kernel results
in Section \ref{sec2} to find equalities between classes of
Wick, anti-Wick operators and kernel operators. As a first taste we have
the following result which follows by a combination of
\eqref{Eq:KernelWickRelByT}, and
Propositions \ref{Prop:KernelsDifficultDirectionExt},
\ref{Prop:KernelsEasyDirectionExt} and
\ref{Prop:TransKernelWickASpaces}.
The details are left for the reader. The
result is also an immediate consequence
of Theorems 2.7 and 2.8 in \cite{Teofanov2}.

\par

\begin{prop}\label{Prop:KernelOpWickEqual}
Let $j\in \{ 1,2\}$, $s_1,s_2\in \overline{\mathbf R}_\flat$ be such that
$s_1<\frac 12$ and
$s_2\le \frac 12$. Then the mappings
\begin{alignat*}{3}
\wideparen \maclA _{s_1}'(\cc {2d}) &\ni a &
\ &\mapsto \ &\op _{\mathfrak V}(a)
&\in \maclL (\maclA _{s_1}(\cc d),\maclA _{s_1}'(\cc d))
\intertext{and}
\wideparen \maclA _{0,s_2}'(\cc {2d}) &\ni a &
\ &\mapsto \ &\op _{\mathfrak V}(a)
&\in \maclL (\maclA _{0,s_2}(\cc d),\maclA _{0,s_2}'(\cc d))
\end{alignat*}
are isomorphisms.
\end{prop}

\par

The corresponding result for anti-Wick operators is the following.
The result follows by combining \eqref{Eq:AntiWicktoWickRevised}
and Theorem \ref{Thm:BasicTtStarCont}.
The details are left for the reader.

\par

\begin{thm}\label{Thm:antiWickWickEqual}
Let $j\in \{ 1,2\}$, $s_1,s_2\in \overline{\mathbf R}_\flat$ be such that
$s_1<\frac 12$ and
$s_2\le \frac 12$. Then the mappings
\begin{alignat}{3}
\wideparen \maclA _{s_1}(\cc {2d}) &\ni a &
\ &\mapsto \ &\op _{\mathfrak V}^{\aw}(a)
&\in 
\op _{\mathfrak V}(\wideparen \maclA _{s_1}(\cc {2d}))
\label{Eq:antiWickWickEqual}
\intertext{and}
\wideparen \maclA _{0,s_2}(\cc {2d}) &\ni a &
\ &\mapsto \ &\op _{\mathfrak V}^{\aw}(a)
&\in 
\op _{\mathfrak V}(\wideparen \maclA _{0,s_2}(\cc {2d}))
\end{alignat}
are isomorphisms.
\end{thm}

\par

We notice that if $s_1=0$, then $\wideparen \maclA _{s_1}(\cc {2d})
=\wideparen \maclA _0(\cc {2d})$ is the set of all polynomials $a(z,w)$ which
are analytic in $z$ and conjugated analytic in $w$. Hence
\eqref{Eq:antiWickWickEqual} in this case means that the sets of
Wick operators with polynomial symbols agree with the set of anti-Wick
operators with polynomial symbols, which is well-known
(see e.{\,}g. \cite{Berezin71}). On the other
hand, the other cases in Theorem \ref{Thm:antiWickWickEqual},
seems not to be
available in the literature.

\par

\begin{rem}
Note that $\wideparen \maclA _{s_1}(\cc {2d})$ and
$\wideparen \maclA _{s_2}(\cc {2d})$ in Theorem
\ref{Thm:antiWickWickEqual} are images of those
Pilipovi{\'c} spaces, which are not Gelfand-Shilov spaces, under the map
$\Theta \circ \mathfrak V_{2d}$. In particular, 
Theorem \ref{Thm:antiWickWickEqual} gives further motivations
for detailed studies of Pilipovi{\'c} spaces.
\end{rem}

\par

The next two results follows from
\eqref{Eq:KernelWickRelByT}, \eqref{Eq:AntiWicktoWickRevised},
Propositions \ref{Prop:KernelsDifficultDirectionNew} and
\ref{Prop:KernelsEasyDirectionNew}, and
Theorems \ref{Thm:TthomeomorphismEllC} and
\ref{Thm:TtDualhomeomorphismEllC}.
The details are left for the reader.

\par

\begin{thm}\label{Thm:OpReprIdentB}
Let $s_1,s_2 \in \overline{\mathbf R_\flat}$ be such that
$s_1 < \frac 12$ and $s_2 \leq \frac 12$. Then the mappings
\begin{alignat*}{4}
\wideparen \maclB _{s_1}(\cc {2d}) &\ni a
&\ &\mapsto \ &
\op _{\mathfrak V}(a)
&\in
\maclL (\maclA _{s_1}(\cc d)),
\\[1ex]
\wideparen \maclB _{s_1}^\star (\cc {2d}) &\ni a
&\ &\mapsto \ &
\op _{\mathfrak V}(a)
&\in
\maclL (\maclA _{s_1}'(\cc d)),
\\[1ex]
\wideparen \maclB _{0,s_2}(\cc {2d}) &\ni a
&\ &\mapsto \ &
\op _{\mathfrak V}(a)
&\in
\maclL (\maclA _{0,s_2}(\cc d))
\intertext{and}
\wideparen \maclB _{0,s_2}^\star (\cc {2d}) &\ni a
&\ &\mapsto \ &
\op _{\mathfrak V}(a)
&\in
\maclL (\maclA _{0,s_2}'(\cc d)),
%
%
%
\end{alignat*}
are isomorphisms.
\end{thm}

\par

\begin{thm}\label{Thm:OpReprIdentC}
Let $s_1,s_2 \in \overline{\mathbf R_\flat}$ be such that
$s_1 < \frac 12$ and $s_2 \leq \frac 12$. Then the mappings
\begin{alignat*}{3}
\wideparen \maclC _{s_1}(\cc {2d}) &\ni a &
\ &\mapsto \ &\op _{\mathfrak V}^{\aw}(a)
&\in 
\op _{\mathfrak V}(\wideparen \maclC _{s_1}(\cc {2d})),
\\[1ex]
\wideparen \maclC _{s_1}^\star (\cc {2d}) &\ni a &
\ &\mapsto \ &\op _{\mathfrak V}^{\aw}(a)
&\in 
\op _{\mathfrak V}(\wideparen \maclC _{s_1}^\star (\cc {2d})),
\\[1ex]
\wideparen \maclC _{0,s_2}(\cc {2d}) &\ni a &
\ &\mapsto \ &\op _{\mathfrak V}^{\aw}(a)
&\in 
\op _{\mathfrak V}(\wideparen \maclC _{0,s_2}(\cc {2d}))
\intertext{and}
\wideparen \maclC _{0,s_2}^\star (\cc {2d}) &\ni a &
\ &\mapsto \ &\op _{\mathfrak V}^{\aw}(a)
&\in 
\op _{\mathfrak V}(\wideparen \maclC _{0,s_2}^\star (\cc {2d}))
\end{alignat*}
are isomorphisms.
\end{thm}

\par

\begin{rem}\label{Rem:OpReprIdentC}
Let $s\in \overline {\mathbf R}_\flat$ be such that $s_1<\frac 12$
and $s_2\le \frac 12$.
By the preparing results which lead to Definition \ref{Def:AntiWickDefExt},
and Theorem \ref{Thm:OpReprIdentC}, it follows that
\begin{equation}\label{Eq:SetsAntiWickWickIdent}
\begin{aligned}
\sets {\op _{\mathfrak V}^{\aw}(a)}{a\in \wideparen \maclC _{s_1}(\cc {2d})}
&=
\sets {\op _{\mathfrak V}(a)}{a\in \wideparen \maclC _{s_1}(\cc {2d})},
\\[1ex]
\sets {\op _{\mathfrak V}^{\aw}(a)}{a\in \wideparen \maclC _{s_1}^\star (\cc {2d})}
&=
\sets {\op _{\mathfrak V}(a)}{a\in \wideparen \maclC _{s_1}^\star(\cc {2d})},
\\[1ex]
\sets {\op _{\mathfrak V}^{\aw}(a)}{a\in \wideparen \maclC _{0,s_2}(\cc {2d})}
&=
\sets {\op _{\mathfrak V}(a)}{a\in \wideparen \maclC _{0,s_2}(\cc {2d})},
\\[1ex]
\sets
{\op _{\mathfrak V}^{\aw}(a)}{a\in \wideparen \maclC _{0,s_2}^\star (\cc {2d})}
&=
\sets {\op _{\mathfrak V}(a)}{a\in \wideparen \maclC _{0,s_2}^\star(\cc {2d})},
\end{aligned}
\end{equation}
and that
\begin{equation}\label{}
\op _{\mathfrak V}^{\aw}(a_1)\neq \op _{\mathfrak V}^{\aw}(a_2)
\quad \Leftrightarrow \quad
\op _{\mathfrak V}(a_1)\neq \op _{\mathfrak V}(a_2)
\quad \Leftrightarrow \quad
a_1\neq a_2,
\end{equation}
when $a_1$ and $a_2$ belong to any of the spaces in
Theorem \ref{Thm:OpReprIdentC}.
\end{rem}

\par

In Appendix \ref{App:C} we have listed some identities of
spaces of kernels, Wick and anti-Wick operators which are
immediate consequences of Corollary
\ref{Cor:KernelLinOpChar}, Proposition
\ref{Prop:KernelOpWickEqual} and Theorems
\ref{Thm:antiWickWickEqual}--\ref{Thm:OpReprIdentC}.

\par

For other choices of $s_1$ and $s_2$
it seems that the equalities on the left-hand sides
in Theorems \ref{Thm:OpReprIdentB} and \ref{Thm:OpReprIdentC}
between sets of Wick and anti-Wick operators are violated. 
On the other hand, the following result shows that we still may identify
the operator classes on the right-hand sides in \eqref{Thm:OpReprIdentB}
with suitable classes of Wick operators, with
Wick symbols bounded by
\begin{align}
\vartheta _{1,s,r_1,r_2}(z,w)
=
\begin{cases}
e^{\frac 12|z-w|^2-r_1|z|^{\frac 1s}+r_2|w|^{\frac 1s}}, & s<\infty ,
\\[1ex]
e^{\frac 12|z-w|^2}\eabs z^{-r_1}\eabs w^{r_2}, & s=\infty ,
\end{cases}
\label{Eq:CondWickGS1}
\intertext{and}
\vartheta _{2,s,r_1,r_2}(z,w)
=
\begin{cases}
e^{\frac 12|z-w|^2+r_2|z|^{\frac 1s}-r_1|w|^{\frac 1s}}, & s<\infty ,
\\[1ex]
e^{\frac 12|z-w|^2}\eabs z^{r_2}\eabs w^{-r_1}, & s=\infty .
\end{cases}
\label{Eq:CondWickGS2}
\end{align}

\par

\begin{thm}\label{Thm:OpWickGSCase}
Let $s\in [\frac 12 ,\infty ]$ and $\vartheta _{k,s,r_1,r_2}$
be given by \eqref{Eq:CondWickGS1} and \eqref{Eq:CondWickGS2}, $k=1,2$.
Then the following is true:
\begin{enumerate}
\item if $s <\infty$, then $\maclL (\maclA _s(\cc d))$
($\maclL (\maclA _s'(\cc d))$)
consists of all $\op _{\mathfrak V}(a)$ such that $a\in \wideparen A(\cc {2d})$
and for every $r_2>0$, there is an $r_1>0$ such that
\begin{equation}\label{Eq:OpWickGSCase}
|a(z,w)|\lesssim \vartheta _{1,s,r_1,r_2}(z,w)
\qquad
\text ( \, |a(z,w)|\lesssim \vartheta _{2,s,r_1,r_2}(z,w) \, \text )
\end{equation}
holds;

\vrum

\item if $s>\frac 12$, then $\maclL (\maclA _{0,s}(\cc d))$
($\maclL (\maclA _{0,s}'(\cc d))$)
consists of all $\op _{\mathfrak V}(a)$ such that $a\in \wideparen A(\cc {2d})$
and for every $r_1>0$, there is an $r_2>0$ such that
\eqref{Eq:OpWickGSCase} holds.
\end{enumerate}
\end{thm}

\par

\begin{proof}
We only prove (1). The assertion (2) follows by similar arguments and is
left for the reader.

\par

Suppose that $T$ is a linear and continuous map from $\maclA _{\flat _1}(\cc d)$
to $\maclA _{\flat _1}'(\cc d)=A(\cc d)$ with kernel $K$ and Wick symbol $a$. Then
$a,K \in \wideparen \maclA _{\flat _1}'(\cc {2d})$. By Propositions
\ref{Prop:KernelsDifficultDirectionNew} and \ref{Prop:KernelsEasyDirectionNew}
it follows that $T\in \maclL (\maclA _s(\cc d))$ ($T\in \maclL (\maclA _s'(\cc d))$),
if and only if
for every $r_2>0$, there is an $r_1>0$ such that
$$
|K(z,w)| \lesssim e^{\frac 12(|z|^2+|w|^2)-r_1|z|^{\frac 1s}+r_2|w|^{\frac 1s}}
\quad
\big ( |K(z,w)| \lesssim e^{\frac 12(|z|^2+|w|^2)+r_2|z|^{\frac 1s}-r_1|w|^{\frac 1s}}
\big ).
$$
The asserted equivalence now follows from the previous equivalence and the fact that
$a(z,w)=K(z,w)e^{-(z,w)}$, which implies that
$$
|a(z,w)e^{-\frac 12|z-w|^2}| = |K(z,w)e^{-\frac 12(|z|^2+|w|^2)}|. \qedhere
$$
\end{proof}

\par

\subsection{Multiplications, differentiations
for power series expansions, and compositions of Wick operators}

\par

First we combine \eqref{Eq:MultCoeffRel}
and Proposition \ref{Prop:RingModPropAnSpacesHelpAdj}
to deduce ring structures of the spaces in \eqref{Eq:SesAnalSp},
\eqref{Eq:SesAnalSp2} and in Definition \ref{Def:tauSpaces2}.
The details are left for the reader.

\par

\begin{thm}\label{Thm:RingPropAnSpaces}
Let $s_1,s_2\in \overline {\mathbf R}_{\flat}$, $s=(s_1,s_2)$ and
$W=\cc {d_2}\times \cc {d_1}$. Then the following is true:
\begin{enumerate}
\item if $s_1,s_2<\frac 12$, then the ring $\wideparen \maclA _0(W)$
under addition and multiplication extends uniquely to topological
rings
\begin{alignat}{7}
& \wideparen \maclA _{s}(W), &
\quad
& \wideparen \maclA _{s}^\star (W), &
\quad
& \wideparen \maclB _{s}(W), &
\quad
& \wideparen \maclB _{s}^\star (W), &
\quad
& \wideparen \maclC _{s}(W), &
\quad & \text{and} & \quad
&\wideparen \maclC _{s}^\star (W),
\label{Eq:DualtauRoumSpaces}
\end{alignat}
under addition and multiplication;

\vrum

\item if $0<s_1,s_2\le \frac 12$, then the ring $\wideparen \maclA _0(W)$
under addition and multiplication extends uniquely to topological
rings
\begin{alignat}{7}
&\wideparen \maclA _{0,s}(W), &
\quad
&\wideparen \maclA _{0,s}^\star (W), &
\quad
&\wideparen \maclB _{0,s}(W), &
\quad
&\wideparen \maclB _{0,s}^\star (W), &
\quad
&\wideparen \maclC _{0,s}(W) &
\quad &\text{and} &\quad
\quad
&\wideparen \maclC _{0,s}^\star (W),
\label{Eq:DualtauBeurSpaces}
\end{alignat}
under addition and multiplication.
\end{enumerate}
\end{thm}

\par

\begin{rem}
In Section \ref{sec6} we present trace results and
results on linear pullbacks for the spaces in
\eqref{Eq:DualtauRoumSpaces}
and \eqref{Eq:DualtauBeurSpaces}.
The ring assertions in Theorem \ref{Thm:RingPropAnSpaces}
also follows by combining such trace and pullback properties. 
\end{rem}

\par

From now on we let the spaces in \eqref{Eq:DualtauRoumSpaces}
and in \eqref{Eq:DualtauBeurSpaces}
be equipped with the ring structure guaranteed by
Theorem \ref{Thm:RingPropAnSpaces}.

\medspace

Next we discuss extensions of the product $\diamond$ 
on $\wideparen \maclA _0(W)$, given in \eqref{Eq:MonomMultAdjoint3}.
The following result is a straight-forward consequence of
\eqref{Eq:MultDiamCoeffRel} and Proposition
\ref{Prop:RingModPropAnSpacesHelp}. The details are left for the reader.

\par

\begin{thm}\label{Thm:RingModPropAnSpaces}
Let $s_1,s_2\in \overline {\mathbf R}_{\flat}$, $s=(s_1,s_2)$,
$W=\cc {d_2}\times \cc {d_1}$
and let $T_\diamond$ be the map from
$\maclA _0(\cc d)\times \maclA _0(\cc d)$
to $\maclA _0(\cc d)$, given by
$T_\diamond (K_1,K_2) = K_1\diamond K_2$,
where $K_1\diamond K_2$ is given by \eqref{Eq:MonomMultAdjoint3}.
If $s_1,s_2<\frac 12$, then
$T_\diamond$ is uniquely extendable to continuous mappings
\begin{alignat*}{4}
T_\diamond &:
\wideparen \maclA _{s}(W)\times \wideparen \maclA _{s}^\star (W)
& &\to \wideparen \maclA _{s}^\star (W) &
\qquad
T_\diamond &:
\wideparen \maclA _{s}^\star (W)\times \wideparen \maclA _{s}(W)
& &\to \wideparen \maclA _{s}(W)
\\[1ex]
T_\diamond &:
\wideparen \maclB _{s}(W)
\times \wideparen \maclC _{s}^\star (W)
& &\to \wideparen \maclC _{s}^\star (W) &
\qquad
T_\diamond &:
\wideparen \maclB _{s}^\star (W)
\times \wideparen \maclC _{s}^\star (W)
& &\to \wideparen \maclC _{s}^\star (W)
\\[1ex]
T_\diamond &:
\wideparen \maclC _{s}(W)
\times \wideparen \maclB _{s}^\star (W)
& &\to \wideparen \maclB _{s}^\star (W) &
\quad \text{and} \quad
T_\diamond &:
\wideparen \maclC _{s}^\star (W)
\times \wideparen \maclB _{s}^\star (W)
& &\to \wideparen \maclB _{s}^\star (W).
\end{alignat*}
If instead $0<s_1,s_2\le \frac 12$, then
the same holds true with
$\wideparen \maclA _{0,s}$, $\wideparen \maclB _{0,s}$
and $\wideparen \maclC _{0,s}$ in place of
$\wideparen \maclA _{s}$, $\wideparen \maclB _{s}$
and $\wideparen \maclC _{s}$, respectively, at each occurrence.
\end{thm}

\par

\begin{rem}
Let $d_2=d$, $d_1=0$ and $s\in \overline {\mathbf R} _\flat$
satisfies $s<\frac 12$.
Then it follows from Theorem
\ref{Thm:RingModPropAnSpaces} that
\begin{equation}
T_\diamond :
\maclA _{s}(\cc d)\times \maclA _{s}'(\cc d)
\to \maclA _{s}'(\cc d)
\quad \text{and} \quad
T_\diamond :
\maclA _{s}'(\cc d)\times \maclA _{s}(\cc d)
\to \wideparen \maclA _{s}(\cc d)
\end{equation}
are continuous. If instead $0<s\le \frac 12$, then the same holds
true with $\maclA _{0,s}$ in place of $\maclA _s$ at each
occurrence.
\end{rem}

\par

\begin{rem}
It is tempting to proclaim that Theorem \ref{Thm:RingModPropAnSpaces}
is the dual result to Theorem \ref{Thm:RingPropAnSpaces}, where
the former should follow from the latter by using the adjoint relation
\eqref{Eq:ComplFourMult}, and that the spaces in
\eqref{Eq:ellCparenSpaces1to4} are close to the duals of
\eqref{Eq:ellBparenSpaces1to4} and vise versa.
A problem here is that the literature seems not to support
completeness and thereby ensure useful duality properties for some
of the involved spaces (see e.{\,}g. \cite{SchWol,Vog,Wen}).
It therefore seems not so straight-forward to apply duality arguments
here.
\end{rem}

\par

Next we discuss extensions of the twisted product $\wpr _{\mathfrak V}$
in \eqref{Eq:DefCompTwistProd0} as a map from
$\wideparen \maclA _0(\cc {2d})\times \wideparen \maclA _0(\cc {2d})$
to $\wideparen A(\cc {2d})$.
The following result is a straight-forward consequence of
\eqref{Eq:MultWprCoeffRel} and Proposition
\ref{Prop:RingModPropAnSpacesHelp2}. The details are left for the reader.

\par

\begin{thm}\label{Thm:WickComp}
Let $s\in \overline {\mathbf R}_{\flat}$
be such that $s< \frac 12$, $W =\cc d\times \cc d$,
and let $T_{\wpr}$ be the map from
$\wideparen \maclA _0(W)\times \wideparen \maclA _0(W)$
to $\wideparen A(W)$, given by $T_{\wpr}(a_1,a_2)=
a_1 \wpr _{\mathfrak V} a_2$.
Then $T_{\wpr}$ is uniquely extendable to continuous mappings
\begin{equation}
\begin{alignedat}{2}
T_{\wpr} &: \wideparen \maclA _s (W)\times \wideparen \maclA _s (W)
\to \wideparen \maclA _s (W), & &
\\[1ex]
T_{\wpr} &: \maclA _s' (W)\times \maclB _s (W)
\to
\maclA _s'(W), &
\quad
T_{\wpr} &: \wideparen \maclB _s^\star (W)\times
\wideparen \maclA _s' (W)
\to \wideparen \maclA _s' (W),
\\[1ex]
T_{\wpr} &: \wideparen \maclB _s(W)\times \wideparen \maclB _s(W)
\to {\maclB _s} (W), &
\quad
T_{\wpr} &: \wideparen \maclB _s^\star (W)\times
\wideparen \maclB _s^\star (W)
\to
\wideparen \maclB _s^\star (W),
\\[1ex]
T_{\wpr} &: \wideparen \maclC _s(W)\times \wideparen \maclC _s(W)
\to \wideparen \maclC _s(W), &
\quad
T_{\wpr} &: \wideparen \maclC _s^\star (W)\times
\wideparen \maclC _s^\star (W)
\to
\wideparen \maclC _s ^\star (W),
\end{alignedat}
\end{equation}

\par

If instead $0<s\le \frac 12$, then
the same holds true with
$\wideparen \maclA _{0,s}$, $\wideparen \maclB _{0,s}$
and $\wideparen \maclC _{0,s}$ in place of
$\wideparen \maclA _s$, $\wideparen \maclB _s$
and $\wideparen \maclC _s$, respectively, at
each occurrence.
\end{thm}

\par

By Theorems \ref{Thm:OpReprIdentC} and \ref{Thm:WickComp}
the twisted anti-Wick product $\wpr _{\mathfrak V}^{\aw}$, defined
by
\begin{equation}\label{Eq:AntiWickSymbComp}
\op _{\mathfrak V}^{\aw}(a_1)\circ \op _{\mathfrak V}^{\aw}(a_2)
=
\op _{\mathfrak V}^{\aw}(a_1\wpr _{\mathfrak V}^{\aw}a_2),
\end{equation}
makes sense, when $a_1$ and $a_2$ belong to the symbol
classes in Theorem \ref{Thm:OpReprIdentC}. The following result now
follows by combining Theorems \ref{Thm:OpReprIdentC}
and \ref{Thm:WickComp}. The details are left for the reader.

\par

\begin{thm}\label{Thm:AntiWickComp}
Let $s\in \overline {\mathbf R}_{\flat}$
be such that $s< \frac 12$, $W =\cc d\times \cc d$,
and let $T_{\wpr}^{\aw}$ be the map from
$\wideparen \maclA _0(W)\times \wideparen \maclA _0(W)$
to $\wideparen A(W)$, given by $T_{\wpr}^{\aw}(a_1,a_2)=
a_1 \wpr _{\mathfrak V}^{\aw} a_2$.
Then $T_{\wpr}^{\aw}$ is uniquely extendable to continuous mappings
\begin{equation}
\begin{alignedat}{2}
T_{\wpr}^{\aw} &: \wideparen \maclA _s (W)\times \wideparen \maclA _s (W)
\to \wideparen \maclA _s (W), & &
\\[1ex]
T_{\wpr}^{\aw} &: \wideparen \maclC _s(W)\times \wideparen \maclC _s(W)
\to \wideparen \maclC _s(W), &
\quad
T_{\wpr}^{\aw} &: \wideparen \maclC _s^\star (W)\times
\wideparen \maclC _s^\star (W)
\to
\wideparen \maclC _s ^\star (W),
\end{alignedat}
\end{equation}

\par

If instead $0<s\le \frac 12$, then
the same holds true with $\wideparen \maclA _{0,s}$
and $\wideparen \maclC _{0,s}$ in place of
$\wideparen \maclA _s$ and $\wideparen \maclC _s$, respectively, at
each occurrence.
\end{thm}

\par

\subsection{Some further relationships between operator kernels,
Wick and anti-Wick symbols}\label{Subsec4.3}

\par

Next we show some transition properties between
Wick, anti-Wick and kernel operators with symbols
and kernels belonging to refined
classes of $\wideparen \maclA _{0,1/2}(\cc {2d})$ and
$\wideparen \maclA _{0,1/2}'(\cc {2d})$.
For any $r>0$, let $\wideparen \maclA _{[r]}^2(\cc {2d})$ be the
Banach space which consists of all
$a\in \wideparen A(\cc {2d})$ such that
$$
\nm a{\wideparen \maclA _{[r]}^2}
\equiv
\left (
\iint _{\cc {2d}}|a(z,w)|^2e^{-r(|z|^2+|w|^2)}\, d\lambda (z)d\lambda (w)
\right )^{\frac 12}
$$
is finite. We observe that $\wideparen \maclA _{0,1/2}(\cc {2d})$ 
($\wideparen \maclA _{0,1/2}'(\cc {2d})$) is the projective (inductive)
limit of $\wideparen \maclA _{[r]}^2(\cc {2d})$ with respect to $r>0$.
First we have the following two propositions.

\par

\begin{prop}\label{Prop:WickAntiWickA2Symbols}
Let $r_1\in (0,1)$ and $r_2>r_1/(1-r_1)$. Then the following is true:
\begin{enumerate}
\item if $a\in \wideparen \maclA _{[r_1]}^2(\cc {2d})$,
then there is a unique $a^{\aw}\in \wideparen \maclA _{[r_2]}^2(\cc {2d})$
such that \eqref{Eq:AntiWickAnalPseudoRel} holds. Furthermore,
\begin{equation}\label{Eq:WickAntiWickA2Symbols1}
\nm {a^{\aw}}{\wideparen \maclA _{[r_2]}^2}
\lesssim
\nm {a}{\wideparen \maclA _{[r_1]}^2} \text ;
\end{equation}

\vrum

\item if $a^{\aw}\in \wideparen \maclA _{[r_1]}^2(\cc {2d})$,
then there is a unique $a\in \wideparen \maclA _{[r_2]}^2(\cc {2d})$
such that \eqref{Eq:AntiWickAnalPseudoRel} holds. Furthermore,
\begin{equation}\label{Eq:WickAntiWickA2Symbols1Op}
\nm {a}{\wideparen \maclA _{[r_2]}^2}
\lesssim
\nm {a^{\aw}}{\wideparen \maclA _{[r_1]}^2} .
\end{equation}
\end{enumerate}
\end{prop}

\par

\begin{prop}\label{Prop:WickToKernelsA2}
Let $r_1>0$ and $r_2>r_1+1$. Then the following is true:
\begin{enumerate}
\item if $a\in \wideparen \maclA _{[r_1]}^2(\cc {2d})$,
then there is a unique $K\in \wideparen \maclA _{[r_2]}^2(\cc {2d})$
such that \eqref{Eq:KernelWickRelByT}$'$ holds. Furthermore,
\begin{equation}\label{Eq:WickAntiWickA2Symbols1Ker}
\nm {K}{\wideparen \maclA _{[r_2]}^2}
\lesssim
\nm {a}{\wideparen \maclA _{[r_1]}^2} \text ;
\end{equation}

\vrum

\item  if $K\in \wideparen \maclA _{[r_1]}^2(\cc {2d})$,
then there is a unique $a\in \wideparen \maclA _{[r_2]}^2(\cc {2d})$
such that \eqref{Eq:KernelWickRelByT}$'$ holds. Furthermore,
\begin{equation}\label{Eq:WickAntiWickA2Symbols1KerOp}
\nm {K}{\wideparen \maclA _{[r_2]}^2}
\lesssim
\nm {a}{\wideparen \maclA _{[r_1]}^2} .
\end{equation}
\end{enumerate}
\end{prop}

\par

For the proofs we need the following continuity result for the operators
$\maclT _{0,t}$ and $\maclT _{0,t}^*$.
Here let
$$
\ell _r^p(\nn d) = \ell _{[\vartheta _r]}^p(\nn d),\qquad
\vartheta _r(\alpha )=r^{-\frac 12\cdot |\alpha |}, \ r>0,\ \alpha \in \nn d
$$
(cf. Definition \ref{DefSeqSpaces}).

\par

\begin{lemma}\label{Lemma:ContTOps1}
Let $p\in [1,\infty ]$, $r_1>0$, $t\in \mathbf C$ be such that $|t|\le 1$, and let 
$\maclT _{0,t}$ and $\maclT _{0,t}^*$ be given by \eqref{Eq:T_{0,t}Def}
and \eqref{Eq:T_{0,t}Dual}, $j=1,2$. Then the following is true:
\begin{enumerate}
\item if $r_2\ge 1+r_1$ with strict inequality when $p<\infty$,
then $\maclT _{0,t}$ restricts to a continuous map from
$\ell _{r_1}^p(\nn {2d})$ to $\ell _{r_2}^p(\nn {2d})$;

\vrum

\item if in addition $r_1<1$ and $r_2\ge r_1/(1-r_1)$ with strict inequality when
$p>1$, then $\maclT _{0,t}^*$ restricts to a continuous
map from $\ell _{r_1}^p(\nn {2d})$ to $\ell _{r_2}^p(\nn {2d})$.
\end{enumerate}
\end{lemma}

\par

For the proof we observe that
\begin{equation}\label{Eq:BasicEstA}
|b(\alpha ,\beta )|\le r^{\frac 12|\alpha +\beta |}\nm b{\ell _r^p},
\qquad b\in \ell _r^p(\nn {2d}).
\end{equation}

\par

\begin{proof}
By duality it suffices to prove (1).
We only prove the result for $p<\infty$. The other cases follow by similar arguments
and are left for the reader.

\par

Let $r=r_1$ and $b\in \ell _r^p(\nn {2d})$. Then it follows from
\eqref{Eq:BasicEstA}, Cauchy-Schwartz inequality and
straight-forward computations that
\begin{multline*}
\nm {\maclT _{0,t}b}{\ell _{r_2}^p}^p
\le
\sum _{\alpha ,\beta} \left (
\sum _{\gamma \le \alpha ,\beta}
{{\alpha}\choose {\gamma}}^{\frac 12}
{{\beta}\choose {\gamma}}^{\frac 12}|b(\alpha -\gamma ,\beta -\gamma)|
\right )^pr_2^{-\frac p2|\alpha +\beta |}
\\[1ex]
\le
\nm b{\ell _r^p}^p\sum _{\alpha ,\beta} \left (
\sum _{\gamma \le \alpha ,\beta}
{{\alpha}\choose {\gamma}}^{\frac 12}
{{\beta}\choose {\gamma}}^{\frac 12}r^{-|\gamma |}
\right )^pr^{\frac p2 |\alpha +\beta |}r_2^{-\frac p2 |\alpha +\beta |}
\\[1ex]
\le
\nm b{\ell _r^p}^p\sum _{\alpha ,\beta} \left (
\sum _{\gamma \le \alpha}
{{\alpha}\choose {\gamma}}
r^{-|\gamma |}
\right )^{\frac p2}
\left (
\sum _{\gamma \le \beta}
{{\beta}\choose {\gamma}}r^{-|\gamma |}
\right )^{\frac p2}
\left (
\frac r{r_2}
\right )^{\frac p2|\alpha +\beta |}
\\[1ex]
=
\nm b{\ell _r^p}^p\sum _{\alpha ,\beta}
\left (
\frac {1+r}{r_2}
\right )^{\frac p2 |\alpha +\beta |}.
\end{multline*}
Since $1+r<r_2$, it follows that the last series converges, which gives the asserted
continuity.
\end{proof}

\par

\begin{proof}[Proof of Propositions \ref{Prop:WickAntiWickA2Symbols}
and \ref{Prop:WickToKernelsA2}]
If $r>0$ and $a(z,w)\in \wideparen A(\cc {2d})$ has the expansion
\begin{equation}\label{Eq:FormalConjPowerSeriesSymb}
a(z,w)
=
\sum
c(a;\alpha ,\beta)e_{\alpha }(z)e_{\beta}(\overline w),
\quad z,w \in \cc {d}, j=1,2, 
\end{equation}
then \cite[Corollary 3.8]{Toft18} shows that
$$
\nm {a}{\wideparen \maclA _{[r]}^2}
\asymp
\sum _{\alpha ,\beta \in \nn d} |c(a;\alpha ,\beta )|^2r^{-|\alpha |}.
$$
The result now follows from the latter relation, Lemma
\ref{Lemma:ContTOps1}, the fact that
$\ell _0(\nn {2d})$ is dense in $\ell ^2_r(\nn {2d})$, leading to that
$\wideparen \maclA _0(\cc {2d})$ is dense in
$\wideparen \maclA _{[r]}^2(\cc {2d})$, and that 
\eqref{Eq:AntiWickAnalPseudoRel} and \eqref{Eq:KernelWickRelByT}$'$
hold true when $a,a^{\aw}\in \wideparen \maclA _0(\cc {2d})$.
\end{proof}

\par

In \cite[Section 3]{TeToWa} some results concerning the Wick symbols of anti-Wick operators
are presented. For example we have the following result. We omit the proof
since the result is a special case of Theorems 3.1, 3.3 and Theorem 3.7
in \cite{TeToWa}. Here $\mascP _E(\cc d)$ is the set of all
$\omega \in L^\infty _{loc}(\cc d;\mathbf R_+)$ such that
$$
\omega (z+w)\lesssim \omega (z)v(w),
$$
for some $v\in L^\infty _{loc}(\cc d;\mathbf R_+)$.

\par

\begin{thm}\label{Thm:WickSymbolAntiWickOpGS}
Let $s\ge \frac 12$ ($s> \frac 12$), $a,a^{\aw}\in \wideparen A(\cc {2d})$
be such that \eqref{Eq:AntiWickAnalPseudoRelIntForm}$'$
holds, and let $\omega \in \mascP _E(\cc d)$.
Then the following is true:
\begin{enumerate}
\item  if $|a (w,w)| \lesssim e^{-r_0 |w|^{\frac 1s}}$
holds for some $r_0>0$ (for every $r_0>0$), then
\begin{equation}\label{Eq:AWGSEst1}
|a^{\aw} (z,w)|
\lesssim
e^{\frac 14|z-w|^2-r |z+w|^{\frac 1s}}
\end{equation}
for some $r>0$ (for every $r>0$);

\vrum

\item if $|a (w,w)| \lesssim e^{r_0 |w|^{\frac 1s}}$
holds for every $r_0>0$
(for some $r_0>0$), then
\begin{equation}\label{Eq:AWGSEst2}
|a^{\aw} (z,w)|
\lesssim
e^{\frac 14|z-w|^2+r |z+w|^{\frac 1s}}
\end{equation}
for every $r>0$ (for some $r>0$);

\vrum

\item if $|a(w,w)| \lesssim e^{r|w|^2}$ for some $r<1$,
then
\begin{equation}\label{Eq:ConjAnalSet2.5}
|a^{\aw}(z,w)|\lesssim 
e^{r_0 |z+w|^2-\operatorname{Re}(z,w)},
\qquad
r_0=4^{-1}(1-r)^{-1}\text ;
\end{equation}

\vrum

\item if $|a(w,w)|\lesssim  \omega (2w)$, then
\begin{equation}\label{Eq:ConjAnalSet3}
|a^{\aw}(z,w)|\lesssim e^{\frac 14|z-w|^2} \omega (z+w),
\quad z,w\in \cc d.
\end{equation}
\end{enumerate}
\end{thm}

\par

By using \eqref{Eq:AntiWickAnalPseudoRelIntForm2} instead of
\eqref{Eq:AntiWickAnalPseudoRelIntForm}, we get the
following result in the other direction compared to
previous result. The details are left for the reader.

\par

\begin{thm}\label{Thm:AntiWickSymbolWickOpGS}
Let $s\ge \frac 12$ ($s> \frac 12$), $a,a^{\aw}\in \wideparen A(\cc {2d})$
be such that \eqref{Eq:AntiWickAnalPseudoRelIntForm}$'$
holds, and let $\omega \in \mascP _E(\cc d)$.
Then the following is true:
\begin{enumerate}
\item  if $|a ^{\aw}(w,-w)| \lesssim e^{-r_0 |w|^{\frac 1s}}$
holds for some $r_0>0$ (for every $r_0>0$), then
\begin{equation}\label{Eq:AWGSEstOp1}
|a(z,w)|
\lesssim
e^{\frac 14|z+w|^2-r |z-w|^{\frac 1s}}
\end{equation}
for some $r>0$ (for every $r>0$);

\vrum

\item if $|a^{\aw} (w,-w)| \lesssim e^{r_0 |w|^{\frac 1s}}$
holds for every $r_0>0$
(for some $r_0>0$), then
\begin{equation}\label{Eq:AWGSEstOp2}
|a(z,w)|
\lesssim
e^{\frac 14|z+w|^2+r |z-w|^{\frac 1s}}
\end{equation}
for every $r>0$ (for some $r>0$);

\vrum

\item if $|a^{\aw}(w,-w)| \lesssim e^{r|w|^2}$ for some $r<1$,
then
\begin{equation}\label{Eq:ConjAnalSetOp2.5}
|a(z,w)|\lesssim 
e^{r_0 |z-w|^2+\operatorname{Re}(z,w)},
\qquad
r_0=4^{-1}(1-r)^{-1}\text ;
\end{equation}

\vrum

\item if $|a^{\aw}(w,-w)|\lesssim  \omega (2w)$, then
\begin{equation}\label{Eq:ConjAnalSet3Op}
|a(z,w)|\lesssim e^{\frac 14|z+w|^2} \omega (z-w),
\quad z,w\in \cc d.
\end{equation}
\end{enumerate}
\end{thm}

\par

\begin{rem}\label{Rem:DifferenceEstWickAntiWick}
The relationships between the Wick and anti-Wick symbols in
Theorems \ref{Thm:WickSymbolAntiWickOpGS} and
\ref{Thm:AntiWickSymbolWickOpGS} are similar, and one might believe
that transitions of continuity properties for Wick operators to anti-Wick
operators are similar as for transitions in reversed directions. Here we notice
that this is not the case.

\par

For example, suppose that $a^{\aw}$ satisfies \eqref{Eq:AWGSEst2} for
every $r>0$, $0<t_0<\frac 34$ and that $F\in A(\cc d)$ satisfies
\begin{equation}\label{Eq:DifferencesEstWickAntiWick1}
|F(z)|\lesssim e^{t_0|z|^2+r|z|^{\frac 1s}}
\end{equation}
for every $r>0$.
(The case when $a^{\aw}$ satisfies conditions of the form
\eqref{Eq:AWGSEst1} is treated in similar ways and leads to even stronger
continuity properties.) If
$$
t_1= \frac {1-t_0}{3-4t_0}
$$
(which is bounded from below by $t_0$), then the size of the integrand in
$$
|(\op _{\mathfrak V}(a^{\aw})F)(z)|
\le \pi^{-d}
\int _{\cc d}|a^{\aw}(z,w)F(w)|e^{\repart (z,w)-|w|^2}\, d\lambda (w)
$$
can be estimated as
\begin{multline}\label{Eq::DifferenceEstWickAntiWick1}
|a^{\aw}(z,w)F(w)|e^{\repart (z,w)-|w|^2}
\lesssim
e^{\frac 14|z-w|^2+r|z+w|^{\frac 1s}}e^{\repart (z,w)-|w|^2}e^{t_0|w|^2+r|w|^{\frac 1s}}
\\[1ex]
=
e^{t_1|z|^2-\frac {3-4t_0}4|w-\frac z{3-4t_0}|^2+ r(|z+w|^{\frac 1s}+|w|^{\frac 1s}}
\\[1ex]
\lesssim
e^{t_1|z|^2+cr|z|^{\frac 1s}}
e^{-\frac {3-4t_0}4|w-\frac z{3-4t_0}|^2+ cr|w-\frac z{3-4t_0}|^{\frac 1s}},
\end{multline}
for some $c\ge 1$ which is independent of $z,w\in \cc d$ and $r>0$. By
integrating the last estimate with respect to $w$, it follows that
\eqref{Eq:DifferencesEstWickAntiWick1} holds true for every $r>0$ with
$\op _{\mathfrak V}(a^{\aw})F$ and $t_1$ in place of $F$ and $t_0$.

\par

By choosing $t_0=\frac 12$, then $t_1=\frac 12$, and we have proved that
$\op _{\mathfrak V}(a^{\aw})$ is continuous on $\maclA _s'(\cc d)$ for such
$t_0$.

\par

As side remark we observe that stronger continuity properties for anti-Wick operators with
symbols satisfying estimates of the form $a(w,w)e^{\pm r|w|^{\frac 1s}}$
are obtained with more direct
computational methods, without rewriting $\op _{\mathfrak V}^{\aw}(a)$
as a Wick operator. (See Proposition 3.6 to Corollary 3.10 in \cite{TeToWa}.)

\par

The corresponding estimates \eqref{Eq:AWGSEstOp1}
and \eqref{Eq:AWGSEstOp2} when passing from Wick to
anti-Wick operators lead to strongly different conclusions.
In fact, let $F\in A(\cc d)$. For the integrand in
$$
|(\op _{\mathfrak V}^{\aw}(a)F)(z)|
\le \pi^{-d}
\int _{\cc d}|a(w,w)F(w)|e^{\repart (z,w)-|w|^2}\, d\lambda (w)
$$
a similar type of estimate gives
\begin{multline}\label{Eq::DifferenceEstWickAntiWick2}
|a(w,w)F(w)|e^{\repart (z,w)-|w|^2}
\lesssim
e^{\frac 14|w+w|^2\pm r|w-w|^{\frac 1s}}e^{\repart (z,w)-|w|^2}|F(w)|
\\[1ex]
=
e^{\repart (z,w)}|F(w)|.
\end{multline}
Since $F$ should be entire, the right-hand side is integrable with
respect to $w$, only when $F$ is identically zero. Consequently, 
\eqref{Eq:AWGSEstOp1} and \eqref{Eq:AWGSEstOp2} might be useful only
when $F$ is the trivial zero function.

\par

In this context we observe that continuity properties for
$\op _{\mathfrak V}(a^{\aw})$ are reachable with more direct
computational methods, without rewriting it as an anti-Wick
operator. (See \cite{Teofanov2,TeToWa}.)
\end{rem}

\par

\section{Applications to linear operators
on functions and distributions defined on $\rr d$}\label{sec5}

\par

In this section we use the results in previous sections to extend
the definition of Toeplitz operators. In the end we find that if
$s\in \overline{\mathbf R_\flat}$ is suitable, then
any linear and continuous operators on $\maclA _s(\cc d)$,
$\maclA _{0,s}(\cc d)$ and their duals,
can be expressed as Toeplitz operators.

\par

If $a\in \wideparen \maclA _0(\cc {2d})$, then $\tp _{\mathfrak V}(a)$
is the linear and continuous operator on $\maclS _{1/2}(\rr d)$,
given by \eqref{Eq:ToeplitzDef}$''$. For such operators we have the
following extensions.

\par

\begin{thm}\label{Thm:ToeplExt}
Let $s_1,s_2\in \overline{\mathbf R_\flat}$ be such that $s_1<\frac 12$
and $0<s_2\le \frac 12$. Then the following is true:
\begin{enumerate}
\item the map $(a,f)\mapsto \tp _{\mathfrak V}(a)f$ from
$\wideparen \maclA _0(\cc {2d})\times \maclH _0(\rr d)$ to
$\maclH _0'(\rr d)$ is uniquely extendable to a continuous
map $\wideparen \maclC _{s_1}(\cc {2d})\times \maclH _{s_1}(\rr d)$
to $\maclH _{s_1}(\rr d)$, and from $\wideparen \maclC _{s_1}^\star
(\cc {2d})\times \maclH _{s_1}'(\rr d)$
to $\maclH _{s_1}'(\rr d)$;

\vrum

\item  the map $(a,f)\mapsto \tp _{\mathfrak V}(a)f$ from
$\wideparen \maclA _0(\cc {2d})\times \maclH _0(\rr d)$ to
$\maclH _0'(\rr d)$ is uniquely extendable to a continuous
map $\wideparen \maclC _{0,s_2}(\cc {2d})\times \maclH _{0,s_2}(\rr d)$
to $\maclH _{0,s_2}(\rr d)$, and from $\wideparen \maclC _{0,s_2}^\star
(\cc {2d})\times \maclH _{0,s_2}'(\rr d)$
to $\maclH _{0,s_2}'(\rr d)$.
\end{enumerate}
\end{thm}

\par

\begin{proof}
The asserted continuity extensions
follow from \eqref{Eq:ToeplitzDef}$'$, \eqref{Eq:ToeplAntiWick},
Theorem \ref{Thm:OpReprIdentB} and the facts
that the Bargmann transform is homeomorphic from the spaces in
\eqref{clHSpaces} to the spaces in \eqref{clASpaces}. We need to prove
the uniqueness.

\par

By playing with $r_1$ and $r_2$ in Definition \ref{Def:elltauSpaces2},
it follows that $\ell _0(\Lambda )$ is dense in the spaces in
\eqref{Eq:ellBparenSpaces1to4}, which implies that
$\wideparen \maclA _0(W)$ is dense in the spaces in
\eqref{Eq:BparenSpaces}.
The uniqueness in Theorem \ref{Thm:ToeplExt} now follows
by combining these density properties with the fact that $\maclH _0(\rr d)$
is dense in the spaces in \eqref{clHSpaces}.
\end{proof}

\par

\par

\begin{thm}\label{Thm:ToeplExtInj}
Let $s_1,s_2\in \overline{\mathbf R_\flat}$ be such that $s_1<\frac 12$
and $0<s_2\le \frac 12$. If $a_1,a_2\in \wideparen \maclC _{s_1}(\cc {2d})$
and $\tp _{\mathfrak V}(a_1)f=\tp _{\mathfrak V}(a_2)f$ for every
$f\in \maclH _0(\rr d)$, then $a_1=a_2$. The same holds true for
$\wideparen \maclC _{s_1}^\star$, $\wideparen \maclC _{0,s_2}$
and $\wideparen \maclC _{0,s_2}^\star$ in place of
$\wideparen \maclC _{s_1}$.
\end{thm}

\par

\begin{proof}
Suppose that $a_1,a_2\in \wideparen \maclC _{s_1}(\cc {2d})$
satisfy $\tp _{\mathfrak V}(a_1)f=\tp _{\mathfrak V}(a_2)f$ for every
$f\in \maclH _0(\rr d)$.
Then Theorem \ref{Thm:TtDualhomeomorphismEllC} gives
\begin{align*}
\tp _{\mathfrak V}(a_1) = \tp _{\mathfrak V}(a_2) 
\quad
&\Leftrightarrow
\quad
\op _{\mathfrak V}^{\aw}(a_1) = \op _{\mathfrak V}^{\aw}(a_2) 
\quad
\Leftrightarrow
\quad
\op _{\mathfrak V}(\maclT _1^*a_1) = \op _{\mathfrak V}(\maclT _1^*a_2) 
\\[1ex]
&\Leftrightarrow
\quad
\maclT _1^*a_1 = \maclT _1^*a_2
\quad
\Leftrightarrow
\quad
a_1 = a_2. \qedhere
\end{align*}
\end{proof}

\par

\section{Linear pullbacks and trace mappings for spaces
of power series expansions}\label{sec6}

\par

In this section we show that for $0\le s< \frac 12$ (for $0< s\le \frac 12$),
then linear pullbacks and trace mappings are continuous mappings
on the spaces in \eqref{Eq:DualtauRoumSpaces}
(in \eqref{Eq:DualtauBeurSpaces}).

\par

For $B_j\in \GL (d_j,\mathbf C)$, $j=1,2$, we consider linear pullbacks
of the form
\begin{align}
K(z_2,z_1) &\mapsto K(z_2,B_1z_1)
\label{Eq:LinjPullbackMapDel1}
\\[1ex]
K(z_2,z_1) &\mapsto K(B_2z_2,z_1)
\label{Eq:LinjPullbackMapDel2}
\intertext{and}
K(z_2,z_1) &\mapsto K(B_2z_2,B_1z_1).
\label{Eq:LinjPullbackMap}
\end{align}
Here $\GL (d,\mathbf C)$ is the set of all $d\times d$ matrices
with entries in $\mathbf C$. For such pullbacks we have
the following.

\par

\begin{thm}\label{Thm:LinjPullbacks}
Let $s_1,s_2\in \overline{\mathbf R_{\flat ,\infty}}$, $s=(s_2,s_1)$,
$W=\cc {d_2}\times \cc {d_1}$ and $B_j\in \GL (d_j,\mathbf C)$. Then
the following is true:
\begin{enumerate}
\item the map \eqref{Eq:LinjPullbackMap} on $\wideparen A(W)$
extends uniquely to a continuous map on $\wideparen
\maclA _0'(W)$;

\vrum

\item if in addition $s_1<\frac 12$ ($0<s_1\le \frac 12$),
then \eqref{Eq:LinjPullbackMapDel1}
on $\wideparen \maclA _0'(W)$ restricts to continuous
mappings on the spaces in \eqref{Eq:DualtauRoumSpaces}
(in \eqref{Eq:DualtauBeurSpaces});

\vrum

\item if in addition $s_2<\frac 12$ ($0<s_2\le \frac 12$),
then \eqref{Eq:LinjPullbackMapDel2}
on $\wideparen \maclA _0'(W)$ restricts to continuous
mappings on the spaces in \eqref{Eq:DualtauRoumSpaces}
(in \eqref{Eq:DualtauBeurSpaces});

\vrum

\item  if in addition $s_1,s_2< \frac 12$
($0<s_1,s_2\le \frac 12$), then \eqref{Eq:LinjPullbackMap}
on $\wideparen \maclA _0'(W)$ restricts to continuous
mappings on the spaces in \eqref{Eq:DualtauRoumSpaces}
(in \eqref{Eq:DualtauBeurSpaces}).
\end{enumerate}
\end{thm}

\par

\begin{proof}
We only prove (1) and (3). The assertion (2) follows from
similar arguments and then (4) follows by combining
(2) and (3). The details are left for the reader.

\par

Let $B=B_2$, $d=d_2$ and $z=z_2\in \cc d$.
For any integer $N\ge 0$ we let
$$
\Omega _N = \Omega _{N,d} = \sets {\alpha \in \nn d}{|\alpha |=N}.
$$
Then
\begin{equation}\label{Eq:PullbackBasisEl}
e_\alpha (Bz) = \sum _{\beta \in \Omega _N}
C_B(\beta ,\alpha )e_\beta (z),
\quad
\alpha \in \Omega _N,\ z\in \cc d,
\end{equation}
for some constants $C_B(\beta ,\alpha )$ which are uniquely defined
and only depend on $\alpha$, $\beta$ and $B$. If
$K\in \wideparen A(W)$ is given by \eqref{Eq:AnalKernelExp}, then
\begin{align}
K_B(z_2,z_1)\equiv K(Bz_2,z_1)
&=
\sum _{\alpha _j\in \nn {d_j}}
c(K_B;\alpha _2,\alpha _1)e_{\alpha _2}(z_2)e_{\alpha _1}(\overline {z_1}),
\label{Eq:PullbAnalKernelExp}
\intertext{where}
c(K_B;\alpha _2,\alpha _1)
&=
\sum _{\beta \in \Omega _{|\alpha _2|}}C_B(\alpha _2,\beta )
c(K;\beta ,\alpha _1).
\label{Eq:PullbAnalKernelCoeff}
\end{align}

\par

Hence, if, more generally $K\in \wideparen \maclA _0'(W)$, then
the only possibility to define $K(B_2z_2,z_1)$ is given by
\eqref{Eq:PullbAnalKernelExp} and \eqref{Eq:PullbAnalKernelCoeff}.
Since $\Omega _N$ is a finite set for every $N$, it follows that
\eqref{Eq:PullbAnalKernelCoeff} is well-defined and that the map
$$
\{ c(K;\alpha _2,\alpha _1)\}
_{(\alpha _2,\alpha _1)\in \nn {d_2}\times \nn {d_1}}
\mapsto
\{ c(K_B;\alpha _2,\alpha _1)\}
_{(\alpha _2,\alpha _1)\in \nn {d_2}\times \nn {d_1}}
$$
is continuous
\begin{alignat*}{3}
&\ell _{\maclA _0}(\nn {d_2}\times \nn {d_1}), &
\quad
&\ell _{\maclA _0'}(\nn {d_2}\times \nn {d_1}), &
\quad
&\ell _{\maclB _{(0,0)}}(\nn {d_2}\times \nn {d_1}),
\\[1ex]
&\ell _{\maclB _{(0,0)}}^\star (\nn {d_2}\times \nn {d_1}), &
\quad
&\ell _{\maclC _{(0,0)}} (\nn {d_2}\times \nn {d_1}) &
\quad &\text{and on}\quad
\ell _{\maclC _{(0,0)}}^\star (\nn {d_2}\times \nn {d_1}).
\end{alignat*}
Hence
the map \eqref{Eq:LinjPullbackMap} on $\wideparen A(W)$
is uniquely extendable to a continuous map on
$\wideparen \maclA _0'(W)$, which in turn restricts to continuous
mappings on the spaces in \eqref{Eq:DualtauRoumSpaces}
when $s_1=s_2=0$. This gives (1) and (3)
for $s_1=s_2=0$.

\par

We only prove (3) in the case when $s_1,s_2>0$. The case when
$s_1=s_2=0$ is already treated and the case when one of $s_1$
and $s_2$ is zero follows by similar arguments and is left for the
reader.

\par

Let $\{ b(j,k) \} _{k=1}^d$be the entries in the row $j$ of $B$ and let
$$
{N\choose \alpha} = \frac {N!}{\alpha _1!\cdots \alpha _d!},
\quad \alpha =(\alpha _1,\dots ,\alpha _d)\in \Omega _N
$$
Then
\begin{align*}
e_\alpha (Bz)
&=
\frac 1{\alpha !^{\frac 12}}
\prod _{j=1}^d (b(j,1)z_1+\cdots +b(j,d))^{\alpha _j}
\\[1ex]
&=
\frac 1{\alpha !^{\frac 12}}
\prod _{j=1}^d \left (
\sum _{\gamma _j\in \Omega _{\alpha _j}} {{\alpha _j}\choose {\gamma _j}}
\prod _{m=1}^d \left ( b(j,m)^{\gamma _{j,m}}z_m^{\gamma _{j,m}} \right )
\right )
\\[1ex]
&=
\frac 1{\alpha !^{\frac 12}}
\sum _{\gamma _1\in \Omega _{\alpha _1}}
\cdots
\sum _{\gamma _d\in \Omega _{\alpha _d}}
\left ( \prod _{j=1}^d 
{{\alpha _j}\choose {\gamma _j}}
\prod _{m=1}^d \left ( b(j,m)^{\gamma _{j,m}}z_m^{\gamma _{j,m}} \right )
\right )
\end{align*}
Here $\gamma _j=(\gamma _{j,1},\dots ,\gamma _{j,d})$ in the expressions
above. If we let
\begin{equation}\label{Eq:cgammaDef}
c(\gamma ) = \prod _{j,m=1}^d b(j,m)^{\gamma _{j,m}},
\end{equation}
and using \eqref{Eq:basiselements},
then it follows by straight-forward computations that
\begin{align}
e_\alpha (Bz)
&=
\sum _{\gamma _1\in \Omega _{\alpha _1}}
\cdots
\sum _{\gamma _d\in \Omega _{\alpha _d}}
\left (
{{\beta _1}\choose {\delta _1}}\cdots {{\beta _d}\choose {\delta _d}}
\right )^{\frac 12}
\left (
\frac {\alpha !}{\beta !}
\right )^{\frac 12}
c(\gamma )e_\beta (z),
\intertext{where}
\delta _j &= (\gamma _{1,j},\dots ,\gamma _{d,j})
\quad \text{and}\quad
\beta _j=\gamma _{1,j}+\cdots +\gamma _{d,j}.
\label{Eq:NewMultiIndices}
\end{align}
This in turn implies that
$$
K_B(z,z_1) =
\sum 
\left (
{{\beta _1}\choose {\delta _1}}\cdots {{\beta _d}\choose {\delta _d}}
\right )^{\frac 12}
\left (
\frac {\alpha !}{\beta !}
\right )^{\frac 12}
c(K;\alpha ,\varrho )c(\gamma )e_\beta (z)e_\varrho (\overline z_1),
$$
where the sum is taken over all
\begin{equation}\label{Eq:SumOver}
\alpha \in \nn d,\quad \varrho \in \nn {d_1}
\quad \text{and}\quad
\gamma _j\in \Omega _{\alpha _j},\ j=1,\dots ,d.
\end{equation}
Hence
\begin{equation}\label{Eq:cBKExact}
c(K_B;\beta ,\varrho ) =
\sum 
\left (
{{\beta _1}\choose {\delta _1}}\cdots {{\beta _d}\choose {\delta _d}}
\right )^{\frac 12}
\left (
\frac {\alpha !}{\beta !}
\right )^{\frac 12}
c(K;\alpha ,\varrho )c(\gamma ),
\end{equation}
where the sum is taken over all $\alpha$ and $\gamma _j$ in
\eqref{Eq:SumOver} such that \eqref{Eq:NewMultiIndices} holds.
Here we observe that
\begin{equation}\label{Eq:ModalphaModbeta}
|\alpha | = |\beta | = \sum _{j,k=1}^d \gamma _{j,k}.
\end{equation}
This implies that
$$
{{\beta _1}\choose {\delta _1}}\cdots {{\beta _d}\choose {\delta _d}}+
\frac {\alpha !}{\beta !} \le C^{|\alpha |}
\quad \text{and}\quad |c(\gamma )|\le M^{|\alpha |},\
M=\max _{1\le j,m\le d}|b(j,m)|,
$$
for some constant $C\ge 1$. Here the last estimate follows from
\eqref{Eq:cgammaDef}. Since it is clear that the number of elements
in the sum in \eqref{Eq:cBKExact} is bounded by
$$
(1+\alpha _1)^d\cdots (1+\alpha _d)^d\lesssim C^{|\alpha |},
$$
for some constant $C\ge 1$,
it follows by combining these estimates and \eqref{Eq:ModalphaModbeta}
that
$$
\max _{|\alpha |=N}|c(K_B;\alpha ,\varrho )|
\le
C^N \max _{|\alpha |=N}|c(K;\alpha ,\varrho )|
$$
for some $C$ which is independent of $N$ and $\varrho$.

\par

If $\vartheta _{r,(s_2,s_1)}$ and $\omega _{s_2,s_1;r_2,r_1}$
are the same as in Definitions
\ref{Def:SeqSpacesStraightMixed} and \ref{Def:elltauSpaces2},
it follows by the assumptions on $s_2$ that
\begin{align*}
\sup _{\alpha \in \nn d}
|c(K_B;\alpha ,\varrho )\vartheta _{r,(s_2,s_1)}(\alpha ,\varrho )|
&\le
C\sup _{\alpha \in \nn d}
|c(K;\alpha ,\varrho )\vartheta _{c_1r,(s_2,s_1)}(\alpha ,\varrho )|,
\\[1ex]
\sup _{\alpha \in \nn d}
|c(K_B;\alpha ,\varrho )/\vartheta _{r,(s_2,s_1)}(\alpha ,\varrho )|
&\le
C\sup _{\alpha \in \nn d}
|c(K;\alpha ,\varrho )/\vartheta _{c_2r,(s_2,s_1)}(\alpha ,\varrho )|,
\\[1ex]
\sup _{\alpha \in \nn d}
|c(K_B;\alpha ,\varrho )\omega _{r_2,r_1;s_2,s_1}(\alpha ,\varrho )|
&\le
C\sup _{\alpha \in \nn d}
|c(K;\alpha ,\varrho )\omega _{c_1r_2,r_1;s_2,s_1}(\alpha ,\varrho )|,
\intertext{and}
\sup _{\alpha \in \nn d}
|c(K_B;\alpha ,\varrho )/\omega _{r_2,r_1;s_2,s_1}(\alpha ,\varrho )|
&\le
C\sup _{\alpha \in \nn d}
|c(K;\alpha ,\varrho )/\omega _{c_2r_2,r_1;s_2,s_1}(\alpha ,\varrho )|,
\end{align*}
for some constants $c_1,c_2,C>0$ which only depend on $s_2$ and $d$.
The asserted continuity properties in (3) now follows from these estimates.
\end{proof}

\medspace

Next we consider generalized and twisted forms of trace mappings on the
the spaces in \eqref{Eq:DualtauRoumSpaces} and
\eqref{Eq:DualtauBeurSpaces}, given in the following definition.

\par

\begin{defn}\label{Def:TwistTraceMap}
Let $d,d_j,n,n_j\in \mathbf N$,
$\{ e_{j,d}\} _{j=1}^d$ be the canonical basis on $\cc d$
and let $S_d$ be the permutation group on $\{ 1,\dots ,d\}$.
\begin{enumerate}
\item For any $\tau \in S_{d+n}$, $\iota _{\tau ,d}$ is the
linear map from $\cc d$ to $\cc {d+n}$ such that
$$
\iota _{\tau ,d}e_{j,d} = e_{\tau (j),d+n} , \quad j=1,\dots ,d \text .
$$

\vrum

\item The \emph{twisted trace map}
$\operatorname{Tr}_{\tau ,d}$ with respect to $\tau \in S_{d+n}$,
$d$ and $n$ is the continuous (pullback) map from $\maclA _0'(\cc {d+n})$
to $\maclA _0'(\cc d)$ given by
$$
(\operatorname{Tr}_{\tau ,d}F)(z) \equiv F(\iota _{\tau ,d}(z)),
\qquad   F \in \maclA _0'(\cc {d+n}),\ z\in \cc d \text .
$$

\vrum

\item The \emph{twisted trace map}
$\operatorname{Tr}_{\tau,d}$
with respect to $\tau =(\tau _2,\tau _1)$ and $d=(d_2,d_1)$,
$\tau _j\in S_{d_j+n_j}$ and $d_j$, $j=1,2$,
is the continuous (pullback) map from
$\wideparen \maclA _0'(\cc {d_2+n_2}\times \cc {d_1+n_1})$
to $\wideparen \maclA _0'(\cc {d_2}\times \cc {d_1})$ given by
$$
(\operatorname{Tr}_{\tau ,d}K)(z_2,z_1)
\equiv K(\iota _{\tau _2,d_2}(z_2),\iota _{\tau _1,d_1}(z_1))
\qquad   K \in \wideparen \maclA _0'(\cc {d_2+n_2}\times \cc {d_1+n_1}),
\ z_j\in \cc {d_j} \text .
$$
\end{enumerate}
\end{defn}

\par

If $\tau \in S_{d+n}$ and $\iota ^*_{\tau ,d}$
is the map from $\nn d$ to $\nn {n+d}$, given by
\begin{equation}
\iota ^*_{\tau ,d}(\alpha ) = (\beta _1,\dots ,\beta _{d+n}),
\quad \text{where}\quad
\beta _j
=
\begin{cases}
\alpha _{\tau ^{-1}(j)}, & 1\le \tau ^{-1}(j)\le d,
\\[1ex]
0, & d< \tau ^{-1}(j)\le d+n,
\end{cases}
\end{equation}
then it follows that
\begin{equation}
(\operatorname{Tr}_{\tau ,d}F)(z) = \sum _{\alpha \in \nn d}
c(F;\iota ^*_{\tau ,d}(\alpha ))e_\alpha (z),
\qquad
z\in \cc d,
\end{equation}
when
$$
F(z) = \sum _{\alpha \in \nn {d+n}}
c(F;\alpha )e_\alpha (z),
\qquad
z\in \cc {d+n}.
$$

\par

Now let $\vartheta _{r,(s_2,s_1)}$ and $\omega _{s_2,s_1;r_2,r_1}$
be the same as in Definitions \ref{Def:SeqSpacesStraightMixed}
and \ref{Def:elltauSpaces2}. Then it follows that
\begin{multline*}
\sup _{\alpha _j\in \nn {d_j}}
|c(K;\iota ^*_{\tau ,d_2}(\alpha _2),\iota ^*_{\tau ,d_1}(\alpha _1))
\vartheta _{r,(s_2,s_1)}
(\iota ^*_{\tau ,d_2}(\alpha _2),\iota ^*_{\tau ,d_1}(\alpha _1))|
\\[1ex]
\le
\sup _{\alpha _j\in \nn {d_j+n_j}}
|c(K;\alpha _2,\alpha _1)
\vartheta _{r,(s_2,s_1)} (\alpha _2,\alpha _1)|,
\end{multline*}
when $K\in \wideparen \maclA _0'
(\cc {d_2+n_2}\times \cc {d_1+n_1})$, and similarly with
$$
1/\vartheta _{r,(s_2,s_1)},\quad
\omega _{s_2,s_1;r_2,r_1} \quad \text{or}\quad
1/\omega _{s_2,s_1;r_2,r_1}
$$
in place of $\vartheta _{r,(s_2,s_1)}$. The following twisted
trace result is now
a straight-forward consequence of these observations. The details
are left for the reader.

\par

\begin{prop}\label{Prop:TwistTrace}
Let $s_1,s_2\in \overline{\mathbf R_{\flat}}$, $s=(s_1,s_2)$,
$\operatorname{Tr}_{\tau ,d}$ be as in {\rm{(3)}}
in Definition \ref{Def:TwistTraceMap}, $W_1=\cc {d_2}\times \cc {d_1}$
and $W_2=\cc {d_2+n_2}\times \cc {d_1+n_1}$. Then 
$\operatorname{Tr}_{\tau ,d}$ restricts to continuous mappings
from
\begin{alignat*}{7}
& \wideparen \maclA _{s}(W_2), &
\quad
& \wideparen \maclA _{s}^\star (W_2), &
\quad
& \wideparen \maclB _{s}(W_2), &
\quad
& \wideparen \maclB _{s}^\star (W_2), &
\quad
& \wideparen \maclC _{s}(W_2), &
\quad &\text{and} & \quad
&\wideparen \maclC _{s}^\star (W_2)
\intertext{to}
& \wideparen \maclA _{s}(W_1), &
\quad
& \wideparen \maclA _{s}^\star (W_1), &
\quad
& \wideparen \maclB _{s}(W_1), &
\quad
& \wideparen \maclB _{s}^\star (W_1), &
\quad
& \wideparen \maclC _{s}(W_1), &
\quad &\text{and} & \quad
&\wideparen \maclC _{s}^\star (W_1) \text .
\end{alignat*}

\par

If instead $s_1,s_2\in \mathbf R_{\flat ,\infty}$, then the same holds
true with $\wideparen \maclA _{0,s}$, $\wideparen \maclB _{0,s}$
and $\wideparen \maclC _{0,s}$ in place of
$\wideparen \maclA _{s}$, $\wideparen \maclB _{s}$
and $\wideparen \maclC _{s}$, respectivley, at each occurrence.
\end{prop}

\par

We observe that we may combine 
Proposition \ref{Prop:TwistTrace} and Theorem
\ref{Thm:LinjPullbacks} to get an alternative proof of
Theorem \ref{Thm:RingPropAnSpaces}.

\par

\appendix

\par

\section{Identifications of spaces of power series
expansions in terms of spaces of analytic functions}
\label{App:A}

\par

In this appendix we identify the spaces in \eqref{Eq:BparenSpaces}
and \eqref{Eq:CparenSpaces} by convenient subspaces of
$\wideparen A(W)$, for suitable $s\in \overline {\mathbf R_\flat}$. 

\par

We start with the following. We omit the proof since the result is an
immediate consequences of Theorems
4.1, 4.2, 5.2 and 5.3 in \cite{Toft18} and Definition \ref{Def:tauSpaces}.
Here let
\begin{align}
\kappa _{1,r,s}(z)
&=
\begin{cases}
e^{r(\log \eabs z)^{\frac 1{1-2s}}}, & s<\frac 12
\\[1ex]
e^{r|z|^{\frac {2\sigma}{\sigma +1}}}, & s=\flat _\sigma ,\ \sigma >0,
\\[1ex]
e^{\frac {|z|^2}2-r|z|^{\frac 1s}}, & \frac 12 \le s<\infty ,
\\[1ex]
e^{\frac {|z|^2}2}\eabs z^{-r}, & s=\infty ,
\end{cases}
\label{Eq:kappa1Def}
\\[1ex]
\kappa _{2,r,s}(z)
&=
\begin{cases}
e^{r|z|^{\frac {2\sigma}{\sigma -1}}}, & s=\flat _\sigma ,\ \sigma >1,
\\[1ex]
e^{\frac {|z|^2}2+r|z|^{\frac 1s}}, & \frac 12 \le s<\infty ,
\\[1ex]
e^{\frac {|z|^2}2}\eabs z^{r}, & s=\infty ,
\end{cases}
\label{Eq:kappa2Def}
\intertext{and}
\kappa _{j,r,s}^0(z)
&=
\begin{cases}
\kappa _{j,r,s}(z), & s\neq \frac 12,
\\[1ex]
e^{r|z|^2}, & s=\frac 12,
\end{cases}
\qquad j=1,2.
\label{Eq:kappa0Def}
\end{align}

\par

\begin{thm}\label{Thm:AnalSpacesChar}
Let $s_j,t_j\in \overline{\mathbf R_{\flat ,\infty}}$
be such that $t_j>\flat _1$, $j=1,2$, $W=\cc {d_2}\times \cc {d_1}$
and let $\kappa _{j,r,s}$ and $\kappa _{j,r,s}^0$ be given by
\eqref{Eq:kappa1Def}--\eqref{Eq:kappa0Def}, $j=1,2$, when $r>0$.
Then the following is true:
\begin{enumerate}
\item if $s_1,s_2<\infty$, then $\wideparen \maclA _{(s_2,s_1)} (W)$
consists of all $K\in \wideparen A(W)$ such that
$|K(z_2,z_1)|\lesssim \kappa _{1,r,s_2}(z_2)\kappa _{1,r,s_1}(z_1)$
for some $r>0$;

\vrum

\item if $s_1,s_2>0$, then $\wideparen \maclA _{0,(s_2,s_1)} (W)$ consists
of all $K\in \wideparen A(W)$ such that
$|K(z_2,z_1)|\lesssim \kappa _{1,r,s_2}^0(z_2)\kappa _{1,r,s_1}^0(z_1)$
for every $r>0$;

\vrum

\item if $t_1,t_2<\infty$, then $\wideparen \maclA _{(t_2,t_1)}' (W)$ consists
of all $K\in \wideparen A(W)$ such that
$|K(z_2,z_1)|\lesssim \kappa _{2,r,t_2}(z_2)\kappa _{2,r,t_1}(z_1)$
for every $r>0$;

\vrum

\item $\wideparen \maclA _{0,(t_2,t_1)}' (W)$ consists
of all $K\in \wideparen A(W)$ such that
$|K(z_2,z_1)|\lesssim \kappa _{2,r,t_2}^0(z_2)\kappa _{2,r,t_1}^0(z_1)$
for some $r>0$.
\end{enumerate}
\end{thm}

\par

By Remark \ref{Rem:SpaceSpecCase} it follows that Theorem
\ref{Thm:AnalSpacesChar} remains true after the spaces in
\eqref{Eq:SesAnalSp} are replaced by corresponding spaces
in \eqref{clASpaces}.

\par

By similar arguments as in the proofs of Theorems
4.1, 4.2, 5.2 and 5.3 in \cite{Toft18} we get the following two theorems.
The details are left for the reader.

\par

\begin{thm}\label{Thm:OtherAnalSpacesChar}
Let $s_1,s_2\in \overline{\mathbf R_{\flat ,\infty}}$, $W=\cc {d_2}\times \cc {d_1}$,
and let $\kappa _{j,r,s}$
and $\kappa _{j,r,s}^0$ be given by
\eqref{Eq:kappa1Def}--\eqref{Eq:kappa0Def}, $j=1,2$, when $r>0$.
Then the following is true:
\begin{enumerate}
\item if $s_1>\flat _1$ and $s_1,s_2<\infty$, then
$\wideparen \maclB _{(s_2,s_1)}(W)$
consists of all $K\in \wideparen A(W)$ such that
for every $r_1>0$ there is an $r_2>0$ such that 
$|K(z_2,z_1)|\lesssim \kappa _{1,r_2,s_2}(z_2)\kappa _{2,r_1,s_1}(z_1)$;

\vrum

\item if $s_1>\flat _1$ and $s_2>0$, then
$\wideparen \maclB _{0,(s_2,s_1)}(W)$
consists of all $K\in \wideparen A(W)$ such that
for every $r_2>0$ there is an $r_1>0$ such that 
$|K(z_2,z_1)|\lesssim \kappa _{1,r_2,s_2}^0(z_2)\kappa _{2,r_1,s_1}^0(z_1)$;

\vrum

\item if $s_2>\flat _1$ and $s_1,s_2<\infty$, then
$\wideparen \maclB _{(s_2,s_1)}^\star (W)$
consists of all $K\in \wideparen A(W)$ such that
for every $r_2>0$ there is an $r_1>0$ such that 
$|K(z_2,z_1)|\lesssim \kappa _{2,r_2,s_2}^0(z_2)\kappa _{1,r_1,s_1}^0(z_1)$;

\vrum

\item if $s_1>0$ and $s_2>\flat _1$, then
$\wideparen \maclB _{0,(s_2,s_1)}^\star (W)$
consists of all $K\in \wideparen A(W)$ such that
for every $r_1>0$ there is an $r_2>0$ such that 
$|K(z_2,z_1)|\lesssim \kappa _{2,r_2,s_2}^0(z_2)\kappa _{1,r_1,s_1}^0(z_1)$.
\end{enumerate}
\end{thm}

\par

\begin{thm}\label{Thm:OtherAnalSpacesChar2}
Let $s_1,s_2\in \overline{\mathbf R_{\flat ,\infty}}$, $W=\cc {d_2}\times \cc {d_1}$,
and let $\kappa _{j,r,s}$
and $\kappa _{j,r,s}^0$ be given by
\eqref{Eq:kappa1Def}--\eqref{Eq:kappa0Def}, $j=1,2$, when $r>0$.
Then the following is true:
\begin{enumerate}
\item if $s_1>\flat _1$ and $s_1,s_2<\infty$, then
$\wideparen \maclC _{(s_2,s_1)}(\cc {d_2}\times \cc {d_1})$
consists of all $K\in \wideparen A(\cc {d_2}\times \cc {d_1})$ such that
for some $r_2>0$ it holds
$|K(z_2,z_1)|\lesssim \kappa _{1,r_2,s_2}(z_2)\kappa _{2,r_1,s_1}(z_1)$
for every $r_1>0$;

\vrum

\item if $s_1>\flat _1$ and $s_2>0$, then
$\wideparen \maclC _{0,(s_2,s_1)}(\cc {d_2}\times \cc {d_1})$
consists of all $K\in \wideparen A(\cc {d_2}\times \cc {d_1})$ such that
for some $r_1>0$ it holds
$|K(z_2,z_1)|\lesssim \kappa _{1,r_2,s_2}^0(z_2)\kappa _{2,r_1,s_1}^0(z_1)$
for every $r_2>0$;

\vrum

\item if $s_2>\flat _1$ and $s_1,s_2<\infty$, then
$\wideparen \maclC _{(s_2,s_1)}^\star (\cc {d_2}\times \cc {d_1})$
consists of all $K\in \wideparen A(\cc {d_2}\times \cc {d_1})$ such that
for some $r_1>0$ it holds
$|K(z_2,z_1)|\lesssim \kappa _{2,r_2,s_2}^0(z_2)\kappa _{1,r_1,s_1}^0(z_1)$
for every $r_2>0$;

\vrum

\item if $s_1>0$ and $s_2>\flat _1$, then
$\wideparen \maclC _{0,(s_2,s_1)}^\star (\cc {d_2}\times \cc {d_1})$
consists of all $K\in \wideparen A(\cc {d_2}\times \cc {d_1})$ such that
for some $r_2>0$ it holds
$|K(z_2,z_1)|\lesssim \kappa _{2,r_2,s_2}^0(z_2)\kappa _{1,r_1,s_1}^0(z_1)$
for every $r_1>0$.
\end{enumerate}
\end{thm}

\par

\section{The link between Wick and anti-Wick symbols of rank one}
\label{App:B}

\par

In this appendix we show that if
$$
b_{z,w}(\alpha ,\beta ) =e_\alpha (z)e_\beta (\overline w),
$$
$t,t_0\in \mathbf C\setminus 0$ are such that $t_0^2=t$ and
\begin{align*}
b_{t_0z,\overline {t_0}w}^t(\alpha ,\beta )
&\equiv
\pi ^{-d} \int _{\cc d}b_{t_0w_1,\overline{t_0}w_1}(\alpha ,\beta )
e^{-(z-w_1,w-w_1)}
\, d\lambda (w_1)
\\[1ex]
&=
\pi ^{-d} \int _{\cc d}e_\alpha (t_0w_1)e_\beta (t_0\overline {w_1})
e^{-(z-w_1,w-w_1)}
\, d\lambda (w_1)
\end{align*}
(cf. \eqref{Eq:AntiWickAnalPseudoRelIntForm}), then
\begin{align}
b_{z,w}^t(\alpha ,\beta ) &= (\maclT _{0,t}b_{z,w})(\alpha ,\beta ) ,
\qquad \alpha ,\beta \in \nn d.
\label{Eq:WickAntiWickBasicTransf2}
\end{align}

\par

In fact, since all integrations can be done coordinate wise, we may assume
that $d=1$.

\par

First we prove \eqref{Eq:WickAntiWickBasicTransf2} for $t=1$ and consider
first the case when $\alpha =\beta =0$,
which is the same as
\begin{equation}\label{Eq:SimplestCase}
\pi ^{-1}\int _{\mathbf C} e^{(z-w_1)(\overline w-\overline w_1)}\, d\lambda (w_1)=1.
\end{equation}
In fact, if $z$ and $w$ are real, then the latter integral is equal to
\begin{equation*}
\pi ^{-1} e^{-zw} \iint _{\rr 2}e^{(z+w)x}e^{-i(z-w)y}e^{-x^2-y^2}\, dxdy
=
e^{-zw} e^{\frac 14(z+w)^2}e^{-\frac 14(z-w)^2} = 1.
\end{equation*}
The searched identity now follows by analytic continuation.

\par

For general $\alpha$ and $\beta$, \eqref{Eq:AntiWickAnalPseudoRelIntForm}
and integrations by parts give
\begin{multline*}
\pi (\alpha !\beta !)^{\frac 12}b_{z,w}^1(\alpha ,\beta )
=
\int _{\mathbf C} w_1^\alpha \overline w_1^\beta
e^{-(z-w_1)(\overline w -\overline w_1)}\, d\lambda (w_1)
\\[1ex]
=
(-1)^\beta e^{-z\overline w} 
\int _{\mathbf C} w_1^\alpha 
e^{z\overline w}e^{\overline ww_1}
\big ( \partial _{w_1}^\beta (e^{-|w_1|^2}) \big )\, d\lambda (w_1)
\\[1ex]
=
e^{z\overline w}
\sum _{\gamma \le \beta} {{\beta} \choose {\gamma}}
\int _{\mathbf C} \big ( \partial ^\gamma (w_1^\alpha )\big ) 
e^{z\overline w}
\big ( \partial _{\overline w_1}^{\beta -\gamma} ( e^{\overline ww_1} )\big )
e^{-|w_1|^2} \, d\lambda (w_1)
\\[1ex]
=
e^{z\overline w}
\sum _{\gamma \le \alpha ,\beta}
{{\alpha} \choose {\gamma}} {{\beta} \choose {\gamma}} \gamma !
\overline w^{\beta -\gamma}
\int _{\mathbf C} w_1^{\alpha -\gamma} 
e^{z\overline w} e^{\overline ww_1}
e^{-|w_1|^2} \, d\lambda (w_1)
\\[1ex]
=
e^{z\overline w}
\sum _{\gamma \le \alpha ,\beta}
{{\alpha} \choose {\gamma}} {{\beta} \choose {\gamma}} \gamma !
\overline w^{\beta -\gamma}
(-1)^{\alpha -\gamma}\int _{\mathbf C}
e^{z\overline w} e^{\overline ww_1}
\big ( \partial _{\overline w_1}^{\alpha -\gamma}(e^{-|w_1|^2})\big )
\, d\lambda (w_1)
\\[1ex]
=
\sum _{\gamma \le \alpha ,\beta}
{{\alpha} \choose {\gamma}} {{\beta} \choose {\gamma}} \gamma !
z^{\alpha -\gamma}\overline w^{\beta -\gamma}
e^{z\overline w} \int _{\mathbf C}
e^{z\overline w} e^{\overline ww_1}
e^{-|w_1|^2} \, d\lambda (w_1)
\\[1ex]
=
\pi (\alpha !\beta !)^{\frac 12}
\sum _{\gamma \le \alpha ,\beta}
{{\alpha} \choose {\gamma}}^{\frac 12} {{\beta} \choose {\gamma}}^{\frac 12}
e_{\alpha -\gamma} (z)e_{\beta -\gamma}(\overline w),
\end{multline*}
and \eqref{Eq:WickAntiWickBasicTransf2} follows for general $\alpha$ and
$\beta$. Here the last equality follows from \eqref{Eq:SimplestCase}
and the identity
$$
\alpha !^{-\frac 12}\gamma !^{\frac 12}z^{\alpha -\gamma}
= 
{{\alpha} \choose {\gamma}}^{-\frac 12} e_{\alpha -\gamma} (z),
$$
and \eqref{Eq:WickAntiWickBasicTransf2} follows in the case when
$t=1$.

\par

Next suppose that $t\in \mathbf C\setminus 0$ and recall that $t_0^2=t$.
By the definitions it follows that
$$
(\maclT _{0,t}b_{z,w})(\alpha ,\beta ) = t_0^{|\alpha +\beta |} (\maclT _{0,1}
b_{z/t_0,w/\overline {t_0}})(\alpha ,\beta ),
$$
and the case $t=1$ shows that
\begin{align*}
(\maclT _{0,t}b_{t_0z,\overline {t_0}w})(\alpha ,\beta )
&=
\pi ^{-d} t_0^{|\alpha +\beta |} \int _{\cc d} b_{w_1,w_1}(\alpha ,\beta )
e^{-(z-w_1,w-w_1)}\, d\lambda (w_1)
\\[1ex]
&=\pi ^{-d} \int _{\cc d} b_{t_0w_1,\overline{t_0}w_1}(\alpha ,\beta )
e^{-(z-w_1,w-w_1)}\, d\lambda (w_1),
\end{align*}
which gives the assertion.

\par

\section{Identities of spaces of kernel operators, Wick
operators and anti-Wick operators}
\label{App:C}

\par

Let $s_1,s_2 \in \overline{\mathbf R_\flat}$ be such that
$s_1 < \frac 12$ and $s_2 \leq \frac 12$. By Corollary
\ref{Cor:KernelLinOpChar}, Proposition
\ref{Prop:KernelOpWickEqual} and Theorems
\ref{Thm:antiWickWickEqual}--\ref{Thm:OpReprIdentC},
it follows that
\begin{alignat}{2}
\op _{\mathfrak V}(\wideparen \maclA _{s_1}'(\cc {2d})) &=&
\sets {T_K}{K\in \wideparen \maclA _{s_1}'(\cc {2d})}
&=
\maclL (\maclA _{s_1}(\cc d),\maclA _{s_1}'(\cc d)),
\\[1ex]
\op _{\mathfrak V}(\wideparen \maclA _{0,s_2}'(\cc {2d})) &= \ &
\sets {T_K}{K\in \wideparen \maclA _{0,s_2}'(\cc {2d})}
&=
\maclL (\maclA _{0,s_2}(\cc d),\maclA _{0,s_2}'(\cc d)),
\end{alignat}

\vspace{-0.33cm}

\begin{equation}
\op _{\mathfrak V}^{\aw}(\wideparen \maclA _{s_1}(\cc {2d}))
=
\op _{\mathfrak V}(\wideparen \maclA _{s_1}(\cc {2d})),
\quad
\op _{\mathfrak V}^{\aw}(\wideparen \maclA _{0,s_2}(\cc {2d}))
=
\op _{\mathfrak V}(\wideparen \maclA _{0,s_2}(\cc {2d})),
\end{equation}

\vspace{-0.35cm}

%
\begin{alignat}{4}
\op _{\mathfrak V}^{\aw}(\wideparen \maclC _{s_1}(\cc {2d}))
&= &
&\op _{\mathfrak V}(\wideparen \maclC _{s_1}(\cc {2d})) &
&\subseteq 
\op _{\mathfrak V}(\wideparen \maclB _{s_1}(\cc {2d})) &&
\notag
\\
&&&&
& = 
\sets {T_K}{K\in \wideparen \maclB _{s_1}(\cc {2d})} &
&=
\maclL (\maclA _{s_1}(\cc d)),
\\[1ex]
\op _{\mathfrak V}^{\aw}(\wideparen \maclC _{s_1}^\star (\cc {2d}))
&= &
&\op _{\mathfrak V}(\wideparen \maclC _{s_1}^\star (\cc {2d})) &
&\subseteq 
\op _{\mathfrak V}(\wideparen \maclB _{s_1}^\star (\cc {2d})) &&
\notag
\\
&&&&
&= 
\sets {T_K}{K\in \wideparen \maclB _{s_1}^\star (\cc {2d})} &
&=
\maclL (\maclA _{s_1}'(\cc d)),
\\[1ex]
\op _{\mathfrak V}^{\aw}(\wideparen \maclC _{0,s_2}(\cc {2d}))
&= \ &
&\op _{\mathfrak V}(\wideparen \maclC _{0,s_2}(\cc {2d})) &
&\subseteq 
\op _{\mathfrak V}(\wideparen \maclB _{0,s_2}(\cc {2d})) &&
\notag
\\
&&&&
&= 
\sets {T_K}{K\in \wideparen \maclB _{0,s_2}(\cc {2d})} &
&=
\maclL (\maclA _{0,s_2}(\cc d)),
\intertext{and}
\op _{\mathfrak V}^{\aw}(\wideparen \maclC _{0,s_2}^\star (\cc {2d}))
&= &
&\op _{\mathfrak V}(\wideparen \maclC _{0,s_2}^\star (\cc {2d})) &
&\subseteq 
\op _{\mathfrak V}(\wideparen \maclB _{0,s_2}^\star (\cc {2d})) &&
\notag
\\
&&&&
&= 
\sets {T_K}{K\in \wideparen \maclB _{0,s_2}^\star (\cc {2d})} &
&=
\maclL (\maclA _{0,s_2}'(\cc d)).
\end{alignat}
%

\par

\end{document}